\documentclass[10pt]{book}

\usepackage[ngerman,french,english]{babel}

\usepackage{graphicx}
\usepackage{tikz}
\usepackage{atbegshi}

\AtBeginShipoutFirst{%
    \begin{tikzpicture}[remember picture, overlay]
        \node[anchor=north west, xshift=1in, yshift=-1.5cm] at (current page.north west) {%
            \includegraphics[width=4cm]{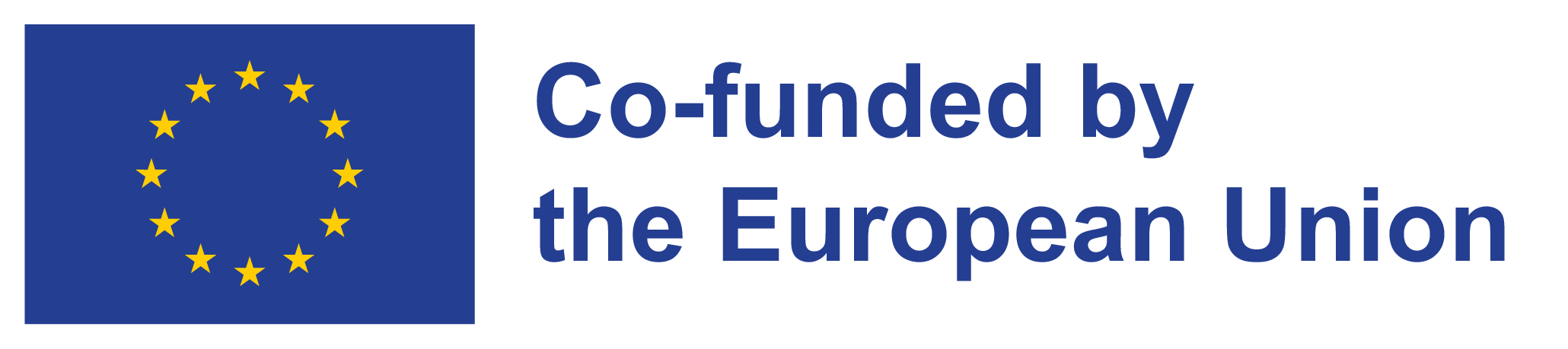} 
        };
    \end{tikzpicture}%
}

\usepackage{pdfsync}
\usepackage{tikz}
\usepackage{framed}
\usepackage{amsmath}                                                        
\numberwithin{equation}{section}
\usepackage{graphicx}                                                       
\usepackage{subfigure}
\usepackage{enumitem}                                                    
\usepackage{hyperref}
\usepackage{amssymb}                                                        
\usepackage[mathscr]{eucal}                                             
\usepackage{cancel}                                                             
\usepackage[normalem]{ulem}                                                                 
\usepackage{pstricks}
\usepackage{rotating}
\usepackage{lscape}
\usepackage[paperwidth=8.5in,paperheight=11in,top=1.00in, bottom=1.00in, left=1.00in, right=1.00in]{geometry}
\usepackage{mathtools}                                                      
\mathtoolsset{showonlyrefs=true}                                    
\usepackage{fixltx2e,amsmath}                                           
\MakeRobust{\eqref}

\usepackage{dsfont}

\usepackage{mathdots}
\usepackage{amsthm}                                                             
\allowdisplaybreaks

\theoremstyle{plain}
\newtheorem{theorem}{Theorem}
\numberwithin{theorem}{section}

\newtheorem{lemma}[theorem]{Lemma}                              
\newtheorem{proposition}[theorem]{Proposition}

\newtheorem{corollary}[theorem]{Corollary}
\theoremstyle{definition}
\newtheorem{definition}[theorem]{Definition}
\newtheorem{example}[theorem]{Example}
\newtheorem{notation}[theorem]{Notation}
\newtheorem{remark}[theorem]{Remark}

\newtheorem{assumption}[theorem]{Assumption}

\counterwithin{table}{chapter}

\def \R {\mathbb{R}}

\newcommand\Eb{\mathbb{E}}

\newcommand\Pb{\mathbb{P}}

\newcommand\Rd{\mathbb{R}^d}
\newcommand\Rdd{\mathbb{R}\times \mathbb{R}^d}

\newcommand\Pcc{\mathscr{P}}

\newcommand\Kcc{\mathscr{K}}

\newcommand\Bc{\mathcal{B}}
\newcommand\Cc{\mathcal{C}}
\newcommand\Fc{\mathcal{F}}

\newcommand\Pc{\mathcal{P}}

\newcommand\eps{\varepsilon}

\newcommand\dd{d}

\def \R  {{\mathbb {R}}}

\def \eps {{\varepsilon}}

\def \It\^o {It\^o }

\def \R {{\mathbb {R}}}
\def \N {{\mathbb {N}}}

\def \eps {{\varepsilon}}

\def \tilde {\widetilde}

\def \Ã  {{\`a }}
\def \Ã¨ {{\`e }}
\def \Ã² {{\`o }}
\def \Ã¹ {{\`u }}

\def \rr {\mathfrak{r}}

\begin{document}

\title{\bf McKean-Vlasov Differential Equations:\\ 
{An introduction and a focus on some kinetic models}}

\author{
Stefano Pagliarani*
\thanks{Dipartimento di Matematica, Universit\`a di Bologna, Bologna, Italy. 
\textbf{e-mail}: stefano.pagliarani9@unibo.it.}
\\ 
\vspace{0pt}\\
 \textit{*These notes are dedicated to my father.}
 \\ \\}

\date{This version: \today}

\maketitle




%
%


\chapter*{Introduction}

These notes were prepared for a series of lectures delivered at the 43rd Finnish Summer School on Probability and Statistics, held in Lammi, Finland, from May 26 to 30, 2025. Their purpose is to give a partial account of the theory of McKean-Vlasov differential equations, focusing: first, on its probabilistic foundations within the theory of diffusion processes and its connection with non-linear parabolic partial differential equations (PDEs); second, on a subclass of kinetic-type models characterized by degenerate noise and irregular coefficients. McKean-Vlasov stochastic differential equations (MKV SDEs), also known as \emph{distribution-dependent} SDEs, are formal expressions of the form
\begin{equation}\label{eq:mkv_SDE_intro}
dX_t = {\bf b}(t,X_t,{\boldsymbol\mu}_t)\, dt + {\boldsymbol\sigma}(t,X_t, {\boldsymbol\mu}_t) \, d W_t , \qquad {\boldsymbol\mu}_t = [X_t],
\end{equation}
which differ from ``standard" It\^o SDEs in that the coefficients of the equation explicitly depend on the law of the solution. 

Historically, MKV SDEs were introduced by McKean in his seminal paper \emph{A class of Markov processes associated with nonlinear parabolic equations} (\cite{mckean1966class}). The title of this paper is highly suggestive of the purpose of the author, which was to provide a probabilistic representation of the solutions to non-linear PDEs arising from physics, such as Boltzmann and Vlasov equations, mimicking the well-known connection between linear parabolic PDEs and diffusion processes. The easiest way to see this connection is by means of the It\^o formula. Indeed, under very mild assumptions, the latter promptly yields
\begin{equation}
\int \varphi(y) {\boldsymbol\mu}_t (dy) =  \int \varphi(y) {\boldsymbol\mu}_0 (dy) + \int_0^t \int \Big( {\bf b}(t,y,{\boldsymbol\mu}_r) \varphi'(y) + \frac{1}{2} {\boldsymbol\sigma}(t,y, {\boldsymbol\mu}_r)\varphi''(y) \Big) {\boldsymbol\mu}_r(dy) \, dr,
\end{equation}
for any test function $\varphi$, which amounts to saying that the flow of laws $(t\mapsto [X_t])$ solves, in the distributional sense, the non-linear Fokker-Planck (FP) equation
\begin{equation}
\partial_t {\boldsymbol\mu}_t = - \partial_y \big( {\bf b}(t,y,\textcolor{black}{{\boldsymbol\mu}_t}) \,  {\boldsymbol\mu}_t \big)  +\frac{1}{2} \partial_y^2 \big(    {\boldsymbol\sigma}{\boldsymbol\sigma}^\top(t,y,\textcolor{black}{{\boldsymbol\mu}_t}) \, {\boldsymbol\mu}_t \big).
\end{equation}
The reverse implication is also true and can be proved by means of so-called superposition principles. Namely, it can be shown that if the FP equation has a solution ${\boldsymbol\nu}$, then the MKV SDE has a solution $X$ with ${\boldsymbol\nu}_t = [X_t]$. 
This important connection with non-linear PDEs earned the solutions to the MKV SDEs the name \emph{non-linear diffusions}. 

In his subsequent paper, \emph{Propagation of chaos for a class of non-linear parabolic equations} (\cite{mckean1967propagation}), building upon the pioneering works by Kac on kinetic theory (see \cite{kac1956foundations}), McKean revealed the connection of this ``new" family of diffusions with so-called mean-field interacting particle systems of the type 
\begin{equation}
dX^{i}_t = {\bf b}\big(t,X^{i}_t,\mu^{\text{emp}}_{ N,t}\big) dt + {\boldsymbol\sigma}\big(t,X^{i}_t,\mu^{\text{emp}}_{ N,t}\big) dW^i_t ,\qquad i\in \N, 
\end{equation}
where $\mu^{\text{emp}}_{ N,t}$ is the empirical law of the first $N$ particles, i.e.
\begin{equation}
\mu^{\text{emp}}_{N,t}:= \frac{1}{N} \sum_{j=1}^N \delta_{X^{j}_t}.
\end{equation}
The link of the MKV SDE with the particle system is technically more difficult to establish than the one with the nonlinear PDE. Heuristically, it can be seen as follows. For $M\in \N$, 
\begin{equation}
\mu^{\text{emp}}_{N,t} \approx \frac{1}{N} \sum_{j=M+1}^N \delta_{X^{j}_t}, \qquad 
N>>1,
\end{equation}
and thus the interaction among the first $M$ particles vanishes as $N$ becomes large. Therefore, if the initial data and the driving noises are independent, the particles $X^1 , \dots , X^M$ are asymptotically independent as $N\to \infty$. Furthermore, the system is exchangeable and thus all particles have the same law. Therefore, Glivenko-Cantelli's theorem yields 
\begin{equation}
[X^i_t] \longrightarrow {\boldsymbol\mu}_t, \qquad \text{as}\quad N\to \infty,
\end{equation}  
where ${\boldsymbol\mu}_t$ is the law of the solution to the MKV SDE above. 

To sum up, MKV SDEs provide a link between mean-field particle systems, which represent the microscopic description of a system, and non-linear partial differential equations, which describe macroscopic evolutions:

\begin{center}
\begin{tikzpicture}[scale=0.3,thick, every node/.style={draw=blue, draw, rounded corners, minimum width=3.5cm, minimum height=0.7cm, align=center, font=\large}]
  \node (A) at (0,4) {MKV SDE};
  \node (B) at (-10,-1.6) {Mean-field system};
  \node (C) at (10,-1.6) {Non-linear PDE};
  \draw[<->] (A) -- (B);
  \draw[<->] (C) -- (A);
\end{tikzpicture}
\end{center} 
This connection has provided the foundation for numerous mathematical theories, computational tools, and applications developed over the last decades. On the mathematical side, this framework has led to a rich theory of propagation of chaos and mean-field limits, and has played a central role in the development of mean-field control and mean-field games (\cite{carmona2018probabilistic}). Its range of applications is remarkably broad: in the natural and physical sciences, mean-field models arise in kinetic theory, statistical mechanics, collective dynamics, and models of interacting populations; in neuroscience, they describe the collective activity of large networks of interacting neurons (see \cite{galves2024probabilistic}); and in the social sciences, they are used to model phenomena such as opinion formation, crowd and population dynamics, economic interactions, and systemic effects in financial markets. 

\vspace{2pt} 

Despite the rapidly growing number of research papers connected to McKean-Vlasov diffusions, there are relatively few textbooks and notes that are entirely, or even partially, dedicated to this topic. Among the existing contributions, we recall: the book by Kolokoltsov \cite[Ch. 6.8]{kolokoltsov2019differential}, where MKV SDEs are introduced from a PDE point of view; the book by Gobet \cite[Ch. 9]{gobet2016monte}, where a concise all-around presentation covers, in the Lipschitz framework, well-posedness, propagation of chaos and numerical approximations (see also \cite{antonelli2002rate}, \cite{dos2022simulation} among others); the book by Carmona and Delarue \cite{carmona2018probabilistic}, in which MKV SDEs are presented as fundamental building blocks of stochastic mean-field games; the lecture notes by Carmona \cite{carmona2016lectures}, again in the Lipschitz framework and with a focus on financial applications. We also recall the classical lecture notes by Sznitman \cite{Sznitman91}, on propagation of chaos; and the extensive survey by Chaintron \cite{chaintron2022propagation,chaintron2022propagationb}.

The purpose of these notes is twofold: 
\begin{itemize}
\item[(i)] Provide a general probabilistic framework for McKean-Vlasov equations, with the language of stochastic differential calculus and partial differential equations. The plan is to extend the ``linear" theory of weak and pathwise solutions and their link with parabolic PDEs, with definitions and results that are entirely model-free and require no assumptions on the coefficients \footnote{Only some basic measurability and local integrability assumptions will be made.}. In particular, the class of density-dependent MKV SDEs, connected to well-known non-linear PDEs such as Burgers-type equations, is regarded as a particular instance of the general framework. Though these notions are well-known to specialists, we believe it might be useful for students, as well as for researchers who approach this topic for the first time, to have these concepts summarized in one single text, with coherent and uniform notation.
\item[(ii)] Provide a focus on a class of density-dependent McKean-Vlasov SDEs. 
As the coefficients depend on the distribution through the pointwise evaluation of the density at the current state, these models exhibit low regularity with respect to the measure variable. Therefore their analysis must rely on the regularization properties of the underlying PDEs. The plan is to provide a fairly general well-posedness and regularity result, with a full proof, in a kinetic setting where the diffusion is degenerate but the noise is restored by means of a suitable H\"ormander-type condition. The latter induces an intrinsic, non-Euclidean, geometry in which the regularity assumptions and results are naturally stated. Though these results can by no means be considered a break-through in the field, we believe they slightly sharpen and generalize some existing ones. More importantly, we hope the way the results are presented can shed some light on the interplay between well-known PDE techniques and superposition results.
\end{itemize}

In these notes we will always consider diffusive equations, namely integral equations of It\^o type, although different types of noise and stochastic integrals could be considered as well, for instance including processes driven by Poisson jump measures. Furthermore, more general models may include infinite-dimensional noise, or dependence not only on current time-marginals but on the whole law of the process. The document is divided into two chapters, basically in correspondence with the two goals above, and one appendix. 

\emph{Chapter 1.} The class of McKean-Vlasov diffusions is introduced in Section \ref{sec:mkv_diffusions}. In particular, the concepts of general, weak and pathwise solutions are introduced, in terms of the so-called \emph{frozen equation}. The latter allows one to connect the class of MKV diffusions to the standard theory of diffusion processes, in particular to the standard theory of It\^o SDEs and to the Stroock-Varadhan theory of martingale problems. In particular, the application of the Watanabe-Yamada Theorem identifies a pathway to pathwise well-posedness, which consists in establishing first weak well-posedness for the MKV SDE and then pathwise uniqueness for the frozen SDE. In Section \ref{sec:nonlin_fp} we establish the two-way link between MKV SDEs and non-linear Fokker-Planck equations. After giving the definition of weak (or distributional) solution to the latter, we state and prove the double implication that is well known in the linear case, which roughly reads as follows: \emph{a flow of non-negative finite measures is a solution to the FP equation if and only if it is the flow of time-marginals corresponding to a solution to the associated stochastic differential equation}. Once more, the frozen equation allows one to make use of standard It\^o theory and Ambrosio-Figalli-type superposition principles. In Section \ref{sec:part_sys} we provide a brief glance at the link between MKV diffusions and mean-field particle systems, referring to the classic reference \cite{Sznitman91} and to the survey \cite{chaintron2022propagation,chaintron2022propagationb} for detailed presentations on propagation of chaos results. Section \ref{sec:den_dep_mkv} is devoted to casting the class of density-dependent MKV SDEs into the general framework. While, conceptually, the latter can be promptly seen as a particular subclass, where the dependence on the law takes place through the pointwise evaluation of the density, one has to make sure that all the definitions are independent of the specific version of the density that one selects. While checking this is a simple exercise, it is necessary in order to apply all the results in Sections \ref{sec:mkv_diffusions} and \ref{sec:nonlin_fp} to this subclass.
 
 \emph{Chapter 2.} In Section \ref{sec:kin_mkv} we consider a class of degenerate MKV SDEs, where the noise is present only in some directions, and state a H\"ormander-type assumption on the drift which allows one to restore the regularity of the underlying semigroups. We refer to these equations as \emph{kinetic-type} models, the classical Langevin kinetic model a particular case. In Section \ref{sec:kinet_semigroups}, we introduce the kinetic-type semigroups that play the same role as the heat semigroup in the theory of uniformly parabolic equations. In Section \ref{sec:ani_int} we study the non-Euclidean geometry induced by the H\"ormander condition and the H\"older spaces that are naturally associated to this geometry. These spaces can be introduced in different ways. Here we present them as the natural spaces in which Schauder-type estimates for the heat semigroup can be extended to the kinetic-type semigroups. In particular, in Section \ref{sec:anisot_spaces} we introduce the so-called \emph{anisotropic spaces}, which are relevant as far as space regularity is concerned, while in Section \ref{sec:intrinsic_spaces} we extend them to the so-called \emph{intrinsic spaces}, which arise when regularity in both space and time is concerned. In both cases, rigorous definitions and results are preceded by heuristic derivations in the prototypical kinetic setting. In Section \ref{sec:kin_den_result} we study a kinetic-type density-dependent MKV SDE, with measurable drift and H\"older-continuous diffusion. Under two boundedness assumptions on the drift, a global and a local one, we prove that there exists a unique (in law) weak solution and, in the case of constant diffusion and  $\gamma-$H\"older continuous initial datum, that the corresponding flow of time-marginal densities is intrinsically (almost) $\gamma-$H\"older continuous. Owing to the non-linear superposition theorem in Section \ref{sec:den_dep_mkv}, the proof mainly relies on the study of weak solutions to the associated Fokker-Planck PDE. Existence and uniqueness of the latter, together with the relevant regularity estimates, are proved in Section \ref{sec:proof_prop_fp} by means of tools known as \emph{mild solutions}. In the case of non-degenerate noise, we also prove existence and (almost sure) uniqueness of pathwise solutions. 
 
   
The appendix contains a review of the basic elements of $d$-dimensional diffusion theory, including It\^o stochastic differential equations, martingale problems and linear Fokker-Planck PDEs.

\vspace{2pt} Originally, these notes were supposed to contain one more chapter, placed between Chapter 1 and Chapter 2, devoted to the Lipschitz theory for MKV SDEs. Unfortunately, due to time constraints, the writing of this chapter was not possible. While we plan to fill this gap in future versions, we refer here to the classic references \cite{Sznitman91}, \cite{carmona2016lectures}, \cite{carmona2018probabilistic}, \cite{gobet2016monte}, among others.

\vspace{10pt} 

{\noindent\textbf{Acknowledgements}: The author has received funding from the European Union’s Horizon Europe research and innovation programme under the Marie Sklodowska-Curie Actions Staff Exchanges (Grant Agreement No. 101183168, Call: HORIZON-MSCA-2023-SE-01).}

\vspace{10pt}

{\noindent\textbf{Declaration on the use of AI:}  No AI tool was used significantly, at any stage, for the writing of these notes.

\newpage
\section*{General notations}
In these notes the following notations will be employed:
\begin{itemize}
\item $\mathcal{M}^{d\times q}$: space of ($d\times q$)-matrices with real entries
        \item $\mathcal{B}_d$ (or simply $\mathcal{B}$): Borel $\sigma$-algebra of the Euclidean space $\R^d$. 
          \item $\Bc(S)$: Borel $\sigma$-algebra on $S\subset \R^d$
	\item $m\mathcal{F}$: measurable functions from a measurable space $(\Omega, \mathcal{F})$ to a Euclidean space $(\R^d , \mathcal{B})$
\item $m\mathcal{F}^+$: non-negative real-valued functions in $m\mathcal{F}$
	\item $bm\mathcal{F}$: bounded functions in $m\mathcal{F}$
	\item $m\mathcal{B}$: Borel measurable-functions between two Euclidean spaces $\R^d$ and $\R^q$
	\item $m\mathcal{B}^+$: non-negative real-valued functions in $m\mathcal{B}$
	\item $bm\mathcal{B}$: bounded functions in $m\mathcal{B}$
\item \textcolor{black}{$[X]$}: law of a random variable/stochastic process $X$
\item \textcolor{black}{$\mathcal{P}(\Rd)$}: probability distributions on $(\Rd,\Bc)$, equipped with the weak topology
\item \textcolor{black}{$\mathcal{P}^p(\Rd)$}: elements of $\mathcal{P}(\Rd)$ with finite $p$-th moment
\item \textcolor{black}{$\boldsymbol\nu$, $\boldsymbol\mu$}: measurable functions from $[0,\infty)$ to $\mathcal{P}(\Rd)$
\item \textcolor{black}{$\mathcal{M}(\Rd)$ and $\mathcal{M}_{+}(\Rd)$:} respectively, signed and non-negative finite measures on $(\Rd,\Bc)$
\item \textcolor{black}{$(\Fc_t)_{t}$}: complete and right-continuous filtration on a probability space $(\Omega,\Fc,\Pb)$
\item \textcolor{black}{$W=(W_t)_t$}: a (multi-dimensional) Brownian motion, either defined on a filtered probability space $(\Omega,\Fc,\Pb, (\Fc_t)_{t})$, or utilized as a dummy variable in a stochastic differential equation
\item \textcolor{black}{$(C_{[t,T)}, \Bc )$}: canonical space of continuous functions from $[t,T)$ to $\Rd$, with $0\leq t < T\leq \infty$, equipped with the Borel $\sigma$-algebra of the topology induced by uniform convergence on compacts. In particular, $(C_{T}, \Bc ) := (C_{[0,T)}, \Bc )$
\item \textcolor{black}{$\Cc_b = \Cc_b(\Rd)$}: space of bounded continuous functions from $\Rd$ to $\R$, equipped with the topology of uniform convergence
\item \textcolor{black}{$C_T \Cc_b$}: space of continuous functions from $[0,T]$ to $\Cc_b$, with $T\in [0,\infty)$
\item \textcolor{black}{$C_0^{\infty}(\Rd)$}: space of smooth functions from $\Rd$ to $\R$, with compact support.
\end{itemize}

\tableofcontents

\chapter{Introduction to McKean-Vlasov equations}\label{ch:1}

We introduce the class of McKean-Vlasov stochastic differential equations and analyze their connection with non-linear Fokker-Planck equations and mean-field particle systems. 

In order to improve readability, all the functions and stochastic processes of this Chapter are defined on the time interval $[0,T)$ with $T=\infty$. However, all the definitions and results naturally extend to the case $T\in (0,\infty)$.

\section{McKean-Vlasov diffusions}\label{sec:mkv_diffusions}


A \emph{McKean-Vlasov stochastic differential equation (MKV-SDE)} is 
a formal expression of the form
\begin{equation}\label{eq:mkv_SDE}
\hspace{-70pt}{\blue\text{[MKV$({\bf b},{\boldsymbol\sigma})$]}}\qquad\qquad
dX_t = {\bf b}(t,X_t,[X_t]) dt + {\boldsymbol\sigma}(t,X_t, [X_t]) dW_t, 
\end{equation}
where the functions 
\begin{equation}\label{eq:coefficients_mkv}
{\bf b}: [0,\infty)\times \Rd \times \mathcal{P}(\Rd)\to \R^d, \qquad {\boldsymbol\sigma}: [0,\infty)\times \Rd \times \mathcal{P}(\Rd)\to \mathcal{M}^{d\times q}
\end{equation}
are called the coefficients of the equation. Here, $\mathcal{P}(\Rd)$ denotes the set of probability measures on $(\R^d,\Bc)$, equipped with the topology of weak convergence. Hereafter, we will assume that the coefficients ${\bf b}, {\boldsymbol\sigma}$ are Borel measurable functions. As with standard SDEs, the MKV-SDE  \eqref{eq:mkv_SDE} can be formulated in the weak sense, where the ($q$-dimensional) Brownian motion is part of the solution and an initial distribution is assigned, or in the pathwise sense, where the Brownian motion is fixed and an initial random variable is assigned.
\begin{framed}
\noindent {\bf(!)} The key distinguishing feature here is that the coefficients ${\bf b}$, ${\boldsymbol\sigma}$ exhibit explicit dependence on  the time-marginal laws of the solution, $[X_t]$.
\end{framed}

MKV-SDEs may be alternatively referred to as mean-field equations and their solutions as non-linear diffusions, the reasons for this terminology being clarified below, in Sections \ref{sec:nonlin_fp} and \ref{sec:part_sys}.

\begin{example}[linear and convolutional MKV-SDEs]\label{ex:linear} Set
\begin{equation}
{\bf b}(t,x,\nu) = \int b(t,x,y) \nu(\dd y), \qquad  {\boldsymbol\sigma}(t,x,\nu) = \int \sigma(t,x,y) \nu(dy),
\end{equation}
with $b: [0,\infty)\times \Rd \times \Rd\to \R^d$ and $\sigma: [0,\infty)\times \Rd \times \Rd\to \mathcal{M}^{d\times q}$ being measurable functions.
In this case, MKV$({\bf b},{\boldsymbol\sigma})$ can then also be written as
\begin{equation}\label{eq:mkv_SDE_linear}
dX_t = \Eb\big[b(t,x,X_t)\big]\big|_{x=X_t} dt + \Eb\big[\sigma(t,x,X_t)\big]\big|_{x=X_t} dW_t.
\end{equation}
As the coefficients ${\bf b}, {\boldsymbol\sigma}$ are linear with respect to the measure variable, we refer to this class of equations as \emph{linear MKV-SDEs}. A particular case of a linear MKV-SDE is given by setting
\begin{equation}
b(t,x,y) = B(t,x-y), \qquad \sigma(t,x,y) = \Sigma(t,x-y),
\end{equation}
with $B: [0,\infty)\times \Rd \to \R^d$ and $\Sigma: [0,\infty)\times \Rd\to \mathcal{M}^{d\times q}$ being measurable functions,
which indeed yields 
\begin{equation}
{\bf b}(t,x,\nu) = [ B(t, \cdot)\ast \nu ] (x), \qquad  {\boldsymbol\sigma}(t,x,\nu) =[ \Sigma(t, \cdot)\ast \nu ] (x).
\end{equation}
In this particular case, MKV$({\bf b},{\boldsymbol\sigma})$ can be also written as
\begin{equation}
dX_t = \big[ B(t, \cdot)\ast [X_t] \big] (X_t) dt + \big[ \Sigma(t, \cdot)\ast [X_t] \big] (X_t) dW_t.
\end{equation}
We refer to this class of equations as \emph{convolutional MKV-SDEs}. 
\end{example}


A crucial tool to define and study the solutions to MKV-SDEs is given by the so-called \emph{frozen MKV-SDE}, which is the standard It\^o SDE
\begin{equation}
\hspace{-110pt}{\blue\text{[SDE$({\bf b}_{\boldsymbol\nu},{\boldsymbol\sigma}_{\boldsymbol\nu})$]}}\qquad\qquad dZ_t = {\bf b}_{\textcolor{black}{\boldsymbol\nu}}(t,Z_t) dt + {\boldsymbol\sigma}_{\textcolor{black}{\boldsymbol \nu}}(t,Z_t) dW_t ,
\end{equation}
where ${\boldsymbol\nu}:[0,\infty) \to \Pc(\Rd)$ is a \underline{fixed} measurable flow of distributions, and where ${\bf b}_{\textcolor{black}{\boldsymbol\nu}}: [0,\infty)\times \Rd \to \R^d$, $ {\boldsymbol\sigma}_{\textcolor{black}{\boldsymbol \nu}}: [0,\infty)\times \Rd\to \mathcal{M}^{d\times q}$ are the measurable functions defined by
\begin{equation}\label{eq:def_frozen_coeff}
{\bf b}_{\textcolor{black}{\boldsymbol\nu}}(t,x) :=   {\bf b}(t, x , {\boldsymbol\nu}_t ), \qquad {\boldsymbol\sigma}_{\textcolor{black}{\boldsymbol\nu}}(t,x) :=   {\boldsymbol\sigma}(t, x , {\boldsymbol\nu}_t ), \qquad (t,x)\in[0,\infty)\times \Rd.
\end{equation}
\begin{remark}\label{rem:frozen_coeff}
If ${\boldsymbol\nu}:[0,\infty) \to \Pc(\Rd)$ is a measurable flow of distributions, then the coefficients ${\bf b}$ and ${\boldsymbol\sigma}$ are jointly (in time-space) Borel measurable on $[0,\infty)\times\R^d$.
\end{remark}
As we mentioned above, sometimes MKV-SDEs are referred to as nonlinear SDEs. For this reason, some authors refer to the frozen equation SDE$({\bf b}_{\boldsymbol\nu},{\boldsymbol\sigma}_{\boldsymbol\nu})$ as the \emph{linearized equation}. In order to avoid confusion with the class of linear MKV-SDEs introduced in Example \ref{ex:linear}, we prefer to avoid this term and only refer to SDE$({\bf b}_{\boldsymbol\nu},{\boldsymbol\sigma}_{\boldsymbol\nu})$ as the frozen equation, to highlight the fact that the measure argument in $\bf b$ and $\boldsymbol\sigma$ is a priori fixed.

We can now give the definition of general solution to a MKV-SDE.

\begin{framed}
\vspace{-10pt}
\begin{definition}[General solution to an MKV-SDE]\label{def:sol_gen_MKV}
A (general) solution to MKV$({\bf b},{\boldsymbol\sigma})$ is a triple $( W, (\Fc_t)_t, X)$ which is a (general) solution to SDE$({\bf b}_{\boldsymbol\mu},{\boldsymbol\sigma}_{\boldsymbol\mu};0)$ (see Definition \ref{def:sol_gen_SDE}), where ${\bf b}_{\boldsymbol\mu},{\boldsymbol\sigma}_{\boldsymbol\mu}$ are as in \eqref{eq:def_frozen_coeff} with
\begin{equation}
{\boldsymbol\mu}_t  = [X_t], \qquad t\geq 0.\vspace{-10pt}
\end{equation}
\end{definition}\end{framed}

We recall that (see Definition \ref{def:sol_gen_SDE}) the triple  $( W, (\Fc_t)_t , X)$ above is given, respectively, by a complete filtration, a Brownian motion relative to the filtration $(\Fc_t)_t$, and a progressively measurable process $X$, all three being defined on an underlying probability space $(\Omega, \Fc , \Pb)$, which is not explicitly included in the notation.

\begin{framed}\noindent
{\bf(!)} For ease of reading, we will sometimes refer to a solution simply by the component $X$ of the triple. However, the reader should bear in mind that the filtration, the Brownian motion, and hence the underlying probability space, are part of the solution.
\end{framed}

\begin{remark}\label{rem:solu}
By Definition \ref{def:sol_gen_MKV}, a solution $X$ to MKV$({\bf b},{\boldsymbol\sigma})$ is, in particular, a solution to a standard SDE (the frozen one). In particular, $X$ is almost surely continuous, and thus the flow $(t\mapsto {\boldsymbol\mu}_t = [X_t])$ is continuous.
Also, by Theorem \ref{th:equiv_MP_weak}, the law of $X$ solves the martingale problem MP$({\bf b}_{\boldsymbol\mu},{\boldsymbol\sigma}_{\boldsymbol\mu};0,[X_0])$ in the sense of Stroock-Varadhan (see Definition \ref{def:MP_lin}). If weak uniqueness holds for SDE$({\bf b}_{\boldsymbol\mu},{\boldsymbol\sigma}_{\boldsymbol\mu})$ (equivalently uniqueness for MP$({\bf b}_{\boldsymbol\mu},{\boldsymbol\sigma}_{\boldsymbol\mu}$)), then we have two important consequences:
\begin{itemize}
\item[(i)] 
If $X'$ is another solution to MKV$({\bf b},{\boldsymbol\sigma})$ with the same time marginals as $X$, namely $[X'_t] =\boldsymbol\mu_t  = [X_t]$ for any $t\geq 0$, then $X'$ solves the same frozen SDE with the same initial distribution, and thus $[X]=[X']$. In other words, the law of a solution $X$ to MKV$({\bf b},{\boldsymbol\sigma})$ is determined by its time-marginal flow ${\boldsymbol\mu}_t = [X_t]$, as long as the solution to the frozen SDE is unique in law.
\item[(ii)] By \cite[Th. 6.2.2]{stroock_multidimensional_1979}, the law $[X]$ is a (strong) Markov process on $(C_{\infty}, \Bc)$ in the standard sense of \cite[Definition 2.2.1]{stroock_multidimensional_1979}, with transition kernel given by 
\begin{equation}
P(t, x; T , \dd y) := \Pb^{\boldsymbol\mu}_{t, \delta_x}(x_T \in dy), \qquad 0\leq t \leq T, \ x\in\Rd, 
\end{equation} 
where we denote by $\Pb^{\boldsymbol\mu}_{t, \delta_x}$ the solution to MP$({\bf b}_{\boldsymbol\mu},{\boldsymbol\sigma}_{\boldsymbol\mu};t,\delta_x)$. Notice instead that the flow property typically fails for MKV-SDEs. This causes the backward Kolmogorov operator associated to a MKV-SDE to include partial derivatives with respect to probability measures (see \cite{de2022well}).
\end{itemize}
We recall (see \cite[Th. 6.2.3]{stroock_multidimensional_1979}) that uniqueness for MP$({\bf b}_{\boldsymbol\mu},{\boldsymbol\sigma}_{\boldsymbol\mu}$) holds if and only if, for any $(t,x)\in[0,\infty)\times \Rd$, the time-marginals of $\Pb^{\boldsymbol\mu}_{t, \delta_x}$ are unique.
\end{remark}

In analogy with standard SDEs, we can now define the notions of weak and pathwise solutions to a MKV-SDE.

\begin{framed}
\vspace{-10pt}
\begin{definition}[Weak solution to a MKV-SDE]\label{def:weak_sol_MKV}
Let \textcolor{black}{$\nu\in \mathcal{P}(\R^d)$}. A (weak) solution to MKV$({\bf b},{\boldsymbol\sigma};\textcolor{black}{\nu})$ 
is a 
(general) solution to MKV$({\bf b},{\boldsymbol\sigma})$ such that
\begin{equation}
[X_0] =  \nu. \vspace{-10pt}
\end{equation}
\end{definition}
\end{framed}

\begin{framed}\vspace{-10pt}
\begin{definition}[Pathwise solution to a MKV-SDE]\label{def:path_sol_MKV}
Let $(W,(\Fc_t)_t)$ be a Brownian motion and \textcolor{black}{$\xi\in m\Fc_0$}. A (pathwise) solution to MKV$({\bf b},{\boldsymbol\sigma};\textcolor{black}{ \xi, W, (\Fc_t)_t})$ 
is a 
process $X$ such that the triple $(\textcolor{black}{W}, (\Fc_t)_t, X)$ 
is a (general) solution to MKV$({\bf b},{\boldsymbol\sigma})$ with 
\begin{equation}
\textcolor{black}{X_0 = \xi}. \vspace{-10pt}
\end{equation}
\end{definition}
\end{framed}


The situation can be summarized as in Table \ref{tab:weak_pathwise}. Consequently, as for standard SDEs, we have a weak and a pathwise notion of well-posedness (existence and uniqueness) for MKV-SDEs.

\begin{table}[h]
\begin{center}
\begin{tabular}{c|c|c}
 &weak formulation (MP) & pathwise formulation\\
   \hline
   Brownian motion & part of the solution & fixed \\
   initial datum & $X_0 \sim \nu$ & $X_0 = \xi$
    \end{tabular}
    \end{center}
  \label{tab:weak_pathwise}
  \caption{Weak and pathwise formulation.}
\end{table} 
\begin{framed}\vspace{-10pt}
\begin{definition}[Weak/pathwise well-posedness for a MKV-SDE]\label{def:well_posed}
We say that MKV$({\bf b},{\boldsymbol\sigma})$ is:
\begin{itemize}
\item[-] \emph{\textcolor{black}{weakly} well-posed}, if there exists a unique (\textcolor{black}{in law}) solution to MKV$({\bf b},{\boldsymbol\sigma};\textcolor{black}{\nu})$ for any $\nu\in \Pc(\R^d)$;
\item[-] \emph{\textcolor{black}{pathwise} well-posed}, if there exists an (\emph{almost surely}) unique solution to MKV$({\bf b},{\boldsymbol\sigma};\textcolor{black}{\xi, W, (\Fc_t)_t})$ for any $(W, (\Fc_t)_t )$ Brownian motion  and any \textcolor{black}{$\xi\in m\Fc_0$}. \vspace{-8pt}
\end{itemize} 
\end{definition}
\end{framed}

\begin{example}
Consider the MKV-SDE
\begin{equation}\label{eq:example_F_E}
dX_t = F \big( \Eb[X_t] \big) dW_t, 
\qquad t>0.
\end{equation}
If $X$ is a solution, then we must have, a priori, that the function $(t\mapsto F \big( \Eb[X_t] \big))$ is square-integrable, and thus $X$ is a martingale. Therefore, we directly obtain
\begin{equation}\label{eq:example_F_E_expected}
\Eb[X_t] = \Eb[X_0] , \qquad t>0.
\end{equation}
This yields 
\begin{equation}
X_t  = X_0 + F\big(\Eb[X_0]\big) W_t, \qquad t>0.
\end{equation}
On the other hand, it is obvious that defining $X$ as above implies \eqref{eq:example_F_E_expected}.
This shows that \eqref{eq:example_F_E} is pathwise well-posed.
\end{example}

The following result can be proved similarly to the case of standard SDEs. We leave the proof as an exercise for the reader. 
\begin{remark}\label{prop:pathwise_weak}
If MKV$({\bf b},{\boldsymbol\sigma})$ is pathwise well-posed, then it is weakly well-posed. 
Notice that this claim does not stem directly from the corresponding result for standard SDEs. In particular, when showing that 
\begin{equation}
\text{pathwise uniqueness}\quad \Longrightarrow\quad \text{weak uniqueness},
\end{equation}
one does not know, a priori, that two solutions $X$, $X'$ solve the same frozen SDE. This is true only if $X$, $X'$ have the same time-marginals (cf. Remark \ref{rem:solu}-(i)), which is part of the weak uniqueness claim.
\end{remark}

The following result, which strongly relies on the Watanabe-Yamada Theorem (see Theorem \ref{th:watanabe_yamada}), provides a useful criterion to infer the pathwise well-posedness from weak well-posedness.

\begin{lemma}\label{lem:wat_yam_mkv}
Let $\Cc$ be a fixed subset of the space of measurable functions on $[0,\infty)$ with values in $\mathcal{P}(\Rd)$. 
Let $\nu\in \Pc(\Rd)$ and assume there exists a (weak) solution $X$ to MKV$({\bf b},{\boldsymbol\sigma};\textcolor{black}{\nu})$, whose time-marginals are unique in $\Cc$. Assume that weak existence and pathwise uniqueness hold for SDE$({\bf b}_{\boldsymbol\mu},{\boldsymbol\sigma}_{\boldsymbol\mu};0)$ (see Definition \ref{def:well_posed_lin}), with ${\boldsymbol\mu} = [X_{\cdot}]$. Then, for any $(W_t , (\Fc_t)_t )$ Brownian motion and any \textcolor{black}{$\xi\in m\Fc_0$} with $[\xi] = \nu$, there exists a (\emph{pathwise}) unique solution to MKV$({\bf b},{\boldsymbol\sigma};\textcolor{black}{\xi, W, (\Fc_t)_t})$ whose time-marginals belong to $\Cc$.
\end{lemma}
\begin{proof}
By the Yamada-Watanabe Theorem (see Theorem \ref{th:watanabe_yamada}), the frozen SDE$({\bf b}_{\boldsymbol\mu},{\boldsymbol\sigma}_{\boldsymbol\mu};0)$ is pathwise well-posed. Now fix a $(W_t , (\Fc_t)_t )$ Brownian motion and \textcolor{black}{$\xi\in m\Fc_0$} with $[\xi] = \nu$.

\emph{Uniqueness:} Let $Z,\tilde Z$ be two (pathwise) solutions to MKV$({\bf b},{\boldsymbol\sigma};\textcolor{black}{\xi, W, (\Fc_t)_t})$  whose time marginals belong to $\Cc$. By assumption, $[Z_t] = [\tilde Z_t] = [X_t]$. Therefore, $Z$ and $\tilde Z$ are both (pathwise) solutions to SDE$({\bf b}_{\boldsymbol\mu},{\boldsymbol\sigma}_{\boldsymbol\mu} ; \xi, W, (\Fc_t)_t)$, and thus they are equal almost surely.

\emph{Existence:} Let $Z$ be a (pathwise) solution to SDE$({\bf b}_{\boldsymbol\mu},{\boldsymbol\sigma}_{\boldsymbol\mu} ;0, \xi, W, (\Fc_t)_t)$. In particular, both $Z$ and $X$ are weak solutions to SDE$({\bf b}_{\boldsymbol\mu},{\boldsymbol\sigma}_{\boldsymbol\mu} ; \nu)$ and thus they have the same law. Therefore, $[Z_t] = [X_t] = {\boldsymbol\mu}_t $ for any $t\geq 0$, which means that $Z$ is a (pathwise) solution to MKV$({\bf b}_{\boldsymbol\mu},{\boldsymbol\sigma}_{\boldsymbol\mu};\xi, W, (\Fc_t)_t)$.
\end{proof}
The next proposition is a direct consequence of the previous lemma.
\begin{proposition}\label{prop:wea_strong_MKV}
If MKV$({\bf b},{\boldsymbol\sigma})$ is weakly well-posed and all its corresponding frozen SDEs enjoy pathwise uniqueness, then MKV$({\bf b},{\boldsymbol\sigma})$ is pathwise well-posed.
\end{proposition}

Proposition \ref{prop:wea_strong_MKV} suggests the following pathway to proving pathwise well-posedness for MKV-SDEs:
\begin{framed}\vspace{-0pt}
{\bf ``Stairway" to pathwise well-posedness:}
\begin{itemize}[leftmargin=35pt]

\item[Step 1.] \underline{Prove weak well-posedness for frozen SDEs}: for a fixed measurable flow ${\boldsymbol\nu}:[0,\infty)\to \Pc(\Rd)$, prove weak well-posedness of SDE$({\bf b}_{\boldsymbol\nu},{\boldsymbol\sigma}_{\boldsymbol\nu};0)$.

For this task, it is typically required to check that the frozen coefficients ${\bf b}_{\boldsymbol\mu}(t,\cdot)$, ${\boldsymbol\sigma}_{\boldsymbol\mu}(t,\cdot)$ enjoy suitable regularity and satisfy geometric assumptions (e.g. H\"ormander-type conditions) that ensure non-degeneracy of the noise. Note that the regularity in space of the frozen coefficients is only determined by the regularity of the coefficients ${\bf b}, {\boldsymbol\sigma}$ with respect to the state variable. Also, ${\bf b}_{\boldsymbol\mu}$, ${\boldsymbol\sigma}_{\boldsymbol\mu}$ are necessarily jointly (in time-space) Borel measurable on $[0,\infty)\times\R^d$ (see Remarks \ref{rem:frozen_coeff} and \ref{rem:solu}).

\item[Step 2.] \underline{Prove weak well-posedness for MKV$({\bf b},{\boldsymbol\sigma})$}: by Step 1. and Remark \ref{rem:solu}, this is equivalent to finding a unique fixed point of the map defined on the space of measurable functions from $[0,\infty)$ to $\Pc(\Rd)$ into itself, given by
\begin{equation}
 {\boldsymbol\nu} \mapsto [Z^{\boldsymbol\nu}_{\cdot}], \quad \text{with $Z^{\boldsymbol\nu}$ the unique (weak) solution to SDE$({\bf b}_{\boldsymbol\nu},{\boldsymbol\sigma}_{\boldsymbol\nu};0, {\boldsymbol\nu}_0)$}.
\end{equation}
To produce the estimates that are necessary to show the existence of a unique fixed point, the regularity of the coefficients ${\bf b}$ and ${\boldsymbol\sigma}$, in particular with respect to the measure argument, plays a crucial role.

\item[Step 3.] \underline{Prove pathwise uniqueness for the frozen SDEs}: for any (weak) solution $X$ to MKV$({\bf b},{\boldsymbol\sigma})$, prove that pathwise uniqueness holds for SDE$({\bf b}_{\boldsymbol\mu},{\boldsymbol\sigma}_{\boldsymbol\mu};0)$, with ${\boldsymbol\mu} = [X_{\cdot}]$. 
As in Step 1, this task typically requires checking regularity and non-degeneracy conditions for ${\bf b}_{\boldsymbol\mu}(t,\cdot)$, ${\boldsymbol\sigma}_{\boldsymbol\mu}(t,\cdot)$.

\item[Step 4.] \underline{Infer pathwise well-posedness for MKV$({\bf b},{\boldsymbol\sigma})$}: by the previous two steps, it is sufficient to apply Proposition \ref{prop:wea_strong_MKV}. \vspace{-6pt}
\end{itemize}
\end{framed}

\begin{remark}{\bf(!)} In some irregular cases, it may happen that the MKV-SDE is well posed but its corresponding frozen SDEs are not. In particular, the latter may fail to enjoy even weak uniqueness (see Example \ref{example:counter_linearl} below).
\end{remark}

\begin{example}\label{example:counter_linearl}
Let $ {\bf b} \equiv 0$ and 
\begin{equation}
{\boldsymbol\sigma}(t,x, \nu) = 
\begin{cases}
\text{sgn}(x), & \text{if $\nu(\{0\}) = 0$},  \\
1, & \text{\text{if $\nu(\{0\}) > 0$}},  \\
\end{cases}
\end{equation}
where
\begin{equation}\label{eq:sign_funct}
\text{sgn}(x): = 
\begin{cases}
1 & \text{if } x>0,  \\
0 & \text{if } x=0,  \\
-1 & \text{if } x<0.
\end{cases}\qquad x\in\R.
\end{equation}
Let $X$ be a Brownian motion relative to a given filtration $(\Fc_t)_t$. We can set
\begin{equation}
W_t : = \int_0^t \text{sgn}(X_s) dX_s, \qquad t\geq 0,
\end{equation} 
which is again a Brownian motion relative to $(\Fc_t)_t$, for $|\text{sgn}(X_t)| = 1$ almost surely. Furthermore, we have
\begin{equation}
X_t = \int_0^t (\text{sgn}(X_t))^2 d X_t = \int_0^t \text{sgn}(X_t) d W_t =  \int_0^t {\boldsymbol\sigma}(t,X_t, [X_t]) dW_s,
\end{equation}
which means that $(X,(\Fc_t)_t,W)$ is a solution to MKV$({\bf b},{\boldsymbol\sigma}; \delta_0)$. 

Conversely, let $X$ be a solution to the latter problem. Then, we have
\begin{equation}
X_t = \int_0^t {\boldsymbol\sigma}(t,X_t, [X_t]) dW_s, \qquad t\geq 0,
\end{equation}
and thus, owing to the fact that $|{\boldsymbol\sigma}(t,X_t, [X_t])|=1$ almost surely, $X$ is necessarily a Brownian motion. 

To sum up, there exists a solution to MKV$({\bf b},{\boldsymbol\sigma};\delta_0)$; each solution has the same law (the Wiener measure) and must solve the standard SDE
\begin{equation}\label{eq:tanaka_bis}
dZ_t = \text{sgn}(Z_t) dW_t, \qquad Z_0 \sim \delta_0.
\end{equation}
However, the latter has one more solution, which is the null process. Note that the sign function in \eqref{eq:sign_funct} slightly differs from the well-known Tanaka's example, where $\text{sgn}(0) := 1$. With this modification, the solution to \eqref{eq:tanaka_bis} is unique in law as the null process is no longer a solution.
\end{example}

We conclude this section with a useful characterization of weak solutions to MKV-SDEs. In particular, being the latter solutions to the frozen (standard) SDE, one can exploit the well-known correspondence between the solutions to an It\^o SDE and to the associated martingale problem.

\begin{definition}[Non-linear martingale problem]\label{def:non_lin_MP}
Let $\nu\in\Pc(\Rd)$. A probability measure $\Pb_{\nu}$, on $(C_{\infty},\Bc)$, is a solution to NL-MP$({\bf b},{\boldsymbol\sigma};\nu)$ if it is a solution to MP$({\bf b}_{\boldsymbol\mu}, {\boldsymbol\sigma}_{\boldsymbol\mu};0,\nu)$ (see Definition \ref{def:MP_lin}), with 
\begin{equation}\label{eq:cond_non_lin_mp}
\boldsymbol\mu_t = \Pb_{\nu}(x_t \in \dd y), \qquad t\geq 0,
\end{equation}
where $(x_t)_{t\in[0,\infty)}$ denotes the canonical process.
\end{definition}


\begin{proposition}\label{prop:equivalence_nonlinear_mp_weak_sol}
Let $\nu\in\Pc(\Rd)$, and let $\Pb_{\nu}$ be a probability measure on $(C_{\infty}, \Bc )$, such that
\begin{equation}\label{eq:prop_l1_loc}
\big[(t,x)\mapsto {\bf b}(t,x,{\boldsymbol\mu}_t)\big], \ \big[ (t,x)\mapsto {\boldsymbol\sigma} {\boldsymbol\sigma}^\top(t,x,{\boldsymbol\mu}_t)\big] \in L^{1}_{\text{loc}}([0,\infty) \times \Rd , \overline{\boldsymbol\mu}), 
\end{equation}
with ${\boldsymbol\mu}_t:= \Pb_{\nu}(x_t \in \dd y)$ and $\overline{\boldsymbol\mu}$ as in Notation \ref{not:measure_dt_dx}. 

Then $\Pb_{\nu}$ is a solution to NL-MP$({\bf b},{\boldsymbol\sigma};\nu)$ if and only if there exists a (weak) solution $X$ to MKV$({\bf b},{\boldsymbol\sigma};\nu)$ such that $[X] =\Pb_{\nu} $.
\end{proposition}
\begin{proof}
$\Pb_{\nu}$ is a solution to NL-MP$({\bf b},{\boldsymbol\sigma};\nu)$ if and only if it solves MP$({\bf b}_{\boldsymbol\mu}, {\boldsymbol\sigma}_{\boldsymbol\mu};0,\nu)$ with 
\eqref{eq:cond_non_lin_mp}. Therefore, by Theorem \ref{th:equiv_MP_weak}, the latter is true if and only if there exists a solution $X$ to SDE$({\bf b}_{\boldsymbol\mu}, {\boldsymbol\sigma}_{\boldsymbol\mu};\nu)$ such that $[X] = \Pb_{\nu}$, 
which is equivalent to saying that $X$ is a (weak) solution to MKV$({\bf b},{\boldsymbol\sigma};\nu)$ with $[X] = \Pb_{\nu}$. 
%
\end{proof}
The following characterization of weak well-posedness is a direct consequence of Proposition \ref{prop:equivalence_nonlinear_mp_weak_sol}.
\begin{corollary}
Assume that \eqref{eq:prop_l1_loc} holds for any measurable flow ${\boldsymbol\mu}:[0,\infty)\to \Pc(\Rd)$.
Then, MKV$({\bf b},{\boldsymbol\sigma})$ is weakly well-posed if and only if NL-MP$({\bf b},{\boldsymbol\sigma};\nu)$ admits a unique solution for any $\nu\in\Pc(\Rd)$.
\end{corollary}

\section{The non-linear Fokker Planck equation}\label{sec:nonlin_fp}

When McKean introduced McKean-Vlasov diffusions, in \cite{mckean1966class}, his original goal was to provide a stochastic representation of the solutions to non-linear PDEs arising from physical applications, such as Boltzmann and Vlasov equations, which can be written as non-linear Fokker-Planck Cauchy problems of the form
\begin{equation}\label{eq:nl_fp_cp}
\hspace{-0pt}{\blue\text{[NL-FP$({\bf b},{\boldsymbol\sigma};\nu)$]}}\qquad\qquad
\begin{cases}
 \partial_t {\boldsymbol\mu}_t = \text{div}_y \Big(\!- {\bf b}(t,y,\textcolor{black}{{\boldsymbol\mu}_t}) \,  {\boldsymbol\mu}_t  +\frac{1}{2}\text{\bf div}_y \big(    {\boldsymbol\sigma}{\boldsymbol\sigma}^\top(t,y,\textcolor{black}{{\boldsymbol\mu}_t}) \, {\boldsymbol\mu}_t \big)   \Big), \quad t> 0, \\
 {\boldsymbol\mu}_0 = \nu \in  \mathcal{M}(\Rd),
 \end{cases}
\end{equation}
where $\text{div}_y$ and $\text{\bf div}_y$ denote the standard and matrix divergence operators defined in \eqref{eq:div_stand} and \eqref{eq:div_matrix}.
The coefficients in \eqref{eq:nl_fp_cp} are measurable functions 
\begin{equation}\label{eq:b_sig_fp}
{\bf b}: [0,\infty)\times \Rd \times \mathcal{M}(\Rd)\to \R^d, \qquad {\boldsymbol\sigma}: [0,\infty)\times \Rd \times \mathcal{M}(\Rd)\to \mathcal{M}^{d\times q}.
\end{equation}
Here, $\mathcal{M}(\Rd)$ denotes the set of finite (non-negative) measures on $(\R^d,\Bc)$, equipped with the topology of weak convergence. 
\begin{framed}
\noindent {\bf(!)} 
Even though our main focus is the connection with MKV-SDE \eqref{eq:mkv_SDE}, in which the measure argument of ${\bf b}$ and ${\boldsymbol\sigma}$ is a probability measure, it is convenient to consider in \eqref{eq:nl_fp_cp} a measure argument which is, more generally, a non-negative finite measure. The reason is that, in some cases, the solutions to \eqref{eq:nl_fp_cp} may naturally be identified in a subspace of $\mathcal{M}(\Rd)$ that is not included in the set of probability measures, even if $\nu\in \mathcal{P}(\Rd)$. A posteriori, superposition results (see below) ensure that the solution ${\boldsymbol\mu}_t$ is indeed a probability measure. 
\end{framed}
\begin{notation}[Abuse !]
Given ${\bf b}$, ${\boldsymbol\sigma}$ as in \eqref{eq:b_sig_fp}, we denote by ${\bf b}$, ${\boldsymbol\sigma}$ also their restrictions to $[0,\infty)\times\Rd \times \Pc(\Rd)$. 
\end{notation}

In analogy with the connection between linear Fokker-Planck equations and standard SDEs, the link between non-linear Fokker-Planck equations and the McKean-Vlasov SDEs has a twofold direction: 
\begin{itemize}
\item[(i)] On the one hand, solutions to MKV$({\bf b},{\boldsymbol\sigma};\nu)$ are It\^o processes, and thus It\^o formula readily yields that their time marginals are distributional (hereafter \emph{weak}) solutions to NL-FP$({\bf b},{\boldsymbol\sigma};\nu)$. 
\item[(ii)] On the other hand, the Ambrosio-Figalli-Trevisan superposition principle (e.g. \cite{figalli2008existence} under the assumption of bounded coefficients, and \cite{trevisan2016well} under weaker conditions) shows that for any weak solution ${\boldsymbol\mu}$ to NL-FP$({\bf b},{\boldsymbol\sigma};\nu)$, there exists a solution to NL-MP$({\bf b},{\boldsymbol\sigma};\nu)$ whose time-marginals coincide with ${\boldsymbol\mu}_t$, and thus a solution to MKV$({\bf b},{\boldsymbol\sigma};\nu)$ with time marginals ${\boldsymbol\mu}_t$.
\end{itemize}
We start by defining the notion of weak solution to the non-linear Fokker-Planck equation.

\begin{framed}
\vspace{-10pt}
\begin{definition}[Weak solution to non-linear Fokker-Planck Cauchy problem]\label{def:sol_nl_fp}
Let $\nu\in \mathcal{M}(\Rd)$. A measurable flow ${\boldsymbol\mu}:[0,\infty)\to \mathcal{M}(\Rd)$ is a (weak) solution to NL-FP$({\bf b},{\boldsymbol\sigma};\nu)$ if it is a weak solution to FP$({\bf b}_{\boldsymbol\mu},{\boldsymbol\sigma}_{\boldsymbol\mu};0,\nu)$ (see Definition \ref{def:sol_lin_fp}), with ${\bf b}_{\boldsymbol\mu},{\boldsymbol\sigma}_{\boldsymbol\mu}$ as in \eqref{eq:def_frozen_coeff}.
%

${}$
\vspace{-8pt}
\end{definition}\end{framed}
\begin{remark}
In Definition \ref{def:sol_nl_fp}, it is implicit that the solution ${\boldsymbol\mu}$ is continuous (see Remark \ref{rem:sol_weak_fp}). 
\end{remark}


\begin{lemma}\label{lem:fokker_superpo}
Let $\nu\in\Pc(\Rd)$ and let ${\boldsymbol\mu}:[0,\infty)\to \mathcal{M}_{+}(\Rd)$ (with ${\boldsymbol\mu}_0 = \nu$) be such that
\begin{equation}\label{eq:cond_L1_lo}
\big[(t,x)\mapsto {\bf b}(t,x,{\boldsymbol\mu}_t)\big],  \ \big[ (t,x)\mapsto {\boldsymbol\sigma} {\boldsymbol\sigma}^\top(t,x,{\boldsymbol\mu}_t)\big] \in L^{1}_{\text{loc}}([0,\infty) \times \Rd , \overline{\boldsymbol\mu}),
\end{equation}
with $\overline{\boldsymbol\mu}$ as in Notation \ref{not:measure_dt_dx}.
Then:
\begin{itemize}
\item[(i)] If $X$ is a solution to MKV$({\bf b},{\boldsymbol\sigma};\nu)$, with $[X_t]={\boldsymbol\mu}_t $ for $t\geq0$, then ${\boldsymbol\mu}$ is a weak solution to NL-FP$({\bf b},{\boldsymbol\sigma};\nu)$.
\item[(ii)] [Superposition] Under the additional assumption that 
\begin{equation}\label{eq:cond_superpo}
\Big[ t \mapsto \int_{\Rd} \Big( |{\bf b}(t,x,{\boldsymbol\mu}_t)| + \| {\boldsymbol\sigma}{\boldsymbol\sigma}^\top(t,x,{\boldsymbol\mu}_t) \| \Big) {\boldsymbol\mu}_t (\dd x)\Big], \  \big[ t \mapsto {\boldsymbol\mu}_{t} (\Rd) \big]  \in L^{\infty}_{\text{loc}}([0,\infty) ,
\end{equation}
if ${\boldsymbol\mu}$ is a (weak) solution to NL-FP$({\bf b},{\boldsymbol\sigma};\nu)$, 
 then there exists a solution $X$ to MKV$({\bf b},{\boldsymbol\sigma};\nu)$ such that $ [X_t] = {\boldsymbol\mu}_t $.
\end{itemize}
\end{lemma}

\begin{proof} Conditions \eqref{eq:cond_L1_lo} and \eqref{eq:cond_superpo} imply, respectively, conditions \eqref{eq:cond_L1} and \eqref{eq:cond_agg_super} in Proposition \ref{th:lin_FP}, with 
\begin{equation}
b(t,x):= {\bf b}(t,x,{\boldsymbol\mu}_t), \qquad \sigma(t,x):= {\boldsymbol\sigma}(t,x,{\boldsymbol\mu}_t).
\end{equation}
Thus the statement stems directly from Proposition \ref{th:lin_FP} and from the fact that $X$ and ${\boldsymbol\mu}$ are solutions to MKV$({\bf b},{\boldsymbol\sigma};\nu)$ and NL-FP$({\bf b},{\boldsymbol\sigma};\nu)$ if and only if they are solutions, respectively, to MKV$({\bf b}_{\boldsymbol\mu},{\boldsymbol\sigma}_{\boldsymbol\mu};0,\nu)$ and FP$({\bf b}_{\boldsymbol\mu},{\boldsymbol\sigma}_{\boldsymbol\mu};0,\nu)$, with ${\bf b}_{\boldsymbol\mu},{\boldsymbol\sigma}_{\boldsymbol\mu}$ as in \eqref{eq:def_frozen_coeff}.

%
\end{proof}

\begin{remark}[(!)]
Note that Lemma \ref{lem:fokker_superpo} is not enough to conclude that MKV$({\bf b},{\boldsymbol\sigma};\nu)$ admits a unique (in law) solution if and only if NL-FP$({\bf b},{\boldsymbol\sigma};\nu)$ admits a unique solution. In particular, uniqueness for MKV$({\bf b},{\boldsymbol\sigma};\nu)$ may fail even if it holds for NL-FP$({\bf b},{\boldsymbol\sigma};\nu)$. Indeed, Lemma \ref{lem:fokker_superpo} only implies that, under suitable assumptions on the coefficients, two solutions to MKV$({\bf b},{\boldsymbol\sigma};\nu)$ necessarily have the same time-marginals if the solution to NL-FP$({\bf b},{\boldsymbol\sigma};\nu)$ is unique. However, two solutions to 
a MKV SDE  may have the same time-marginals but different laws, if weak uniqueness does not hold for the linearized SDE.
\end{remark}

\begin{example}[Landau equation]
A notable case of non-linear PDE which can be represented as a Fokker-Planck equation is the 
Landau equation on $\R^3$, namely
\begin{equation}\label{eq:landau_kin}
\partial_t v(y) 
 = \text{div}_{y} \int_{\mathbb{R}^3} A_{\gamma}(y - w) \big( v(w) \nabla v(y) - v(y) \nabla v(w) \big) dw, 
\end{equation}
with $A_{\gamma}$ being the Coulomb collision potential
\begin{equation}
\mathcal{M}^{3\times 3}\ni A_{\gamma}(u) := |u|^{2+\gamma} \Big(  I_3 - \frac{ u \otimes u}{|u|^{2}}\Big),\qquad u\in\R^3,
\end{equation}
with $\gamma \in [-3,1]$. A direct computation shows that, formally, equation \eqref{eq:landau_kin} can be cast in the form \eqref{eq:nl_fp_cp} with 
\begin{equation}
{\bf b}(t, y,\nu) = \int_{\R^3} {\bf div}_y A_{\gamma}(y - z) \nu(d z), \quad  {\boldsymbol\sigma}(t, y,\nu) = \Big( \int_{\R^3}  A_{\gamma}(y - z) \nu(d z) \Big)^{\frac{1}{2}}.
\end{equation}

In the case $\gamma \in (-2,0]$, in \cite{fournier2016propagation}, the authors proved well-posedness for the Landau equation in Fokker-Planck form and pathwise well-posedness of the associated MKV SDE, together with some stability and regularity estimates, and also propagation of chaos for the associated particle system. Note that, in the regime $\gamma\in (-2,-1)$, the kernel $A_{\gamma}$ is only H\"older-continuous.


The most singular case, i.e. $\gamma = -3$, is arguably the most meaningful from the physical standpoint, as it represents Coulomb interactions between charged particles in a plasma. In this regime, the existence of global regular solutions to \eqref{eq:landau_kin}, together with Gaussian upper bounds, has been proved only recently in \cite{guillen2025landau}, solving a problem that had been standing for decades. 
\end{example}

\section{The mean-field particle system}\label{sec:part_sys}

In the previous section we presented the connection, studied by McKean in his 1966 seminal paper (\cite{mckean1966class}), between non-linear Fokker-Planck equations and the class of distribution-dependent SDEs in \eqref{eq:mkv_SDE}, which later became known as McKean-Vlasov equations. In his subsequent work (\cite{mckean1967propagation}),  McKean revealed the connection between \eqref{eq:mkv_SDE} and the particle system formally given by
\begin{equation}\label{eq:particle_sys}
\begin{cases}
dX^{i,N}_t = {\bf b}\big(t,X^{i,N}_t,\mu^{\text{emp}}_{ X_t^{N}}\big) dt + {\boldsymbol\sigma}\big(t,X^{i,N}_t,\mu^{\text{emp}}_{ X_t^{N}}\big) dW^i_t ,\qquad i=1,\dots, N \\
\mu^{\text{emp}}_{X_t^{N}}:= \frac{1}{N} \sum_{j=1}^N \delta_{X^{j,N}_t},
\end{cases}
\end{equation}
where $N\in\N$, the functions ${\bf b}$, ${\boldsymbol\sigma}$ are as in \eqref{eq:coefficients_mkv}, and $W^i$, $i=1,\dots, N$, are independent Brownian motions. Building on ideas that originated from Kac's pioneering work in kinetic theory (see \cite{kac1956foundations}), McKean realized that the MKV SDE \eqref{eq:mkv_SDE} can be understood as the limiting equation of \eqref{eq:particle_sys} as the number of particles $N$ tends to infinity. These types of results, known as \emph{propagation of chaos} (PoC), can be morally formulated as follows.

\begin{framed}\vspace{-0pt}
{\bf Propagation of Chaos: the philosophic statement.} If $X^{i,N}_0$, $i=1,\dots, N$, are i.i.d. with distribution $\nu$, as the number of particles $N\to \infty$,
\begin{itemize}
\item[(i)] the components $X^{i,N}$ (particles) are asymptotically independent,
\item[(ii)] and each component $X^{i,N}$ tends to $X$, the solution to the MKV$({\bf b}, {\boldsymbol\sigma};\nu)$,
\end{itemize}
or, somehow equivalently,
\begin{itemize}
\item[(iii)] the empirical law $\mu^{\text{emp}}_{X_t^{N}}$ tends to the law of $X_t$.
\end{itemize}
\vspace{-10pt}
\end{framed}

\begin{remark}
The equivalence between (i)-(ii) and (iii) can be morally understood as a Glivenko-Cantelli type result. Indeed, when ${\bf b}$, ${\boldsymbol\sigma}$ are independent of the measure variable, the random variables $X^{i,N}$, $i=1,\dots, N$, are independent and identically distributed for any finite $N$, and thus the convergence of the empirical law is ensured by Glivenko-Cantelli's theorem.
\end{remark}

Of course, the formal statement above has a number of ``problems" that need to be addressed, starting with the well-posedness of the particle system \eqref{eq:particle_sys}, and then continuing with the type of convergence considered in (ii), which could be in law or in a stronger sense, for instance in $L^p$. And then again, what type of convergence do we mean in (iii)? The empirical measure is a random probability, so its convergence in the weak topology can be itself weak (in the space of probability measures on probability measures) or stronger (e.g. in probability, in $L^p$, etc.). The next questions then involve the quantification of convergence rates, which become particularly relevant when combining PoC with a time-discretization method for the simulation of the particle system.

The answers to the questions above strongly depend on the underlying assumptions on the model, in particular on the coefficients ${\bf b}, {\boldsymbol\sigma}$. 
For this reason, contrary to Section \ref{sec:nonlin_fp}, where we have general theorems connecting MKV SDEs with non-linear FP PDEs under very mild assumptions on the coefficients, we do not have general PoC results without requiring any type of regularity assumption on the coefficients. As the focus of these notes is not on PoC, we refer to the classic reference \cite{Sznitman91} for a detailed introduction to this topic. We also refer to \cite{carmona2016lectures} for an overview of the standard results in the Lipschitz framework and to the recent, comprehensive, survey \cite{chaintron2022propagation,chaintron2022propagationb}.
 


\section{Density-dependent McKean-Vlasov SDEs}\label{sec:den_dep_mkv}

We discuss a relevant subclass of the MKV-SDEs described by \eqref{eq:mkv_SDE}, characterized by the following feature: the coefficients ${\bf b}$, ${\boldsymbol\sigma}$ depend on the pointwise evaluation of the density of $[X_t]$ at the point $X_t$. While, in principle, all the results of Sections \ref{sec:mkv_diffusions} and \ref{sec:nonlin_fp} do apply to this subclass, the following caveat must be taken into account: \vspace{-5pt}
\begin{framed}\noindent
{\bf(!)} The notions of solutions to the MKV-SDE, to the associated non-linear martingale problem, and to the associated non-linear Fokker-Planck equation should not depend on the specific version of the density of $[X_t]$ that one considers. 
\end{framed}\vspace{-5pt}
For this reason, it is convenient to introduce specific notations and definitions for this particular subclass, which is connected to a wide range of applications.

Consider the formal equation
\begin{equation}\label{eq:mkv_SDE_dens}
\hspace{-70pt}{\blue\text{[D-MKV$(b , \sigma)$]}}\qquad\qquad
\begin{cases}
dX_t  = b\big(t,X_t,v_t(X_t)\big) dt +\sigma\big(t,X_t, v_t(X_t)\big) dW_t, \\
[X_t] = v_t(x) \dd x,
\end{cases}
\end{equation}
with 
\begin{equation}
b: [0,\infty)\times \Rd \times \R\to \R^d,\qquad \sigma: [0,\infty)\times \Rd \times \R\to \mathcal{M}^{d\times q}
\end{equation}
being measurable functions. Namely, the function $v_t$ denotes the density of the time-marginal $[X_t]$, and the coefficients depend on the evaluation of $v_t$ at $X_t$. 

In order to cast \eqref{eq:mkv_SDE_dens} into the framework \eqref{eq:mkv_SDE}, we can define 
\begin{equation}
{\bf b}^*: [0,\infty)\times \Rd \times \mathcal{P}(\Rd)\to \R^d, \qquad {\boldsymbol\sigma}^*: [0,\infty)\times \Rd \times \mathcal{P}(\Rd)\to \mathcal{M}^{d\times q}
\end{equation}
as
\begin{equation}\label{eq:def_b_sig_dens}
\big[{\bf b}^*(t, x  , \nu) , {\boldsymbol\sigma}^*(t, x  , \nu)  \big] : = \begin{cases}
 \big[ b\big(t,x,v_{\nu}^* (x)\big), \sigma\big(t,x,v_{\nu}^* (x)\big) \big], & \text{if } \nu \in AC(\Rd), \\
 0  , & \text{if } \nu \notin AC(\Rd),
\end{cases}
\end{equation}
where 
\begin{equation}\label{eq:density_standard}
v_{\nu}^* (x) : = \begin{cases}
h(x), & \text{if } h(x)<\infty, \\
 0  , & \text{if } h(x)=\infty, 
\end{cases} 
\qquad \text{with}\quad
h(x):=\limsup_{r\downarrow0}\frac{\nu(B_r(x))}{|B_r(x)|}.
\end{equation}
Here $B_r(x)$ denotes the Euclidean ball of radius $r>0$, centered at $x$, and $|B_r(x)|$ denotes its Lebesgue measure. Notice that, if $ \nu \in AC(\Rd)$, then $v_{\nu}^{*}$ is a version of its density, i.e. 
\begin{equation}
\nu(H)  = \int_{H} v_{\nu}^{*}(x) \dd x, \qquad H\in\Bc.
\end{equation}
\begin{remark}[!]
The functions ${\bf b}^* , {\boldsymbol\sigma}^*$ defined by \eqref{eq:def_b_sig_dens} basically select one version of the density of $\nu$, when the latter is absolutely continuous with respect to the Lebesgue measure. Furthermore, they are measurable functions (we leave this as an exercise to the reader). Therefore, one would like to define a solution to D-MKV$(b , \sigma)$ as a solution to MKV$({\bf b}^* , {\boldsymbol\sigma}^*)$. However, one should make sure that this definition is independent of this specific choice of the density.
\end{remark}
Let $\tilde{\bf b} , \tilde{\boldsymbol\sigma}$ be defined as 
\begin{equation}\label{eq:def_b_sig_dens_gen}
\big[\tilde{\bf b}(t, x  , \nu) , \tilde{\boldsymbol\sigma}(t, x  , \nu)  \big] : = \begin{cases}
 \big[ b\big(t,x,\tilde v_{\nu} (x)\big), \sigma\big(t,x, \tilde v_{\nu} (x)\big) \big], & \text{if } \nu \in AC(\Rd), \\
 0  , & \text{if } \nu \notin AC(\Rd),
\end{cases}
\end{equation}
where $\tilde v_{\nu}$ is a chosen version of the density of $\nu$. 
\begin{remark}\label{rem:densities_ver}
Recall that $\tilde v_{\nu},v_{\nu}^* \in L^1(\Rd)$ and 
\begin{equation}\label{eq:v_eq_tildev_leb}
\tilde v_{\nu}=v_{\nu}^*\quad\text{Lebesgue-almost everywhere.}
\end{equation}
Therefore, recalling the definition of the frozen coefficients (see \eqref{eq:def_frozen_coeff}) for a fixed measurable flow ${\boldsymbol\nu}:[0,\infty)\to\Pc(\Rd)$, for any $t\in[0,\infty)$ we have one of the following cases: 
\begin{itemize}
\item If ${\boldsymbol\nu}_t \notin AC(\Rd)$, then 
\begin{equation}
\tilde{\bf b}_{\boldsymbol\nu}(t , x) = 0 ={\bf b}^*_{\boldsymbol\nu} (t , x ) , \quad    \tilde{\boldsymbol\sigma}_{\boldsymbol\nu}(t , x) = 0 ={\boldsymbol\sigma}^*_{\boldsymbol\nu} (t , x ), \qquad x\in\Rd.
\end{equation}
 \item If $\boldsymbol\nu_t \in AC(\Rd)$, then \eqref{eq:v_eq_tildev_leb} yields
 \begin{align}
\tilde{\bf b}_{\boldsymbol\nu}(t , \cdot) &= b\big(t , \cdot , \tilde v_{{\boldsymbol\nu}_t} ( t, \cdot)\big) =  b\big(t , \cdot ,  v_{{\boldsymbol\nu}_t}^* ( t, \cdot)\big) ={\bf b}^*_{\boldsymbol\nu} (t , \cdot ),\\
 \tilde{\boldsymbol\sigma}_{\boldsymbol\nu}(t , \cdot) &= \sigma\big(t , \cdot , \tilde v_{{\boldsymbol\nu}_t} ( t, \cdot)\big) =  \sigma\big(t , \cdot ,  v_{{\boldsymbol\nu}_t}^* ( t, \cdot)\big) ={\boldsymbol\sigma}^*_{\boldsymbol\nu} (t , \cdot ),
\end{align}
Lebesgue-almost everywhere on $\Rd$. 
\end{itemize}
\end{remark}
We have the following
\begin{lemma}[!]
A triple $( W, (\Fc_t)_t, X)$ is a (general) solution to MKV$({\bf b}^* , {\boldsymbol\sigma}^*)$ if and only if it is a (general) solution to MKV$(\tilde{\bf b} , \tilde{\boldsymbol\sigma})$ for any $\tilde{\bf b} , \tilde{\boldsymbol\sigma}$ as defined in \eqref{eq:def_b_sig_dens_gen}.
\end{lemma}
\begin{proof}
Let $( W, (\Fc_t)_t, X)$ be a (general) solution to MKV$({\bf b}^* , {\boldsymbol\sigma}^*)$. We show that it is also a solution to  MKV$(\tilde{\bf b} , \tilde{\boldsymbol\sigma})$. The other way around is proved exactly in the same way. By Definition \ref{def:sol_gen_MKV}, $( W, (\Fc_t)_t, X)$ is a solution to SDE$({\bf b}^*_{\boldsymbol\mu} , {\boldsymbol\sigma}^*_{\boldsymbol\mu})$ with
\begin{equation}
{\boldsymbol\mu}_t  = [X_t], \qquad t\geq 0.
\end{equation}
Consider now the frozen coefficients ${\bf b}^*_{\boldsymbol\mu} , {\boldsymbol\sigma}^*_{\boldsymbol\mu}$ and $\tilde{\bf b}_{\boldsymbol\mu} , \tilde{\boldsymbol\sigma}_{\boldsymbol\mu}$. By Remark \ref{rem:densities_ver} with ${\boldsymbol\nu} = {\boldsymbol\mu}$, we have
\begin{equation}
\tilde{\bf b}_{\boldsymbol\mu}(t , X_t) = {\bf b}^*_{\boldsymbol\mu} (t , X_t ) , \quad    \tilde{\boldsymbol\sigma}_{\boldsymbol\mu}(t , X_t) = {\boldsymbol\sigma}^*_{\boldsymbol\mu} (t ,X_t ), \qquad \text{almost surely},
\end{equation}
for any $t\in[0,\infty)$. 
Thus the processes $[\tilde{\bf b}_{\boldsymbol\mu}(t , X_t), \tilde{\boldsymbol\sigma}_{\boldsymbol\mu}(t , X_t)]$ and $[{\bf b}^*_{\boldsymbol\mu}(t , X_t), {\boldsymbol\sigma}^*_{\boldsymbol\mu}(t , X_t)]$ are modifications. Therefore, $( W, (\Fc_t)_t, X)$ is also a  solution to SDE$(\tilde{\bf b}_{\boldsymbol\mu} , \tilde{\boldsymbol\sigma}_{\boldsymbol\mu})$, and thus a solution to MKV$(\tilde{\bf b}, \tilde{\boldsymbol\sigma})$.
\end{proof}

The previous lemma justifies the following definition.

\begin{framed}
\vspace{-10pt}
\begin{definition}[General solution to a density-dependent MKV-SDE]\label{def:sol_gen_MKV_den}
A (general) solution to D-MKV$(b,\sigma)$ is a triple $( W, (\Fc_t)_t, X)$ which is a solution to MKV$({\bf b}^*,{\boldsymbol\sigma}^*)$, with ${\bf b}^*,{\boldsymbol\sigma}^*$ given by \eqref{eq:def_b_sig_dens}-\eqref{eq:density_standard}, and such that
\begin{equation}
[X_t] \in AC(\Rd), \qquad t> 0. \vspace{-10pt}
\end{equation}
\end{definition}\end{framed}
We also give definitions of weak/pathwise solutions and weak/pathwise well-posedness, analogous to those of Section \ref{sec:mkv_diffusions}.

\begin{framed}
\vspace{-10pt}
\begin{definition}[Weak/pathwise solutions to a density-dependent MKV-SDE]\label{def:weak_path_sol_den_MKV}
The definition of weak (pathwise) solution to  D-MKV$(b, \sigma;\textcolor{black}{\nu})$ (to D-MKV$(b,\sigma;\textcolor{black}{ \xi, W, (\Fc_t)_t})$) is given as in Definition \ref{def:weak_sol_MKV} (Definition \ref{def:path_sol_MKV}), by replacing MKV$({\bf b},{\boldsymbol\sigma})$ with D-MKV$(b, \sigma)$. \vspace{-6pt}
\end{definition}
\end{framed}

\begin{framed}
\vspace{-10pt}
\begin{definition}[Weak/pathwise well-posedness for a density-dependent MKV-SDE]\label{def:defined_dens}
The definition of weak (pathwise) well-posedness for D-MKV$(b, \sigma)$ is given as in Definition \ref{def:well_posed}, by replacing MKV$({\bf b},{\boldsymbol\sigma})$, MKV$({\bf b},{\boldsymbol\sigma};\nu)$ and  MKV$({\bf b},{\boldsymbol\sigma};\textcolor{black}{\xi, W, (\Fc_t)_t})$, respectively, with D-MKV$(b , \sigma)$, D-MKV$(b , \sigma;\nu)$ and  D-MKV$(b , \sigma;\textcolor{black}{\xi, W, (\Fc_t)_t})$. \vspace{-6pt}
\end{definition}
\end{framed}


\begin{remark}[!]
Proposition \ref{prop:pathwise_weak}, Lemma \ref{lem:wat_yam_mkv} and Proposition \ref{prop:wea_strong_MKV} do hold, mutatis mutandis, also for D-MKV$(b , \sigma)$.
\end{remark}

We now want to define the non-linear martingale problem associated to the density-dependent MKV-SDE \eqref{eq:mkv_SDE_dens}. Similarly to how we defined solutions to D-MKV$(b , \sigma)$, we would like to define a solution to the density-dependent martingale problem with initial distribution $\nu$ as a solution to NL-MP$({\bf b}^* , {\boldsymbol\sigma}^*, \nu)$. Once more, we need to check that this definition is independent of the specific version of the density selected by ${\bf b}^* , {\boldsymbol\sigma}^*$.

\begin{lemma}
Let $\nu\in\Pc(\Rd)$. A probability measure $\Pb_{\nu}$, on $(C_T,\Bc)$, is a solution to NL-MP$({\bf b}^*,{\boldsymbol\sigma}^*;\nu)$ if and only if it is a solution to NL-MP$(\tilde{\bf b},\tilde{\boldsymbol\sigma};\nu)$ for any $\tilde{\bf b},\tilde{\boldsymbol\sigma}$ as defined in \eqref{eq:def_b_sig_dens_gen}.
\end{lemma}
\begin{proof}
Let $\Pb_{\nu}$ be a solution to NL-MP$({\bf b}^*,{\boldsymbol\sigma}^*;\nu)$. We show that it is also a solution to NL-MP$(\tilde{\bf b},\tilde{\boldsymbol\sigma};\nu)$. The other way around is proved exactly in the same way. By Definition \ref{def:non_lin_MP}, $\Pb_{\nu}$ is a solution to MP$({\bf b}^*_{\boldsymbol\mu}, {\boldsymbol\sigma}^*_{\boldsymbol\mu};0,\nu)$, with 
\begin{equation}\label{eq:cond_non_lin_mp_den}
\boldsymbol\mu_t = \Pb_{\nu}(x_t \in \dd y), \qquad t\geq 0,
\end{equation}
where $(x_t)_{t\in[0,\infty)}$ denotes the canonical process.

Consider now the frozen coefficients ${\bf b}^*_{\boldsymbol\mu} , {\boldsymbol\sigma}^*_{\boldsymbol\mu}$ and $\tilde{\bf b}_{\boldsymbol\mu} , \tilde{\boldsymbol\sigma}_{\boldsymbol\mu}$. By Remark \ref{rem:densities_ver} with ${\boldsymbol\nu} = {\boldsymbol\mu}$, we have
\begin{equation}
\tilde{\bf b}_{\boldsymbol\mu}(t , X_t) = {\bf b}^*_{\boldsymbol\mu} (t , x_t ) , \quad    \tilde{\boldsymbol\sigma}_{\boldsymbol\mu}(t , x_t) = {\boldsymbol\sigma}^*_{\boldsymbol\mu} (t ,x_t ), \qquad \Pb_{\nu}\text{-almost surely},
\end{equation}
for any $t\in[0,\infty)$. This implies that $\Pb_{\nu}$ is also a solution to MP$(\tilde{\bf b}_{\boldsymbol\mu}, \tilde{\boldsymbol\sigma}_{\boldsymbol\mu};0,\nu)$, and thus a solution to NL-MP$(\tilde{\bf b},\tilde{\boldsymbol\sigma};\nu)$.
\end{proof}

The previous lemma justifies the following definition.

\begin{framed}
\vspace{-10pt}
\begin{definition}[density-dependent non-linear martingale problem]\label{def:non_lin_MP}
Let $\nu\in\Pc(\Rd)$. A probability measure $\Pb_{\nu}$, on $(C_T,\Bc)$, is a solution to D-NL-MP$(b,\sigma;\nu)$ if it is a solution to NL-MP$({\bf b}^*,{\boldsymbol\sigma}^*;\nu)$ such that
\begin{equation}
\Pb_{\nu}(x_t \in \dd y)\in AC(\Rd), \qquad t\geq 0 ,
\end{equation}
where $(x_t)_{t\in[0,\infty)}$ denotes the canonical process. \vspace{-6pt}
\end{definition}\end{framed}

We have the following

\begin{proposition}\label{prop:equivalence_nonlinear_mp_weak_sol_den}
Let $\nu\in\Pc(\Rd)$, and let $\Pb_{\nu}$ be a probability measure on $(C_{\infty}, \Bc )$, such that
\begin{equation}
 \Pb_{\nu}(x_t \in \dd y) = v_t(y) dy, 
 \qquad t> 0 ,
\end{equation}
and such that
\begin{equation}\label{eq:hyp_b_s_loc_bound}
\big[(t,x)\mapsto b\big(t,x,v_t( x)\big) v_t(x)\big], \ \big[ (t,x)\mapsto \sigma \sigma^\top \big(t,x,v_t( x)\big) v_t(x)\big] \in L^{1}_{\text{loc}}([0,\infty)\times\Rd).
\end{equation}
Then $\Pb_{\nu}$ is a solution to D-NL-MP$(b,\sigma;\nu)$ if and only if there exists a (weak) solution $X$ to D-MKV$(b,\sigma;\nu)$ such that $[X] =\Pb_{\nu} $.
\end{proposition}
\begin{remark}\label{rem:huy_prop}
Assumption \eqref{eq:hyp_b_s_loc_bound} in Proposition \ref{prop:equivalence_nonlinear_mp_weak_sol_den} is independent of the specific version of the probability density $v_t$, as all the versions are equal Lebesgue almost everywhere on $\Rd$.
\end{remark}
\begin{proof}[Proof of Proposition \ref{prop:equivalence_nonlinear_mp_weak_sol_den}]
Let ${\boldsymbol\mu}:[0,\infty) \to \Pc(\Rd)$ be the measurable flow defined by 
\begin{equation}\label{eq:def_cont_flow_den_bis}
{\boldsymbol\mu}_t : = 
\begin{cases}
 \Pb_{\nu}(x_t \in \dd y),  & t>0,\\
\nu, & t=0 .
\end{cases}
\end{equation}
By 
\eqref{eq:hyp_b_s_loc_bound}, we have
\begin{equation}
[(t,x)\mapsto {\bf b}^*(t,x,{\boldsymbol\mu}_t)], \ [(t,x)\mapsto {\boldsymbol\sigma}^*  ({\boldsymbol\sigma}^*)^\top(t,x,{\boldsymbol\mu}_t)] \in L^{1}_{\text{loc}}([0,\infty)\times\Rd , \overline{\boldsymbol\mu} ),
\end{equation}
where ${\bf b}^*$, ${\boldsymbol\sigma}^*$ are defined as in \eqref{eq:def_b_sig_dens}-\eqref{eq:density_standard}, and with $\overline{\boldsymbol\mu}$ as in Notation \ref{not:measure_dt_dx}. 
Therefore, Proposition \ref{prop:equivalence_nonlinear_mp_weak_sol} yields that $\Pb_{\nu}$ is a solution to NL-MP$({\bf b}^*,{\boldsymbol\sigma}^*;\nu)$ if and only if there exists a (weak) solution $X$ to MKV$({\bf b}^*,{\boldsymbol\sigma}^*;\nu)$ such that $[X] =\Pb_{\nu} $. By Definitions \ref{def:weak_path_sol_den_MKV} and \ref{def:non_lin_MP}, this concludes the proof.
\end{proof}

%
%
The following result is an immediate consequence of Proposition \ref{prop:equivalence_nonlinear_mp_weak_sol_den}.
\begin{corollary}
Assume that 
\begin{equation}
\big[(t,x)\mapsto b(t,x,y)\big], \ \big[(t,x)\mapsto \sigma\sigma^\top (t,x,y)\big] \in L^{\infty}_{\text{loc}}([0,\infty)\times\Rd),
\end{equation}
uniformly w.r.t. $y\in \R$.
Then, D-MKV$(b,\sigma)$ is weakly well-posed if and only if D-NL-MP$(b,\sigma;\nu)$ admits a unique solution for any $\nu\in\Pc(\Rd)$.
\end{corollary}

We now consider the Fokker-Planck Cauchy problem associated to the density-dependent MKV-SDE \eqref{eq:mkv_SDE_dens}, which reads as
\begin{equation}\label{eq:nl_fp_cp_den}
\hspace{-0pt}{\blue\text{[D-NL-FP$(b,\sigma;\nu)$]}}\qquad
\begin{cases}
 \partial_t v_t(y) = \text{div}_y \Big(- b\big(t,y, v_t(y)\big) \,  v_t(y)  +\frac{1}{2}\text{\bf div}_y \big(   \sigma \sigma^\top\big(t,y,v_t(y)\big) \, v_t(y)    \Big), \quad t> 0, \\
\lim_{t\to 0^+} v_t(y) dy = \nu \in \mathcal{M}(\Rd).
 \end{cases}
\end{equation}

\begin{framed}
\vspace{-10pt}
\begin{definition}[Weak solution to density-dependent non-linear Fokker-Planck Cauchy problem]\label{def:sol_nl_fp_den}
Let $\nu\in \mathcal{M}(\Rd)$. A 
function $v:(0,\infty)\to L^{1}(\Rd)$ is a (weak) solution to D-NL-FP$(b,\sigma;\nu)$ if it is a (weak) density solution to FP$(b_v,\sigma_v;\nu)$ (see Definition \ref{def:sol_lin_fp}) with
\begin{equation}
b_v(t,x):= b\big( t, x , v_t(x) \big), \quad \sigma_v(t,x):= \sigma\big( t, x , v_t(x) \big), \qquad (t,x)\in[0,\infty)\times\Rd.
\end{equation}

\vspace{-15pt}
\end{definition}\end{framed} 

Explicitly,  $v:(0,\infty)\to L^{1}(\Rd)$ is a (weak) solution to D-NL-FP$(b,\sigma;\nu)$ if
\begin{align}
\hspace{-6pt} \int_{\Rd} \varphi(y)& v_t(y) \,\dd y    = \int_{\Rd} \varphi(y) \nu (\dd y)\\ 
& + \int_0^{t} \int_{\Rd}  \Big[\big\langle b\big(s , y, v_s(y) \big) , \nabla \varphi(y)\big\rangle + \frac{1}{2} \text{tr}\big( \sigma\sigma^\top\big(s , y,v_s(y) \big) \nabla^2 \varphi(y) \big)   \Big]  v_s(y) \dd y \, \dd s, \quad t\geq 0,\\
\label{eq:def_sol_NL_FP_den}
\end{align}
for any $\varphi\in C^{\infty}_0(\Rd)$. 

\begin{remark}[!]
The definition above is well posed. Indeed, the validity of \eqref{eq:def_sol_NL_FP_den} (including the fact that all the integrals appearing therein are finite) is independent of the specific representative of $v_t \in L^1(\Rd)$. 
\end{remark}


In analogy with Section \ref{sec:nonlin_fp}, we can now establish the usual two-way link between the solutions to the non-linear Fokker-Planck equation \eqref{eq:nl_fp_cp_den} and those to the MKV SDE \eqref{eq:mkv_SDE_dens}. 

\begin{lemma}\label{lem:fokker_superpo_den}
Let $\nu\in\Pc(\Rd)$ and let $v:(0,\infty)\to L^{1}(\Rd)$ be a non-negative function
such that 
\begin{equation}\label{eq:hyp_b_s_loc_bound_bis}
\big[(t,x)\mapsto b\big(t,x,v_t( x)\big) v_t(x)\big], \ \big[ (t,x)\mapsto \sigma \sigma^*\big(t,x,v_t( x)\big) v_t(x) \big] \in L^{1}_{\text{loc}}([0,\infty)\times\Rd).
\end{equation}
Then:
\begin{itemize}
\item[(i)] If $X$ is a solution to MKV$(b,\sigma;\nu)$, with $ [X_t] = v_t (x) \dd x$ for any $t\geq0$, then $v$ is a weak solution to D-NL-FP$(b,\sigma;\nu)$.
\item[(ii)] [Superposition] Under the additional assumption that the functions 
\begin{equation}\label{eq:trevisan_cond_den}
\Big[ t \mapsto \int_{\Rd} \Big( \big|b\big(t,x,v_t(x)\big)\big| +\big\| \sigma \sigma^\top\big(t,x,v_t(x)\big) \big\| \Big)v_t(x) \, \dd x\Big], \ \Big[  t \mapsto  \int_{\Rd}v_t(x) \, \dd x\Big] \  \in L^{\infty}_{\text{loc}}([0,\infty),
\end{equation}
if $v$ is a (weak) solution to D-NL-FP$(b,\sigma;\nu)$, 
then there exists a solution $X$ to D-MKV$(b,\sigma;\nu)$ such that $ [X_t] =v_t(x) \dd x $.
\end{itemize}
\end{lemma}
\begin{proof}
Let the functions
\begin{equation}
{\bf b}^*: [0,\infty)\times \Rd \times \mathcal{M}(\Rd)\to \R^d, \qquad {\boldsymbol\sigma}^*: [0,\infty)\times \Rd \times \mathcal{M}(\Rd)\to \mathcal{M}^{d\times q}
\end{equation}
be defined as in \eqref{eq:def_b_sig_dens}-\eqref{eq:density_standard} (the definition was given for a measure argument in $\Pc(\Rd)$ but it naturally extends to a measure argument in $ \mathcal{M}(\Rd)$), and let ${\boldsymbol\mu}: [0,\infty) \to  \mathcal{M}_{+}(\Rd)$ be the measurable flow defined by 
\begin{equation}\label{eq:def_cont_flow_den}
{\boldsymbol\mu}_t : = 
\begin{cases}
v_t( x) dx,  & t>0,\\
\nu, & t=0 .
\end{cases}
\end{equation}


Conditions \eqref{eq:hyp_b_s_loc_bound_bis} and \eqref{eq:trevisan_cond_den} imply, respectively, \eqref{eq:cond_L1_lo} and \eqref{eq:cond_superpo} for ${\bf b}^*$ and ${\boldsymbol\sigma}^*$.
%
%
Finally, it is straightforward to observe that $v$ is a (weak) solution to D-NL-FP$(b,\sigma;\nu)$ if and only if ${\boldsymbol\mu}$ is a (weak) solution to NL-FP$({\bf b}^*,{\boldsymbol\sigma}^*;\nu)$ (see Definition \ref{def:sol_nl_fp}). Therefore, the statement stems directly from Lemma \ref{lem:fokker_superpo}.

\end{proof}

\begin{example}[Burgers' equation] Consider, for $\eps>0$, the non-linear PDE
\begin{equation}
\partial_t v_t (y)  = \eps\, \partial_{yy} v_t (y) - \textcolor{black}{v}_t (y)\, \partial_y v_t (y),
\end{equation}
or, in the Fokker-Planck form,
\begin{equation}\label{eq:burgers}
\partial_t v_t (y)  = \eps\, \partial_{yy} v_t (y) - \frac{1}{2}  \partial_y v^2_t (y) .
\end{equation}
This famous equation, known as the Burgers equation, arises in many applications, including fluid and gas dynamics. It describes the evolution of a velocity field subject to acceleration induced by viscosity. 
It can be cast into the form \eqref{eq:nl_fp_cp_den} by letting 
\begin{equation}
b(t,y,z) =\frac{z}{2} , \qquad \sigma(t, y , z)\equiv \sqrt{2 \eps} , \qquad (t,y,z)\in [0,\infty)\times \R \times \R.
\end{equation}
Note that, under mild integrability assumptions on the initial data, the Cole-Hopf transformation (\cite{Hopf1950,cole1951quasi}) yields an explicit solution in terms of the solution to the heat equation, i.e. 
\begin{equation}
v_t(y):= -2 \eps \,\partial_y\log\phi_t(y), \qquad t>0,
\end{equation}
solves \eqref{eq:burgers} pointwise on $\R$ whenever $\phi$ is a positive solution to the heat-equation
\begin{equation}
\partial_t\phi(y)=\eps\partial_{yy}\phi_t(y), \qquad t>0.
\end{equation}

Also, in \cite{bossy1996convergence} it was shown that the space derivative of $v$, i.e. ${\boldsymbol\mu}_t = \partial_y v_t(y)$, is a more amenable object than $v$. Indeed, a simple formal computation shows that ${\boldsymbol\mu}$ satisfies the Fokker-Planck equation \eqref{eq:nl_fp_cp} with 
\begin{equation}
 {\bf b}(t,y, \nu) = \int_{-\infty}^y \nu(dz), \qquad {\boldsymbol\sigma}(t,y, \nu)  \equiv \sqrt{2 \eps}.
\end{equation}
Therefore, ${\boldsymbol\mu}$ is associated to a more regular MKV SDE than $v$. Indeed, with respect to the measure variable, the cumulative distribution function enjoys substantially better regularity than the pointwise evaluation of the density. 
\end{example}



\chapter{Density-dependent MKV SDEs with degenerate noise}

In this chapter we analyze a class of density-dependent MKV SDEs with vanishing noise along some directions. Due to the pointwise evaluation of the density, the Lipschitz theory (see \cite{carmona2016lectures}, \cite{carmona2018probabilistic}, \cite{gobet2016monte} among others) does not apply to this class of equations. In order to perform the analysis, one must then rely on the regularizing properties of the corresponding PDEs. Therefore, a suitable H\"ormander condition on the drift is necessary in order to restore the noise in all the directions and define suitable regularizing semigroups. 

\section{Kinetic-type McKean-Vlasov models}\label{sec:kin_mkv}

We fix $d_0,d\in\N$ with $d_0\leq d$, and consider degenerate McKean-Vlasov SDEs of the form
\begin{equation}\label{eq:mkv_degenerate}
\dd X_t = \bigg[  B X_t + \left( 
\begin{matrix}
{\bf b}_0(t,X_t,[X_t]) \\
{\bf 0}_{(d - d_0) \times q}
\end{matrix}
\right)   \bigg] \dd t   +\left( 
\begin{matrix}
 {\boldsymbol\sigma}_0 (t,X_t,[X_t]) \\
{\bf 0}_{(d - d_0) \times q}
\end{matrix}
\right)  \,  \dd W_t    ,
\end{equation}
where the coefficients  
\begin{equation}
{\bf b}_0: [0,\infty)\times \Rd \times \mathcal{P}(\Rd)\to \R^{d_0}, \qquad {\boldsymbol\sigma}_0: [0,\infty)\times \Rd \times \mathcal{P}(\Rd)\to \mathcal{M}^{d_0\times q}
\end{equation}
are measurable functions, and
\begin{equation}\label{eq:def_B0_B1}
B:=  \left( 
\begin{matrix}
B_0 \\
B_1
\end{matrix}
\right) \in \mathcal{M}^{d\times d},\qquad \text{with}\quad B_0 \in \mathcal{M}^{d_0\times d} , \ B_1 \in \mathcal{M}^{(d - d_0)\times d} .
\end{equation}
To be clear, the solution to \eqref{eq:mkv_degenerate} is a process $X
$ with values in $\R^d \cong \R^{d_0} \times   \R^{d - d_0} $, but the diffusion directly acts only on the first $d_0$ components.

\begin{example}[!]\label{ex:kinetic} (Kinetic setting) Let $d = 2 d_0$ and 
\begin{equation}\label{eq:B1_kin}
B_1 =  \left( 
\begin{matrix}
I_{d_0}  &
{\bf 0}_{d_0 \times d_0}
\end{matrix}
\right).
\end{equation}
The MKV-SDE \eqref{eq:mkv_degenerate} reduces to the \emph{kinetic} MKV-SDE
\begin{equation}\label{eq:mkv_degenerate_kin}
\begin{cases}
\dd X^0_t =  \Big( {\bf b}_0(t,X_t,[X_t]) + B_0 X_t \Big)  \dd t   +  {\boldsymbol\sigma}_0 (t,X_t,[X_t])\,   \dd W_t       \\
\dd X^1_t =  X^0_t \dd t \\
\end{cases},
\end{equation}
where the components of the pair $X_t = (X^0_t , X^1_t) \in \R^{d_0} \times \R^{d_0}$ can be understood as the {\em velocity} and {\em position} of a particle subject to random acceleration. If 
\begin{equation}
{\bf b}_0 \equiv 0 , \quad {\boldsymbol\sigma}_0 \equiv \sigma>0,
\end{equation}
and 
\begin{equation}
B_0 = \alpha  \left( 
\begin{matrix}
I_{d_0}  & {\bf 0}_{d_0}
\end{matrix}
\right),
\end{equation}
then \eqref{eq:mkv_degenerate_kin} reduces to the classical Langevin model in \cite{langevin1908theorie}, which serves as a pilot example of more complex kinetic models (see \cite{imbert2021schauderSilvestre}, \cite{imbert2021schauder}). Recently, kinetic McKean-Vlasov models
have attracted the interest of several authors (e.g. \cite{bossy,hao2024singular,veretennikov2023weak,zhang2021second,hao2026second,pascucci2026existence} among others). 
\end{example}

In light of the previous example, we will refer to degenerate MKV SDEs in the form of \eqref{eq:mkv_degenerate} as \emph{kinetic-type MKV-SDEs}.


A particular subclass of \eqref{eq:mkv_degenerate} is given by density-dependent MKV-SDEs of the form
\begin{equation}\label{eq:mkv_degenerate_den}
\hspace{-8pt}{\blue\text{[Kin-D-MKV$(B,b_0,\sigma_0)$]}}\qquad\begin{cases}
\dd X_t = \bigg[  B X_t +\Bigg( 
\begin{matrix}
b_0\big(t,X_t,v_t(X_t)\big)   \\
{\bf 0}_{(d - d_0) \times q}
\end{matrix}
\Bigg)   \bigg] \dd t   +\Bigg( 
\begin{matrix}
\sigma_0 \big(t,X_t,v_t(X_t)\big)  \\
{\bf 0}_{(d - d_0) \times q}
\end{matrix}
\Bigg)  \,  \dd W_t       \\
[X_t] = v_t(x) \dd x
\end{cases},
\end{equation}
with 
\begin{equation}\label{eq:b_0_and_sigma_0}
b_0: [0,\infty)\times \Rd \times \R\to \R^{d_0},\qquad \sigma_0: [0,\infty)\times \Rd \times \R\to \mathcal{M}^{d_0\times q}
\end{equation}
being measurable functions. Equation \eqref{eq:mkv_degenerate_den} can be written in the form of \eqref{eq:mkv_SDE_dens} by setting
\begin{equation}\label{eq:sig_b_deg_den}
b(t,x, z) = B x + \left( 
\begin{matrix}
b_0(t,x , z) \\
{\bf 0}_{(d - d_0) \times 1}
\end{matrix}
\right), \qquad 
\sigma(t,x, z)=  \left( 
\begin{matrix}
\sigma_0(t,x , z)  \\
{\bf 0}_{(d - d_0) \times q}
\end{matrix}
\right).
\end{equation}
Thus we give the following
\begin{framed}
\vspace{-10pt}
\begin{definition}\label{def:sol_kinetic_density_MKV}
We will call 
weak (or pathwise) solution to Kin-D-MKV$(B,b_0,\sigma_0;\dots)$ any weak (or pathwise) solution to D-MKV$(b,\sigma;\dots)$, according to Definition \ref{def:weak_path_sol_den_MKV}, with $b$, $\sigma$ as in \eqref{eq:sig_b_deg_den}. Analogously, we will say that Kin-D-MKV$(B,b_0,\sigma_0)$ is weakly (or pathwise) well-posed if D-MKV$(b,\sigma)$ is so, according to Definition \ref{def:defined_dens}.\vspace{-8pt}
\end{definition}
\end{framed}

Recall from Section \ref{sec:den_dep_mkv} that \eqref{eq:sig_b_deg_den} can be understood as an instance of \eqref{eq:mkv_degenerate}, and thus of \eqref{eq:mkv_SDE}, with irregular coefficients, by letting 
\begin{equation}\label{eq:def_b_sig_dens_gen_kol}
\big[{\bf b}_0(t, x  , \nu) , {\boldsymbol\sigma}_0(t, x  , \nu)  \big] : = \begin{cases}
 \big[ b_0\big(t,x, v_{\nu} (x)\big), \sigma_0\big(t,x,  v_{\nu} (x)\big) \big], & \text{if } \nu \in AC(\Rd), \\
 0  , & \text{if } \nu \notin AC(\Rd),
\end{cases}
\end{equation}
where $ v_{\nu}$ is a chosen version of the density of $\nu$. 
\begin{remark}[!]
Even if $b$ and $\sigma$ are smooth, the functions ${\bf b}_0$ and ${\boldsymbol\sigma}_0$ in \eqref{eq:def_b_sig_dens_gen_kol} are discontinuous in the measure argument with respect to the weak topology, even when restricted to $AC(\Rd)$.
\end{remark}

\begin{framed}
\noindent {\bf(!)} When the coefficients ${\bf b}_0$, ${\boldsymbol\sigma}_0$ lack the regularity requested in the Lipschitz framework, it is necessary to rely on the regularizing properties of the underlying PDEs, which typically require some non-degeneracy conditions on the diffusion coefficient. 
\end{framed}

Although the diffusion coefficient in \eqref{eq:mkv_degenerate} (and in \eqref{eq:mkv_degenerate_den})  is fully degenerate, we will work under a H\"ormander-type condition which ensures (actually means) that the semigroup associated to the standard SDE obtained by setting ${\boldsymbol\sigma}_0 \equiv I_{d_0\times d_0}$ (thus $q=d_0$) and ${\bf b}_0\equiv 0$ in \eqref{eq:mkv_degenerate} is smooth. The following assumption is in force throughout this chapter. 

\begin{assumption}[Weak H\"ormander condition]\label{assump:hor}
The vector fields $\partial_{x_1},\dots,\partial_{x_{d_0}}$ and 
\begin{equation}\label{eq:def_Y}
Y = Y_{s,x}: = \partial_s + \langle B x , \nabla_x \rangle, \qquad x \in \Rd,
\end{equation}
satisfy 
\begin{equation}\label{horcon}
\text{rank } \text{Lie}(\partial_{x_1},\dots,\partial_{x_{d_0}},Y)  = d+1.
\end{equation}
\end{assumption}

Condition \eqref{horcon} means that the vector fields $\partial_{x_1},\dots,\partial_{x_{d_0}}, Y$ and their iterated Lie brackets, i.e.
\begin{equation}
[\partial_{x_i}, Y] , \ [\partial_{x_{i_1}},[ \partial_{x_{i_2}}, Y]],\ \dots,   [\partial_{x_{i_1}},[ \partial_{x_{i_2}}, \dots , [ \partial_{x_{i_k}} ,  Y] \dots ]] , \ \dots, \quad \text{with } i, i_1, \dots , i_k , \ldots \in \{ 1, \dots, d_0  \},
\end{equation}
span the entire tangent space $\R^{d+1}$ at every point $x\in\Rd$.

\begin{example}[!]\label{ex:kinet_hor} Consider once more the kinetic setting in Example \ref{ex:kinetic}. The general result in Theorem \ref{th:hor} below shows that the H\"ormander condition \eqref{horcon} is satisfied, because $B_1$ as in \eqref{eq:B1_kin} has full rank. In the particular case when $B_0 = {\bf 0}_{d_0 \times 2 d_0}$,  
the vector field in \eqref{eq:def_Y} reads as
\begin{equation}
Y = \partial_s + \sum_{i=0}^{d_0} x_i \partial_{x_{i+d_0}},
\end{equation}
and it is straightforward to verify that
\begin{equation}
[\partial_{x_i} , Y ] = \partial_{x_{i+d_0}} , \qquad i = 1,\cdots, d_0,
\end{equation}
which yields \eqref{horcon}.
\end{example}

\section{Kinetic-type semigroups}\label{sec:kinet_semigroups}

Consider the kinetic-type diffusion obtained from \eqref{eq:mkv_degenerate} by letting ${\bf b}_0\equiv 0$ and 
\begin{equation}
 {\boldsymbol\sigma}_0 (t,X_t,[X_t]) \equiv \sigma I_{d_0\times d_0}, \qquad \sigma >0.
\end{equation}
In particular, $W$ is a formal $d_0$-dimensional Brownian motion and \eqref{eq:mkv_degenerate} reduces to
\begin{equation}\label{eq:linear_SDE}
\dd X_t = B X_t\, \dd t   + \bigg( 
\begin{matrix}
\sigma I_{d_0\times d_0} \\ 
{\bf 0}_{(d - d_0)  \times d_0}
\end{matrix}\bigg)   \, \dd W_t      ,
\end{equation}
where the matrix $B$ is as in \eqref{eq:def_B0_B1}.

Note that the equation above has linear drift and additive noise. Therefore, it fits the usual Lipschitz framework and pathwise well-posedness is ensured. Denoting by $X^{s,x}=X$ the unique solution starting at time $s$ from $x$, It\^o formula yields 
\begin{equation}
X^{s,x}_t = X_t = e^{(t-s) B } \bigg( x + \sigma \int_s^t  e^{- (r-s) B }\bigg( 
\begin{matrix}
I_{d_0\times d_0} \\
{\bf 0}_{(d - d_0)  \times d_0}
\end{matrix}\bigg) dW_r  \bigg), \qquad t\geq s.
\end{equation}
In particular, 
$X^{s,x}_t$ 
 is normally distributed. Precisely, 
\begin{equation}\label{eq:conditional_law_kol}
X^{s,x}_t \sim \mathcal{N}\big( e^{(t-s)B} x , \mathcal{C}(t-s)\big), \qquad 0\leq s \leq t, \ x\in\Rd,
\end{equation}
where the covariance matrix $\mathcal{C}$ is given by
\begin{equation}\label{eq:cov_C}
\mathcal{C}(r) = \sigma^2 \int_0^r  e^{\tau B } \bigg( 
\begin{matrix}
I_{d_0\times d_0} & {\bf 0}_{d_0 \times (d - d_0) }  \\
{\bf 0}_{(d - d_0)  \times d_0} & {\bf 0}_{(d - d_0)  \times (d - d_0) }
\end{matrix}\bigg)  e^{\tau B^\top } d\tau      .
\end{equation}
Also note that the generator of $X$ is the backward Kolmogorov operator
\begin{equation}\label{eq:kol}
\Kcc_{\sigma}=\Kcc 
:= Y_{s,x} + \sigma^2 \sum_{i = 1}^{d_0} \partial^2_{x_i} =: Y_{s,x} + \sigma^2 \Delta_{x,d_0},\qquad x\in\Rd,
\end{equation}
where $Y$ is the vector field defined in \eqref{eq:def_Y}.

The following result is a particular case of the well-known H\"ormander Theorem (\cite{hoermander}). In general, H\"ormander's result holds by replacing, in the condition \eqref{horcon} and in $\Kcc$, the vector fields $\partial_{x_1},\dots,\partial_{x_{d_0}}$ with arbitrary smooth vector fields.  We also refer to \cite{hairer2011malliavin} for a probabilistic proof. 
\begin{theorem}\label{th:hor}
Assumption \ref{assump:hor} is equivalent to the following conditions:
\begin{itemize}
\item[(i)] the Kolmogorov operator $\Kcc$ in \eqref{eq:kol} is hypoelliptic;
\item[(ii)] the covariance matrix $\mathcal{C}(\tau)$ in \eqref{eq:cov_C} is strictly positive definite for any $\tau > 0$.
\end{itemize}
\end{theorem}



If the (equivalent) conditions in Theorem \ref{th:hor} are satisfied, by \eqref{eq:conditional_law_kol}, the 
law of $X^{s,x}_t$ has a Gaussian density, $p_{\Kcc_{\sigma}}(t-s, x  , \cdot)$, where
\begin{equation}\label{eq:kol_dens}
p_{\Kcc_{\sigma}}(r, x  , y) = p_{\Kcc}(r, x  , y) = \frac{1}{ (2 \pi)^{\frac{d}{2}} \sqrt{ |\mathcal{C}(r)|} } \exp \Big(-\frac{1}{2} \big\langle \mathcal{C}^{-1}(r) \big(y - e^{r B }x\big) ,  y - e^{r B} x \big\rangle \Big), \qquad r>0, \ x,y\in\Rd.
\end{equation}
The transition kernel $p_{\Kcc_{\sigma}}(t-s, x  , y)$ coincides with the fundamental solution of $\Kcc$, in the backward variables $(s,x)$. Namely, the function $(s,x)\mapsto p_{\Kcc}(t-s, x , y) $ solves the \emph{backward Kolmogorov} equation
\begin{equation}
\Kcc u = 0, \quad \text{on}\ (-\infty,t)\times \R^d
\end{equation}
with the terminal condition
\begin{equation}
p_{\Kcc}(t-s, z  , y) \dd y {\longrightarrow} \delta_x(\dd y) \ \text{weakly},\quad \text{as } (s,z)\to (t^-,x).
\end{equation}
Therefore, defining the \emph{backward semigroup} $(\Pcc_r)_{r> 0}$ as 
\begin{equation}\label{eq:def_semigroup_back}
\Pcc^{\Kcc_{\sigma}}_{r} f = \Pcc_{r} f := \int_{\Rd} p_{\Kcc}(r, \cdot , y) \varphi(y) \dd y,
\qquad \varphi\in L^{\infty}(\Rd),
\end{equation}
 the function
 \begin{equation}
\label{eq:classical PF cauchy solution}
v(s,x):= [\Pcc_{t-s} \varphi](x)  - \int_{s}^t  [\Pcc_{\tau - s} g(\tau,\cdot)](x) \dd \tau ,\qquad t>s,\ z \in\Rd,
\end{equation}  
is an (internally) smooth solution to the backward Cauchy problem
\begin{equation}\label{eq:example_degenerate_pdes}
\begin{cases}
\Kcc' v = g, \qquad \text{on $(-\infty, t)\times \Rd$}, \\
v(t,\cdot) = \varphi ,
\end{cases}
\end{equation}
for any $\varphi \in \Cc_b(\Rd)$ and $g
$ sufficiently smooth.

Furthermore, as a function of the forward variables $(t,y)$, the transition kernel $p_{\Kcc}(t-s, x , y)$ is the fundamental solution of the formal adjoint of $\Kcc$, i.e. the Fokker-Planck operator
\begin{equation}
\Kcc' := 
 - \partial_t   - \text{div}_y ( B y\, \cdot)+ \frac{\sigma^2}{2} \sum_{i = 1}^{d_0} \partial^2_{y_i} 
 = - 
 Y_{t,y}  -Tr(B) + \frac{\sigma^2}{2}\Delta_{y,d_0}.
\end{equation}
Namely, the function $(t,y)\mapsto p_{\Kcc}(t-s , x , y)$ solves the Fokker-Planck equation (or \emph{forward Kolmogorov equation})
\begin{equation}
\Kcc' u = 0, \quad \text{on}\ (s,\infty)\times \R^d
\end{equation}
with the initial condition
\begin{equation}
p_{\Kcc}(t-s, x , z) \dd x {\longrightarrow} \delta_y(\dd x) \ \text{weakly},\quad \text{as } (t,z)\to (s^+,y).
\end{equation}
Therefore, defining the \emph{forward semigroup} $(\Pcc'_r)_{r> 0}$ as 
\begin{equation}\label{eq:def_semigroup_for}
\Pcc'^{\Kcc_{\sigma}}_{r} f = \Pcc'_{r} f := \int_{\Rd} p_{\Kcc}(r, x , \cdot ) \varphi(x) \dd x,
\qquad \varphi\in L^{\infty}(\Rd),
\end{equation}
 the function
 \begin{equation}
\label{eq:classical PF cauchy solution}
v(t,y):= [\Pcc'_{t-s} \varphi](y)  - \int_{s}^t  [\Pcc'_{t - \tau } g(\tau,\cdot)](y) \dd \tau ,\qquad t>s,\ y \in\Rd,
\end{equation}  
is an (internally) smooth solution to the forward Cauchy problem
\begin{equation}\label{eq:example_degenerate_pdes_forw}
\begin{cases}
\Kcc' v = g, \qquad \text{on $(s , + \infty)\times \Rd$}, \\
v(s,\cdot) = \varphi ,
\end{cases}
\end{equation}
for any $\varphi \in \Cc_b(\Rd)$ and $g
$ sufficiently smooth.

\begin{remark}(!)
The semigroups $\Pcc_r$ and $\Pcc'_r$ regularize to arbitrary order for any positive time $r$. Precisely, for any $\varphi\in L^{\infty}(\Rd)$, the functions
\begin{equation}
(r, x)\mapsto [\Pcc_r \varphi] (x), \quad (r, y)\mapsto [\Pcc'_r \varphi] (y)
\end{equation}
are smooth on $(0,\infty)\times \Rd$. In particular, the two semigroups regularize both in the time and space variables. Furthermore, they  are both continuous at $t=0$, in the sense that, when $\varphi\in \Cc_b(\Rd)$,
\begin{equation}
\|\Pcc_r \varphi - \varphi \|_{0} + \|\Pcc'_r \varphi - \varphi \|_{0} \longrightarrow 0 ,\qquad \text{as } r\to 0^+.
\end{equation}
\end{remark}

\begin{remark}[Heat semigroup]\label{rem:heat_sem} When $d= d_0$, then the component in $X^1$ in \eqref{eq:linear_SDE} disappears and the process $X=X^0$ is an Ornstein-Uhlenbeck-type process. Furthermore, if $B_0 = {\bf 0}_{d\times d}$, then $X_t = \sigma W_t$ and $p_{\Kcc}(r, x  , y) $ reduces to the Gaussian density 
\begin{equation}
p_{\Kcc}(r, x  , y) = \frac{1}{ (2 \pi \sigma^2 r)^{\frac{d}{2}}  } \exp \Big(- \frac{|x-y|^2}{2 \sigma^2 r } \Big) =: p_{\text{heat}}(r, x  , y)  , \qquad r>0, \ x,y\in\Rd,
\end{equation}
and the backward and forward semigroups coincide with the usual heat semigroup, i.e.
\begin{equation}
\Pcc_r = \Pcc'_r =:  \Pcc^{\text{heat}}_r.
\end{equation}
\end{remark}


\section{Anisotropic and intrinsic spaces
}\label{sec:ani_int}

The regularity properties of the kinetic-type semigroups $\Pcc$ and $\Pcc'$ introduced in Section \ref{sec:kinet_semigroups} must be studied in suitable function spaces that reflect the non-Euclidean geometry induced by the weak H\"ormander condition. When only regularity in space is considered, the correct structure is identified by the so-called \emph{anisotropic spaces}, which are constructed in analogy with the usual H\"older spaces based on an anisotropic non-Euclidean distance that reflects the different scaling properties of the covariance matrix $\mathcal{C}$ in \eqref{eq:cov_C}. The full regularity structure, which also includes the regularity in the time variable, is given by the so-called \emph{intrinsic spaces}. These should be regarded as the kinetic version of the usual parabolic spaces. They are defined by specifying the regularity only along the H\"ormander fields $\partial_{x_1}, \dots, \partial_{x_{d_0}}$ and $Y$ in \eqref{eq:def_Y}, and the regularity along all the other (time-space) directions is obtained by exploiting the weak H\"ormander condition \eqref{horcon}. Critically, they reduce to the aforementioned anisotropic spaces when only regularity in space is concerned. 

In the remainder of this section, and of this Chapter, we will employ the following notation. For $N\in\N$ and $\Omega\subset \R^N$ an open set, we denote by $C^{\nu}(\Omega)$ the usual H\"older-Zygmund space of order $\nu$, with $\nu>0$. In particular, it coincides with the classical H\"older space of order $\nu$, when $\nu\notin\N$. 
Namely, for $n\in\N_0$ and $\alpha\in(0,1)$, $C^{n+\alpha}(\Omega)$ is the space of functions $f:\Rd \to \R$ such that 
\begin{equation}
\| f \|_{\Omega, n+\alpha}  \asymp \sum_{\substack{\eta \in \N^d_0\\  |\eta| \leq n}} \| \partial^{\eta} f \|_{\infty} + \sum_{\substack{\eta \in \N^d_0\\  |\eta| = n}}  \sup_{x,y\in\Omega} \frac{| \partial^{\eta}f(x) -\partial^{\eta}f(y) | }{|x - y|^{\alpha}} <\infty,
\end{equation}
where $\partial^{\eta}$ denotes the partial derivative
\begin{equation}
\partial^\eta := \partial^{\eta_1}_{x_1} \cdots \partial^{\eta_d}_{x_N},
\end{equation}
and with $|\eta|$ denoting the length (with a slight abuse of notation) of the multi-index $\eta\in \N^{N}_0$. 

When $\Omega = \R^N$ and $N$ is clear from the context, to ease the notation we will write $\| f \|_{\nu} $ in place of $\| f \|_{\Omega, \nu} $. 

We also set $C^{0}(\Omega):= \Cc_b(\Omega)$, equipped with $\| \cdot \|_{0} : = \| \cdot \|_{\infty}$.


\subsection{Anisotropic spaces}\label{sec:anisot_spaces}

In the case of the heat semigroup (cf. Remark \ref{rem:heat_sem}), the following global Schauder estimates hold true: for any $\alpha, \gamma\geq 0$, there exists $C=C(\alpha, \gamma)>0$ such that 
\begin{equation}\label{eq:schauder_heat}
\|\Pcc^{\text{heat}}_r \varphi  \|_{\alpha+\gamma}  \leq C r^{-\frac{\alpha}{2}} \| \varphi \|_{\gamma}, \qquad r>0,
\end{equation}
for any $\varphi\in C_{b}(\Rd)$.  
%
These estimates can be understood as follows:
\begin{framed}
\noindent {\bf(!)} A gain of (space) regularity of order $\alpha$ comes at the cost of a (small time) singularity of order $\alpha/2$.  
\end{framed}

These types of estimates are crucial in the study of parabolic PDEs, and they are typically satisfied beyond the H\"older setting, for instance in the case of $\gamma<0$, which corresponds to Besov spaces with negative regularity (see \cite{gubinelli2015paracontrolled}). Furthermore, Schauder-type estimates akin to \eqref{eq:schauder_heat} can be proved for more general settings that include the heat semigroup as a particular case, for instance for semigroups associated to non-degenerate diffusions with H\"older continuous diffusion. 

In light of this, one would like to define a H\"older norm 
such that the Schauder estimate \eqref{eq:schauder_heat} holds for the kinetic-type semigroups $\Pcc$ and $\Pcc'$ defined in \eqref{eq:def_semigroup_back} and \eqref{eq:def_semigroup_for}. As it turns out, the correct H\"older norms are induced by a suitable anisotropic distance on $\Rd$ that plays the same role played by the Euclidean distance in the construction of the standard isotropic H\"older spaces. We start by deriving some heuristics that help to understand the main ideas underlying the construction of these spaces.  

\paragraph{The Kolmogorov case.} Consider the prototypical kinetic model studied by Kolmogorov in the 1930's, given by setting $d_0 = 1$, $d=2$, and 
\begin{equation}\label{eq:B_Kolm}
B_0 = \begin{pmatrix}
0 &0 
  \end{pmatrix}, \quad B_1=\begin{pmatrix}
 1& 0 
  \end{pmatrix} \quad \Longrightarrow \quad B = \begin{pmatrix}
0 &0 \\
1 &0
  \end{pmatrix}.
\end{equation}
In this case, the equivalent conditions of Theorem \ref{th:hor} are satisfied (see Remark \ref{ex:kinet_hor}). In particular, we have
\begin{equation}
e^{tB} = I_{2\times 2} + t B  = \begin{pmatrix}
1 &0 \\
t  &1
  \end{pmatrix}.
\end{equation}
The Gaussian kernel $p_{\Kcc}(r, x  , y)$ in \eqref{eq:kol_dens} then becomes 
\begin{equation}
p_{\Kcc}(r, x  , y) = \frac{\sqrt{12}}{ 2 \pi r^2} \exp \Big(-\frac{1}{2} \big\langle \mathcal{C}^{-1}(r) \big(y_1 - x_1 , y_2 - x_2 - r x_1\big)^\top ,  \big(y_1 - x_1 , y_2 - x_2 - r x_1\big)^\top \big\rangle \Big), \qquad r>0, \ x,y\in\Rd, 
\end{equation}
where the covariance matrix $C(\tau)$ 
and its inverse read
\begin{equation}\label{eq:cov_C_kol}
\mathcal{C}(r) = \sigma^2  \begin{pmatrix}
r &\frac{r^2}{2} \\
\frac{r^2}{2}  & \frac{r^3}{3}
  \end{pmatrix}, \qquad \mathcal{C}^{-1}(r) =  \frac{1}{\sigma^2}  \begin{pmatrix}
4 r^{-1} &- 6 r^{-2} \\
- 6 r^{-2}  & 12 r^3
  \end{pmatrix}  .
\end{equation}
The density of this kinetic model was first derived by Kolmogorov in 1934 (\cite{kolmogoroff1934zufallige}). For this reason, the density above is often referred to as \emph{Kolmogorov density}. Note, in particular, that the two marginals are Gaussian densities with variances that are proportional to $r$ and $r^3$, respectively, namely
\begin{align}
p^{(1)}_{\Kcc}(r, x  , y_1) &: = \int_{\R} p_{\Kcc}(r, x  , y_1, y_2) \dd y_2 = \frac{1}{ \sqrt{2 \pi \sigma^2 r} } \exp \Big(-\frac{(y_1 - x_1)^2}{2 \sigma^2 r}\Big) = p_{\text{heat}}(r, x_1, y_1) , \\
p^{(2)}_{\Kcc}(r, x  , y_2) &: = \int_{\R} p_{\Kcc}(r, x  , y_1, y_2) \dd y_1 = \frac{\sqrt{3}}{ \sqrt{2 \sigma^2 \pi r^3} } \exp \Big(-3\frac{(y_2 - x_2 - r x_1)^2}{2\sigma^2 r^3}\Big)= p_{\text{heat}}(r^3 /3, x_2 +  r x_1 , y_2),
\end{align}
for any $x= (x_1 ,x_2), y=(y_1,y_2) \in \R^2$ and $r>0$. 

Now, if $\varphi \in \Cc_b(\R^2)$ is constant in $x_2$, i.e. $\varphi(x_1, x_2)=\varphi(x_1)$, then the action of $\Pcc$ on $\varphi$ is exactly the one of the heat-semigroup $\Pcc^{\text{heat}}$. Therefore, the estimate \eqref{eq:schauder_heat} holds for $\Pcc_r$ in place of $\Pcc^{\text{heat}}$ and with the usual (isotropic) H\"older norms, w.r.t the non-degenerate space variable $x_1$. On the other hand, if $\varphi \in \Cc_b(\R^2)$ is constant in $x_1$, i.e. $\varphi(x_1, x_2)=\bar\varphi(x_2)$, then we have 
\begin{equation}\label{eq:semigr_kolm}
[\Pcc_r \varphi] (x) = \big[\Pcc^{\text{heat}}_{\frac{r^3}{3}} \bar\varphi\big] (x_2 + r x_1 ), \qquad x=(x_1,x_2)\in\R^2, \ r>0.
\end{equation}
Freezing $x_2$ and looking at the regularity in $x_1$, \eqref{eq:schauder_heat} and \eqref{eq:semigr_kolm} then yield 
\begin{equation}\label{eq:heur}
\| [\Pcc_r \varphi] (\cdot , x_2)\|_{\alpha + \gamma} \leq C r^{\alpha + \gamma} r^{- 3\frac{\alpha+\gamma - \gamma/3}{2} } \| \bar\varphi \|_{\gamma/3} = C r^{-\frac{ \alpha}{2}} \| \bar\varphi \|_{\gamma/3}.
\end{equation}
When one freezes $x_1$ and looks at the regularity in $x_2$,  \eqref{eq:schauder_heat} and \eqref{eq:semigr_kolm} then yield 
\begin{equation}\label{eq:heur_bis}
\| [\Pcc_r \varphi] (x_1 , \cdot)\|_{(\alpha + \gamma)/3} \leq C  r^{- \frac{\alpha}{2} }  \| \bar\varphi \|_{\gamma/3}.
\end{equation}
Therefore, defining the anisotropic H\"older norm as
\begin{equation}\label{eq:anis_holder}
\| f \|_{B,\nu}: = \sup_{x_2 \in \R} \| f(\cdot, x_2) \|_{\nu} + \sup_{x_1 \in \R} \| f(x_1, \cdot) \|_{\nu/3}, \qquad \nu>0,
\end{equation}
by \eqref{eq:heur} and \eqref{eq:heur_bis} we finally obtain, for any $\alpha, \gamma\geq 0$, 
\begin{equation}\label{eq:schauder_kol}
\| \Pcc_r \varphi \|_{B, \alpha + \gamma} \leq C  r^{- \frac{\alpha}{2} }  \| \varphi \|_{B,\gamma},
\end{equation}
which is the Schauder-type estimate we were aiming for. A more precise analysis, at the level of the joint density, allows one to prove \eqref{eq:schauder_kol} for a generic function $\varphi\in \Cc_b(\R^2)$, not necessarily constant in the $x_1$ variable. Critically, the Schauder estimate \eqref{eq:schauder_kol} shows that a gain of (space) regularity of order $\alpha/3$ in the degenerate variable $x_2$ produces a (small time) singularity of order $\alpha /2$.  

Let us examine more closely the anisotropic H\"older norm in \eqref{eq:anis_holder}. If $\alpha\in (0,1]$, then 
\begin{equation}
\| f \|_{B,\nu} \asymp \| f \|_{\infty} + \sup_{x,y\in \R^2} \frac{|f(x) - f(y)|}{|x - y|^{\alpha}_B}, 
\end{equation}
where $| \cdot |_B$ is the anisotropic ``norm"\footnote{Strictly speaking, this is not a norm because homogeneity does not hold.}
\begin{equation}\label{eq:anis_norm}
| x |_B =  | x_1 | + |x_2|^{\frac{1}{3}}, \qquad x = (x_1,x_2) \in \R^2.
\end{equation}
Furthermore, similarly to the isotropic H\"older spaces, one can recover the mixed regularity at any positive order. Precisely, it is known (see \cite[Th. 6.2]{dachkovski2003anisotropic}, \cite{triebel92}) that the norm defined in \eqref{eq:anis_holder} is equivalent to the anisotropic H\"older-Zygmund norm induced by $|\cdot |_B$. In particular, for any $n\in\N_0$ and $\alpha \in (0,1)$,
\begin{align}
\| f \|_{B,n+\alpha} & \asymp \sum_{\substack{i,j\in \N_0\\  i + 3 j \leq n}} \| \partial^i_{x_1}\partial^j_{x_2} f \|_{\infty} + \sum_{\substack{i,j\in \N_0\\  i + 3 j = n }}  \sup_{x,y\in \R^2} \frac{|\partial^i_{x_1}\partial^j_{x_2} f(x) - \partial^i_{x_1}\partial^j_{x_2} f(y)|}{|x - y|^{\alpha}_B} 
\\
& + \sum_{\substack{i,j\in \N_0\\  k\in\{0,1\} \\  i + 3 j = n-(2-k) }} \sup_{x\in \R^2\atop h\in\R} \frac{|\partial^i_{x_1}\partial^j_{x_2} f(x_1, x_2+h) - \partial^i_{x_1}\partial^j_{x_2} f(x)|}{|h|^{\frac{\alpha+2-k}{3}} }. 
\end{align}
In particular, one can check that
\begin{align}
\| f \|_{B,1+\alpha}& \asymp \|  f \|_{\infty}+  \| \partial_{x_1} f \|_{B, \alpha} + \sup_{x\in \R^2\atop h\in\R} \frac{|f(x_1, x_2+h) - f(x)|}{|h|^{\frac{1+\alpha}{3}} },\\
\| f \|_{B,2+\alpha}& \asymp \|  f \|_{\infty}+  \| \partial_{x_1} f \|_{B,1+ \alpha} + \sup_{x\in \R^2\atop h\in\R} \frac{|f(x_1, x_2+h) - f(x)|}{|h|^{\frac{2+\alpha}{3}} },
\end{align}
and, for $n\geq 3$,
\begin{equation}
\| f \|_{B,n+\alpha} \asymp  \|  f \|_{\infty}+   \| \partial_{x_1} f \|_{B,n-1+ \alpha} +  \| \partial_{x_2} f \|_{B,n-3+ \alpha} .
\end{equation}
\begin{remark}
Notice that the first-order partial derivative with respect to the degenerate variable $x_2$ appears for the first time when $n=3$. In general, the $k$-th order partial derivative with respect to $x_2$ appears when $n\geq 3k$.
\end{remark}
Finally, if $\| f \|_{B,n+\alpha}<\infty$, 
then we have the anisotropic Taylor formula 
\begin{equation}\label{eq:tay_kol_anis}
\Big| f(y) -  \sum_{\substack{i,j\in \N_0\\  i + 3 j \leq n }}   \frac{ \partial^i_{x_1}\partial^j_{x_2} f(x)}{i! j!} {(y_1 - x_1)^i (y_2 - x_2)^j}   \Big| \leq C \| f \|_{B,n+\alpha} \, |y-x|_B^{n+\alpha}, \qquad x,y\in\R^2.
\end{equation}
\paragraph{General case.} To define the anisotropic spaces in the general setting, it is convenient to introduce an appropriate stratification of the state space $\Rd$. In \cite{polidoro1994class} it was proved that,  
up to permuting the indices, the matrix $B_1\in \mathcal{M}^{(d-d_0)\times d}$ admits the block-form
\begin{equation}\label{B}
  B_1=\begin{pmatrix}
 \mathcal{B}_1 & \ast &\cdots& \ast & \ast \\
 0 & \mathcal{B}_2 &\cdots& \ast& \ast \\
 \vdots & \vdots &\ddots& \vdots&\vdots \\
 0 & 0 &\cdots& \mathcal{B}_{\rr}& \ast
  \end{pmatrix}
\end{equation}
where the $*$-blocks are arbitrary and $\mathcal{B}_j$ is a $(d_j \times d_{j-1})$-matrix of rank $d_j$ with
\begin{equation}\label{eq:elements_d}
d_{0}\geq d_1\geq\dots\geq d_{\rr}\geq1,\qquad \sum_{i=1}^{\rr} d_i=d-d_0.
\end{equation}
We will thus assume that $B_1$ admits the block-decomposition above. 
Accordingly, we will employ the notation 
\begin{equation}\label{eq:notation_r}
\Rd \ni x=(x^0,\cdots, x^{\rr})\in\R^{d_0} \times \cdots \R^{d_{\rr}},
\end{equation}
to denote the projections of $x\in\Rd$ on each of the spaces $\R^{d_i}$. Clearly, we have
\begin{equation}
x^0_i = x_i , \qquad i= 1, \dots, d_0.
\end{equation}

For any $\nu>0$, we then set $C^{\nu}_B(\Rd)$ as the space of functions $f:\Rd \to \R$ such that 
\begin{equation}\label{eq:anis_holder_gen}
\| f \|_{B,\nu}: = \sum_{i=0}^{\rr} \ \sup_{x^{j} \in \R^{d_j}, j\neq i} \| f(x^0,\, \dots \, ,x^{i-1}, \,\cdot\, ,x^{i+1},\, \dots \,,   x^{\rr} ) \|_{\frac{\nu}{2i +1}} < \infty,
\end{equation}
We also set $C^{0}_B(\Rd):= \Cc_b(\Rd)$, equipped with $\| \cdot \|_{B,0} : = \| \cdot \|_{\infty}$. 
Note that the norm in \eqref{eq:anis_holder_gen} reduces to \eqref{eq:anis_holder} in the Kolmogorov case. 

Similarly to the Kolmogorov case, we can characterize the anisotropic H\"older spaces $C^{\nu}_B(\Rd)$ in terms of a suitable anisotropic distance, namely
\begin{equation}\label{eq:anis_norm_gen}
| x -y |_B = \sum_{i=0}^{\rr}  | x^{i} - y^{i}  |^{\frac{1}{2i+1}} , \qquad x = (x^{0}, \dots, x^{\rr}), y = (y^{0}, \dots, y^{\rr})  \in \Rd,
\end{equation}
which reduces to the one defined by \eqref{eq:anis_norm} when $\rr = 1$ and $d_0 = d_1 = 1$. Precisely, we have the following result (see \cite[Th. 6.2]{dachkovski2003anisotropic}). 
\begin{theorem}
For any $\nu>0$, the norm defined by \eqref{eq:anis_holder_gen} is equivalent to the H\"older-Zygmund norm induced by $|\cdot|_B$. In particular, for any $n\in\N_0$ and $\alpha\in(0,1)$, we have
\begin{align}
\| f \|_{B,n+\alpha} & \asymp \sum_{\substack{\eta\in \N^d_0\\  |\eta  |_{B} \leq n}} \| \partial^\eta f \|_{\infty} +\sum_{\substack{\eta\in \N_0^d\\  |\eta|_B = n  }}  \sup_{x,y\in \R^d} \frac{|\partial^{\eta} f(x) - \partial^{\eta} f(y)|}{|x - y|^{\alpha}_B} 
\\
&\quad +\sum_{i=1}^{\rr }\sum_{\substack{\eta\in \N_0^d\\ k\in\{0,1\} \\  |\eta|_B = n - (2 i - k) }}  \sup_{x,y\in \R^d} \frac{|\partial^{\eta} f(x) - \partial^{\eta} f(x^0 ,\, \dots \,, x^{i-1}, y^{i}, \, \dots \, ,  y^{\rr})|}{|x - (x^0 ,\, \dots \,, x^{h-1}, y^{h}, \, \dots \, ,  y^{\rr})|^{\alpha + 2i - k}_B},
\end{align}
where $|\eta|_B$ is the \emph{anisotropic length} of $\eta = (\eta^0, \dots, \eta^{\rr}) \in \N^d_0$, defined as
\begin{equation}\label{eq:anis_length}
|\eta|_B := \sum_{i=0}^{\rr} (2i + 1) |\eta^{i}|.
\end{equation}
\end{theorem}
\begin{remark}
It can be checked, by induction, that an equivalent, recursive, characterization of the anisotropic H\"older spaces is given by:
\begin{align}
\hspace{-10pt}\| f \|_{B,\alpha}& \asymp \| f \|_{\infty} + \sup_{x,y\in \Rd} \frac{|f(x) - f(y)|}{|x - y|^{\alpha}_B}, \\
\hspace{-10pt}\| f \|_{B,n+\alpha}& \asymp \| f \|_{\infty} + \sum_{\substack{\eta \in\N_0 \\ |\eta|_B\leq n }}  \| \partial^{\eta} f \|_{B,n-|\eta|_B+\alpha}  + \sup_{x,y\in \Rd} \frac{|f(x) - f(x^0 ,\, \dots \,, x^{i-1}, y^{i}, \, \dots \, ,  y^{\rr})|}{|x - (x^0 ,\, \dots \,, x^{i-1}, y^{i}, \, \dots \, ,  y^{\rr})|^{n+\alpha}_B}, \quad n=1,\dots , 2 \rr, \\
\hspace{-10pt}\| f \|_{B,n+\alpha}& \asymp \| f \|_{\infty} + \sum_{\substack{\eta \in\N_0 \\ |\eta|_B\leq n }}  \| \partial^{\eta} f \|_{B,n-|\eta|_B+\alpha}, \quad n\geq 2 \rr +1,
\end{align}
for any $\alpha\in(0,1]$.
\end{remark}

\begin{remark}
Let $n\in\N_0$ and $\alpha\in(0,1)$. If $f\in C^{n+ \alpha}_B(\Rd)$, then the following Taylor formula holds:
\begin{equation}\label{eq:tay_anis}
\Big| f(y) -  \sum_{\substack{\eta\in \N^d_0\\  |\eta|_B \leq n }}   \frac{ \partial^{\eta} f(x)}{\eta!} {(y-x)^{\eta}}   \Big| \leq C \| f \|_{B,n+\alpha} \, |y-x|_B^{n+\alpha}, \qquad x,y\in\Rd,
\end{equation}
for a positive constant $C=C(n)$. Here, 
\begin{equation}
\eta! :=  \eta_1 ! \cdots \eta_d !,  \qquad x^{\eta} =  x^{\eta_1}_1  \cdots  x^{\eta_d}_d, \qquad \eta
\in\N_0^d, \  x\in\Rd.
\end{equation}
The formula above can be obtained as a particular case of the Taylor formula proved in \cite{pagliarani2016intrinsic} and \cite{pagliarani2022intrinsic} for the intrinsic spaces, which generalize the anisotropic ones (see Section \ref{sec:intrinsic_spaces} below). The reader may prove  \eqref{eq:tay_anis} directly as an exercise.
\end{remark}

The following Schauder estimates, which were proved in \cite[Th. 2.10]{issoglio2026degenerate} within the general context of (possibly negative) Besov-H\"older spaces, extend the one stated in \eqref{eq:schauder_kol} in the Kolmogorov case.
\begin{theorem}\label{th:Schauder_est}
For any $\alpha,\gamma\geq 0 $ and $T>0$, there exists $C=C(\alpha,\gamma,T,B)>0$ such that 
\begin{equation}\label{eq:Schauder_est}
\|\Pcc_r \varphi\|_{B,\gamma+\alpha} + \|\Pcc'_r \varphi\|_{B,\gamma+\alpha} \leq C\, r^{-\frac{\alpha}{2}} \|\varphi\|_{B,\gamma}, \qquad    r\in (0,T], 
\end{equation}
for any $\varphi\in C^{b}(\Rd)$.
\end{theorem}
Recalling the definition of anisotropic H\"older norm in \eqref{eq:anis_holder_gen}, and its subsequent characterization, we can understand the estimates above as follows:  
\begin{framed}
\noindent {\bf(!)} A gain of (space) regularity of order $\alpha/(2i +1)$ in the variables $x^{i}$ comes at the cost of a (small time) singularity of order $\alpha /2$.  
\end{framed}


\subsection{Intrinsic spaces}\label{sec:intrinsic_spaces}

We start once more by an analogy with the heat semigroup. When mixed time-space regularity is concerned, the Schauder estimate \eqref{eq:schauder_heat} becomes as follows: for any $T>0$ and $\alpha, \gamma\geq 0$, there exists $C=C(\alpha, \gamma)>0$ such that 
\begin{equation}\label{eq:schauder_heat_time}
\|\Pcc^{\text{heat}}_\cdot \varphi  \|_{\Cc^{\alpha+\gamma}((r,T)\times\Rd)}  \leq C r^{-\frac{\alpha}{2}} \| \varphi \|_{\gamma}, \qquad r\in (0,T),
\end{equation}
for any $\varphi\in C_{b}(\Rd)$, where the norm on the left-hand side is the \emph{parabolic H\"older norm} defined, for a function $f:(r,T)\times \Rd\to \R$, as
\begin{equation}\label{eq:parab_norm_mix}
\| f  \|_{\Cc^{\nu}((r,T)\times\Rd)} := \sum_{\substack{k\in\N_0\\ 2 k < \nu}} \sup_{t\in (r,T)} \| \partial^{k}_{t} f(t , \cdot) \|_{\nu-2k} + \sum_{\substack{\eta \in\N^d_0\\ |\eta| < \nu}}  \sup_{y\in\Rd} \| \partial_y^{\eta} f(\cdot, y) \|_{(r,T),(\nu-|\eta|)/2},
\end{equation} 
when $\nu>0$, and as usual, as $\| f  \|_{\Cc^{0}((r,T)\times\Rd)} : = \| f \|_{\infty} $. 
This norm introduces the usual parabolic scaling, which can be read within the Schauder estimate \eqref{eq:schauder_heat_time} as follows:

\begin{framed}
\noindent {\bf(!)} A gain of space and time regularity of order $\alpha$ and $\alpha/2$, respectively, comes at the cost of a (small time) singularity of order $\alpha/2$.  
\end{framed}

%

The Schauder estimate \eqref{eq:schauder_heat_time} reduces to \eqref{eq:schauder_heat} when only space regularity is concerned, as the parabolic H\"older norm reduces to the standard H\"older norm when the function is time independent. Likewise, we would like to define an intrinsic norm, which extends the anisotropic norm introduced in Section \ref{sec:anisot_spaces} to functions of both time and space, allowing to prove estimates akin to \eqref{eq:schauder_heat_time} for the kinetic-type semigroups $\Pcc$ and $\Pcc'$ defined in \eqref{eq:def_semigroup_back} and \eqref{eq:def_semigroup_for}. 

In analogy with the parabolic case, one would expect the intrinsic H\"older norm to be induced by a distance on $\Rdd$ that extends the anisotropic distance, respecting the parabolic scaling. For instance, when $\nu\in (0,1)$, definition \eqref{eq:parab_norm_mix} yields
\begin{equation}\label{eq:hold_par_zero}
|f(s,x)-f(t,y)| \lesssim \| f  \|_{(r,T),\alpha} \big( |t-s|^{1/2} + |y-x| \big)^{\nu}, \qquad (s,x),(t,y)\in (r,T)\times\Rd.
\end{equation}
One could expect that the intrinsic version of \eqref{eq:hold_par_zero} is given by replacing $|\cdot|$ with $|\cdot|_B$ on the right-hand side. However, this is not entirely correct, as the correct distance should be defined in terms of the translation 
\begin{equation}
( t - s   , y - e^{(t - s) B}x  ),
\end{equation}
as opposed to the Euclidean translation. 

Again, we first derive some heuristics in the Kolmogorov case, and then present the general construction.


\paragraph{The Kolmogorov case.} Consider again the particular case given by setting $d_0 = 1$, $d=2$, and $B$ as in \eqref{eq:B_Kolm}. We also assume $\sigma=1$ to ease notation. Now fix $\varphi\in \Cc_b(\R^2)$ and set 
\begin{equation}\label{eq:v_kol}
v(t,y):= [\Pcc'_t \varphi](y), \qquad (t,y)\in (0,\infty)\times\R^2.
\end{equation}
Recalling \eqref{eq:example_degenerate_pdes_forw}, we have 
\begin{equation}
 Y v = \partial^2_{y_1} v,  \qquad \text{on $(0 , + \infty)\times \R^2$},
\end{equation}
with
\begin{equation}
Y = \partial_t + y_1 \partial_{y_2}.
\end{equation}
By 
\begin{equation}
Y \partial_{y_1} =  \partial_{y_1} Y - \partial_{y_2},
\end{equation}
we also obtain
\begin{equation}
Y^2 v = Y \partial^2_{y_1} v =  \big(   \partial_{y_1} Y - \partial_{y_2}  \big) \partial_{y_1} v =    (  \partial^2_{y_1} Y  - 2 \partial_{y_2}\partial_{y_1}  \big)  v =  (  \partial^4_{y_1}  - 2 \partial_{y_2}\partial_{y_1}  \big)  v , \qquad \text{on $(0 , + \infty)\times \R^2$}.
\end{equation}
Iterating, one obtains, for any $k\in\N_0$,
\begin{equation}
Y^{k} v = \sum_{j\in\N_0, 3 j \leq 2 k}  c_{k,j}   \partial^j_{y_2}\partial^{2k-3j}_{y_1} v, \qquad \text{on $(0 , + \infty)\times \R^2$}.
\end{equation}
Let us now fix $\alpha, \gamma\geq0$. For any $k\in\N_0$ with $2k< \alpha+ \gamma$, we have 
\begin{equation}\label{eq:bound_Yv}
 \| Y^{k} v(t , \cdot) \|_{B,\gamma+\alpha-2k} \lesssim \sum_{j\in\N_0, 3 j \leq 2 k} \|  \partial^j_{y_2}\partial^{2k-3j}_{y_1} v(t,\cdot) \|_{B,\gamma+\alpha-2k} \lesssim \|   v(t,\cdot) \|_{B,\gamma+\alpha} \lesssim  r^{-\frac{\alpha}{2}} \|\varphi  \|_{B,\gamma} , \qquad t\in (r,T),
\end{equation}
where the last inequality follows from \eqref{eq:Schauder_est}. 
Similarly, a more precise analysis, at the level of the transition density, allows one to control the regularity of the space derivatives of $v$ along the integral curves of the vector field $Y$, i.e.
\begin{equation}
e^{\delta Y} (t,y):= (t + \delta , e^{\delta B} y ) = (t+\delta , y_1 , y_2 + \delta y_1), \qquad (t,y)\in \R\times\R^2, \quad \delta\in\R.
\end{equation}
Precisely, one can prove
\begin{equation}\label{eq:bound_space_int}
\| \partial^{i}_1\partial^{k}_2 v( \cdot, y_1 , y_2 + \cdot \,y_1) 
\|_{(r,T),(\gamma+\alpha-|\eta|)/2} \lesssim r^{-\frac{\alpha}{2}} \|\varphi  \|_{B,\gamma} , \qquad y\in \R^2,
\end{equation}
for any $i,k \in\N_0$ with $i+2k<\gamma+\alpha$, where $\partial^{n}_j v$ denotes the $n$-th order partial derivative with respect to the $j$-th space variable.

By combining \eqref{eq:bound_Yv} with \eqref{eq:bound_space_int}, we obtain the intrinsic version of \eqref{eq:schauder_heat_time}, namely 
\begin{equation}\label{eq:schauder_intr_kol}
\| \Pcc'_\cdot \varphi  \|_{\Cc_B^{\alpha+\gamma}((r,T)\times\R^2)}  \lesssim  r^{-\frac{\alpha}{2}} \| \varphi \|_{B,\gamma}, \qquad r\in (0,T],
\end{equation}
where, for $\nu>0$, we set
\begin{equation}\label{eq:intr_norm_kol}
\|f  \|_{\Cc_B^{\nu}((r,T)\times\R^2)} := \sum_{\substack{k\in\N_0\\ 2 k < \nu}} \sup_{t\in (r,T)}
 \| Y^{k} f(t,\cdot) \|_{B,\nu-2k} +  \sup_{y\in \R^2} \sum_{\substack{i,j \in\N_0\\ i + 3 j < \nu}}   \| \partial^{i}_1 \partial^{j}_2 f(\cdot, y_1 , y_2 + \cdot \,y_1) \|_{(r,T),(\nu - i - 3j)/2}.\quad
\end{equation}


The norm defined above is called the \emph{intrinsic H\"older norm} and can be regarded as a generalization of the parabolic H\"older norm in \eqref{eq:parab_norm_mix}, where the field $\partial_t$ is replaced by $Y$ and the space regularity is anisotropic as opposed to Euclidean (see Table \ref{tab:parab_int}). Also, it reduces to the anisotropic norm in \eqref{eq:anis_holder} when $f$ is constant in time.
\begin{table}[h]
\begin{center}
\begin{tabular}{c|c|c}
 &parabolic & intrinsic \\
   \hline
   ``time" direction & $\partial_t$ & $Y$ \\
   space regularity & $\| \cdot \|_{\nu}$ (Euclidean) & $\| \cdot \|_{B,\nu}$ (anisotropic) 
    \end{tabular}
    \end{center}
  \caption{Parabolic vs. intrinsic regularity.}
      \label{tab:parab_int}
\end{table} 
\begin{remark}[!]
Although the function $v$ in \eqref{eq:v_kol} is smooth, and thus $Y^k v(t,y)$ is well defined as a combination of Euclidean space derivatives, 
the term $Y^k v(t,y)$ in \eqref{eq:intr_norm_kol} should be understood in a weaker sense, namely as $k$-th order $Y$-Lie derivative. Recall that the $Y$-Lie derivative of $f$ at $(t,y)$ is defined as 
\begin{equation}\label{eq:Lie_der}
Y f (t,y) : =  \lim_{\delta \to 0} \frac{f\big(e^{\delta Y}(t,y)\big) - f(t,y)}{\delta} =  \lim_{\delta \to 0} \frac{f(t+\delta , y_1 , y_2 + \delta y_1) -f(t,y)}{\delta}.
\end{equation}
For instance, if $\|f  \|_{\Cc_B^{\nu}((r,T)\times\R^2)} < \infty$ with  $\nu \in (2,3]$, then $f$ supports a first order $Y$-Lie derivative but not a first order $\partial_{2}$ partial derivative, which is only supported when $\nu>3$.
\end{remark}
We now look for some useful characterizations of the intrinsic H\"older norm. The triangle inequality gives
\begin{equation}
|f(t,y_1 , y_2) - f(s,x_1 , x_2)|  \leq |  f(t,y_1 , y_2) -  f(t,x_1 , y_2 - (t-s) x_1 ) | + |  f(t,x_1 , y_2 - (t-s) x_1 ) - f(s,x_1 , x_2)|.
\end{equation}
Thus, if $\nu\in (0,1)$, \eqref{eq:intr_norm_kol} directly yields
\begin{equation}
|f(t,y_1 , y_2) - f(s,x_1 , x_2)| \lesssim \|f  \|_{\Cc_B^{\nu}((r,T)\times\R^2)} \big( | (y_1 - x_1 , y_2 - x_2 - (t-s) x_1) |_B + |t-s|^{\frac{1}{2}} \big)^{\nu}, 
\end{equation}
which is the intrinsic version of \eqref{eq:hold_par_zero}. More generally, proving the following intrinsic Taylor formula is a straightforward, though lengthy, exercise. If $n\in\N_0$, $\alpha\in(0,1)$ and $\|f  \|_{\Cc_B^{n+\alpha}((r,T)\times\R^2)} < \infty$, then
\begin{align}
\Big| f(t,y) -  \sum_{\substack{i,j,k\in \N_0\\  i   +3j + 2k \leq n }}&   \frac{Y^{k} \partial^{i}_{x_1}\partial^{j}_{x_2}f (s,x_1,x_2) }{k!\, i!\, j!} {(y_1 - x_1)^i \big(y_2 - x_2 - (t-s)x_1)\big)^j (t-s)^k}   \Big|\\
& \lesssim  \| f \|_{\Cc_B^{n+\alpha}((r,T)\times\R^2)} \,\big( | (y_1 - x_1 , y_2 - x_2 - (t-s) x_1) |_B + |t-s|^{\frac{1}{2}} \big)^{n+\alpha}, \label{eq:tay_kol_int}
\end{align}
for any $(t,y),(s,x)\in(r,T)\times\R^2$. This formula obviously reduces to \eqref{eq:tay_kol_anis} when $f$ is independent of time.

Remarkably, the H\"ormander condition allows one to control all the terms in the intrinsic H\"older norm \eqref{eq:intr_norm_kol} 
by only controlling the H\"older norms along the H\"ormander fields $\partial_{y_1}$ and $Y$. We provide the full argument when the regularity index is smaller than one, namely when no derivatives are involved. Fix $z = (s,x) \in \R \times \Rd$, $h,k \in \R$, and set 
\begin{align}
z_0 &= z, \\
z_1 &:= z + (0 , h , 0) = (s, x_1 + h ,  x_2 ),\\
  z_2&: = e^{k Y} z_1 = \big(s + k , x_1 + h , x_2 + k (x_1 + h) \big) , \\ 
  z_3& := z_2 - (0, h , 0) = \big(s + k , x_1 , x_2 + k (x_1 + h) \big) , \\ 
   z_4& := e^{- k Y}z_3 = \big(s  , x_1 , x_2 + k (x_1 + h) - k x_1 \big) = (s, x_1,  x_2 + h k  ). 
\end{align}
The points $z_i$ form the following path: move from $z$ along the field $\partial_{y_1}$, with increment $h$, then along the field $Y$, with increment $k$, and finally again along the same directions, with opposite increments. By telescoping on a function $f$, this path provides a discretization of the commutator $[\partial_{y_1}, Y]f$:
\begin{equation}
f(s, x_1 , x_2 + h k ) - f(s, x)  = \sum_{i=1}^4  f(z_i) - f(z_{i-1})  ,
\end{equation}
which readily yields
\begin{equation}\label{eq:telesc_hor}
|f(s, x_1 , x_2 + h k ) - f(s, x) | 
\lesssim h^{p}  \sup_{\substack{t\in (r,T)\\ y_2\in\R}} \| f(t, \cdot , y_2) \|_{p} + k^{q}  \sup_{y\in\R^2} \| f(\cdot, y_1 , y_2 + \cdot \,y_1)\|_{(r,T),q} .
\end{equation}
Fix now $\nu\in (0,1)$ and 
set $p = \nu$, $q = \nu/2$ and $k = h^2$ in \eqref{eq:telesc_hor}. We obtain
\begin{equation}
|f(s, x_1 , x_2 + h k) - f(s, x) | \lesssim (h k)^{\frac{\nu}{3}}  \bigg(    \sup_{\substack{t\in (r,T)\\ y_2\in\R}} \| f(t, \cdot , y_2) \|_{\nu} + \sup_{y\in\R^2} \| f(\cdot, y_1 , y_2 + \cdot \,y_1)\|_{(r,T),\nu / 2}   \bigg),
\end{equation}
or, equivalently,
\begin{equation}
  \sup_{\substack{t\in (r,T)\\ y_1\in\R}} \| f(t , y_1, \cdot) \|_{\nu/3} \lesssim   \sup_{\substack{t\in (r,T)\\ y_2\in\R}} \| f(t, \cdot , y_2) \|_{\nu} + \sup_{ y\in\R^2} \| f(\cdot, y_1 , y_2 + \cdot \,y_1)\|_{(r,T),\nu / 2}.
\end{equation}
Therefore, the  $(\nu/3)$-H\"older norm in the degenerate space variable is controlled by the $\nu$-H\"older norm in the non-degenerate space variable and by the $(\nu/2)$-H\"older norm along the vector field $Y$.
Recalling definitions \eqref{eq:intr_norm_kol} and \eqref{eq:anis_holder}, we conclude that 
\begin{equation}\label{eq:intr_charact_zero_kol}
\|f  \|_{\Cc_B^{\nu}((r,T)\times\R^2)}  \asymp  \sup_{\substack{t\in (r,T)\\ y_2\in\R}} \| f(t, \cdot , y_2) \|_{\nu} + \sup_{ y\in\R^2} \| f(\cdot, y_1 , y_2 + \cdot \,y_1)\|_{(r,T),\nu / 2}, \qquad \nu\in (0,1). 
\end{equation}
In general, 
it was proved in \cite{pagliarani2016intrinsic,pagliarani2022intrinsic} that 
\begin{align}
\|f  \|_{\Cc_B^{\nu}((r,T)\times\R^2)}  &\asymp  \|\partial_1 f  \|_{\Cc_B^{\nu -1 }((r,T)\times\R^2)} + \sup_{y\in\R^2} \| f(\cdot, y_1 , y_2 + \cdot \,y_1)\|_{(r,T),\nu / 2} , && \nu \in (1,2), \\
\|f  \|_{\Cc_B^{\nu}((r,T)\times\R^2)}  &\asymp \|  f \|_{\infty}+   \|\partial_1 f  \|_{\Cc_B^{\nu -1 }((r,T)\times\R^2)} +  \|Y f  \|_{\Cc_B^{\nu -2 }((r,T)\times\R^2)} , && \nu > 2, \nu\notin\N.
\end{align}
These, together with \eqref{eq:intr_charact_zero_kol}, show that the intrinsic H\"older norm is controlled by the H\"older norms in the non-degenerate space variable and along the vector field $Y$.

\paragraph{The general case.} Recalling the notation 
\eqref{eq:anis_length}, for any $\nu>0$ and $r<T$ we set $C^{\nu}_B((r,T)\times\Rd)$ as the space of functions $f:(r,T)\times\Rd \to \R$ such that 
\begin{equation}\label{eq:intr_holder_gen}
\| f \|_{C^{\nu}_B((r,T)\times\Rd)}: = \sum_{\substack{k\in\N_0\\ 2 k < \nu}} \sup_{t\in (r,T)}
 \| Y^{k} f(t,\cdot) \|_{B,\nu-2k} +\sum_{\substack{\eta \in\N^d_0\\ |\eta|_B < \nu}}  \sup_{y\in\Rd}   \| \partial^{\eta} f(\cdot, e^{\cdot\, B } y) \|_{(r,T),(\nu-|\eta|_B)/2} < \infty,\qquad
\end{equation}
where $Y^k$ denotes the $k$-th order Y-Lie derivative along the vector field $Y$ (see 
\eqref{eq:Lie_der}), and $ \partial^{\eta}$ denotes the partial derivative $\partial^{\eta_1}_{1} \cdots \partial^{\eta_d}_{d}$ with respect to the space variables. 
We also set $C^{0}_B((r,T)\times\Rd):= \Cc_b((r,T)\times\Rd)$, equipped with $\| \cdot \|_{C^{0}_B((r,T)\times\Rd)} : = \| \cdot \|_{\infty}$. 
Note that the norm in \eqref{eq:intr_holder_gen} reduces to \eqref{eq:intr_norm_kol} in the Kolmogorov case. 
\begin{remark}
The intrinsic H\"older space $C^{\nu}_B((r,T)\times\Rd)$ extends the anisotropic H\"older space $C^{\nu}_B(\Rd)$ to functions that depend on both time and space. Indeed, the intrinsic norm \eqref{eq:intr_holder_gen} reduces to the anisotropic norm \eqref{eq:anis_holder_gen} when $f$ is independent of the time variable.
\end{remark}

\begin{remark}
As in the Kolmogorov case, one can prove the following Taylor formula. 
Let $n\in\N_0$ and $\alpha\in(0,1)$. If $f\in C^{n+ \alpha}_B((r,T)\times\Rd)$, there exists a positive constant $C=C(n)$ such that
\begin{equation}\label{eq:tay_int_gen}
\Big| f(t,y) -  \sum_{\substack{k\in\N, \eta\in \N^d_0\\ 2k+ |\eta|_B \leq n }}   \frac{ Y^k \partial_x^{\eta} f(s,x)}{k!\,\eta!} {(t-s)^k\big(y - e^{(t-s)B}x\big)^{\eta}}   \Big| \leq C \| f \|_{C^{n+\alpha}_B((r,T)\times\Rd)} \, \big( |t-s|^{\frac{1}{2}} + |y-e^{(t-s)B}x|_B \big)^{n+\alpha}, 
\end{equation}
for any $(t,y),(s,x)\in (r,T)\times \Rd$. Here, 
\begin{equation}
\eta! :=  \eta_1 ! \cdots \eta_d !,  \qquad x^{\eta} =  x^{\eta_1}_1  \cdots  x^{\eta_d}_d, \qquad \eta\in\N_0^d, \  x\in\Rd.
\end{equation}
The proof of \eqref{eq:tay_int_gen}, with $n=0$, is exactly the same as in the Kolmogorov case, while it is a lengthy exercise for $n\in\N$.  We refer to \cite{pagliarani2016intrinsic} and \cite{pagliarani2022intrinsic} for a detailed proof. 

Note that \eqref{eq:tay_int_gen} reduces to \eqref{eq:tay_kol_int} in the Kolmogorov case. In general, it reduces to \eqref{eq:tay_anis} when $f$ does not depend on time.
\end{remark}

In analogy with what we observed in the Kolmogorov case, it can be proved that the intrinsic H\"older norm is controlled by the H\"older norms in the non-degenerate space variables, i.e. the first $d_0$ space variables, and along the vector field $Y$. The following result was proved in \cite{pagliarani2016intrinsic,pagliarani2022intrinsic}.
\begin{theorem}
Recalling notation \eqref{eq:notation_r}, we have
\begin{align}\label{eq:intr_charact_zero_kol_int}
\|f  \|_{\Cc_B^{\nu}((r,T)\times\Rd)}&  \asymp  \sup_{\substack{t\in (r,T)\\ (y^1, \dots, y^{\rr}) \in\R^{d-d_0}}} \| f(t, \cdot , y^{1}, \dots, y^{\rr}) \|_{\nu} + \sup_{y\in\Rd} \| f(\cdot, e^{\cdot \, B} y)\|_{(r,T),\nu / 2}, && \nu\in (0,1), \\
\|f  \|_{\Cc_B^{\nu}((r,T)\times\Rd)}  &\asymp \sum_{i=1}^{d_0}  \|\partial_i f  \|_{\Cc_B^{\nu -1 }((r,T)\times\Rd)} +\sup_{y\in\Rd} \| f(\cdot, e^{\cdot \, B} y)\|_{(r,T),\nu / 2}, && \nu \in (1,2), \\
\|f  \|_{\Cc_B^{\nu}((r,T)\times\Rd)}  &\asymp \|  f \|_{\infty}+ \sum_{i=1}^{d_0}   \|\partial_i f  \|_{\Cc_B^{\nu -1 }((r,T)\times\Rd)} +  \|Y f  \|_{\Cc_B^{\nu -2 }((r,T)\times\Rd)} , && \nu > 2, \nu\notin\N.
\end{align}
\end{theorem}

We conclude with the Schauder estimates that motivated the introduction of the intrinsic spaces, which were derived heuristically in the Kolmogorov case. 

\begin{theorem}\label{th:Schauder_est_intr}
For any $\alpha,\gamma\geq 0 $ and $T>0$, there exists $C=C(\alpha,\gamma,T,B)>0$ such that 
\begin{align}\label{eq:Schauder_est_int}
 \|\Pcc'_{\cdot}\, \varphi\|_{\Cc_B^{\alpha+\gamma}((r,T)\times\Rd)}  &\leq C\, r^{-\frac{\alpha}{2}} \|\varphi\|_{B,\gamma},\\ 
\|\Pcc_{T-\cdot}\, \varphi\|_{{\Cc_B^{\alpha+\gamma}((0,T-r)\times\Rd)} }& \leq C\, (T-r)^{-\frac{\alpha}{2}} \|\varphi\|_{B,\gamma}, 
\end{align}
for any  $r\in (0,T)$ and for any $\varphi\in C^{b}(\Rd)$.
\end{theorem}
Recalling the definition of intrinsic H\"older norm in \eqref{eq:intr_holder_gen}, it is clear that Theorem \ref{th:Schauder_est_intr} extends Theorem \ref{th:Schauder_est}. The estimates above can be understood as follows:  
\begin{framed}
\noindent {\bf(!)} A gain of space regularity of order $\alpha/(2i +1)$ in the variables $x^{i}$, and of order $\alpha/2$ along $Y$, comes at the cost of a (small time) singularity of order $\alpha /2$.  
\end{framed}

For $\gamma\in (0,2)$ and $\alpha=2$, the estimates in Theorem \ref{th:Schauder_est_intr} are particular instances of those in \cite[Th. 2.7]{lucertini2023optimal}, which are obtained in the case of non-constant diffusion. In general, they can be obtained, following the arguments in \cite{lucertini2023optimal}, as consequences of
the following regularity estimates, at the level of the transition density $p_{\Kcc_{\sigma}}$. 
\begin{lemma}[Gaussian estimates] For any $n,m\in\N_0$ and $\gamma\in(0,1)$, and for any $\sigma,T>0$, there exists a positive constant $C=C(n,m, \gamma, T,\sigma,B)$ such that the upper bounds
\begin{equation}\label{eq:bound_Grad_kol}
\big| \nabla^{n}_{y^0} \nabla^{m}_{x^0} p_{\Kcc_{\sigma}}(r,x,y)  \big|  \leq  \frac{C}{r^{\frac{m+n }{2}}}  p_{\Kcc_{2\sigma}}(r, x  , y) ,
\end{equation}
the regularity bounds in forward variables
\begin{align}
\big| \nabla^{n}_{y^0} \nabla^{m}_{x^0} p_{\Kcc_{\sigma}}(r,x,y) - \nabla^{n}_{y^0}\nabla^{m}_{x^0} p_{\Kcc_{\sigma}}(r,x,y') \big| & \leq C \, \frac{|y - y'|^{\gamma}_B}{r^{\frac{m+n + \gamma}{2}}} \big( p_{\Kcc_{2\sigma}}(r, x  , y) + p_{\Kcc_{2\sigma}}(r, x  , y') \big), \\
\label{eq:bound_gauss_for_grad} \\
\big| \nabla^{n}_{y^0} \nabla^{m}_{x^0}  p_{\Kcc_{\sigma}}(r,x,y) - \nabla^{n}_{y^0} \nabla^{m}_{x^0}  p_{\Kcc_{\sigma}}\big(r',x, e^{(r'-r) B} y\big) \big| & \leq C \bigg(\frac{r'}{r} \bigg)^{Q/2}  \frac{(r'-r)^{\gamma}}{r^{\frac{m}{2}+\gamma}} \, p_{\Kcc_{2\sigma}}(r', x  , y),\\
\label{eq:bound_gauss_for_Y}
\end{align}
and the regularity bounds in backward variables
\begin{align}
\big| \nabla^{n}_{y^0} \nabla^{m}_{x^0} p_{\Kcc_{\sigma}}(r,x,y) - \nabla^{n}_{y^0}\nabla^{m}_{x^0} p_{\Kcc_{\sigma}}(r,x',y) \big| & \leq C \, \frac{|x - x'|^{\gamma}_B}{r^{\frac{m+n + \gamma}{2}}} \big( p_{\Kcc_{2\sigma}}(r, x  , y) + p_{\Kcc_{2\sigma}}(r, x'  , y) \big), \\
\big| \nabla^{n}_{y^0} \nabla^{m}_{x^0}  p_{\Kcc_{\sigma}}(r',x,y) - \nabla^{n}_{y^0} \nabla^{m}_{x^0}  p_{\Kcc_{\sigma}}\big(r,e^{(r'-r) B}x,  y\big) \big| & \leq C \bigg(\frac{r'}{r} \bigg)^{Q/2}  \frac{(r'-r)^{\gamma}}{r^{\frac{m}{2}+\gamma}} \, p_{\Kcc_{2\sigma}}(r', x  , y),
\end{align}
hold for any $r'>r>0$ and $x,x',y,y'\in\Rd$. Here, $p_{\Kcc,\cdot}$ is the Gaussian density in \eqref{eq:kol_dens} and $Q$ is the so-called intrinsic dimension, i.e.
\begin{equation}
Q := \sum_{j = 0}^{\rr} d_j (2j +1).
\end{equation}
\end{lemma}

For a proof of the estimates above we refer to \cite[Lemma 3.4 and Proposition A.8]{lucertini2022optimal} (see also \cite{difpas}), for $m,n=0,1,2$. Although these estimates are well-known, we are not aware of a reference containing a detailed proof for $m,n>2$. 


\section{A result for density-dependent MKV-SDEs}\label{sec:kin_den_result}

We present here a result for the density-dependent kinetic-type MKV-SDE \eqref{eq:mkv_degenerate_den}. We follow the approach recently developed by Issoglio, Russo and their coauthors (see \cite{le2019forward}, \cite{lieber2018well}, \cite{issoglio2023mckean} and \cite{issoglio2026degenerate}, among others), which mainly relies on two ingredients: the superposition principle for the non-linear Fokker-Planck equation, and the characterization of its weak (or distributional) solutions as mild solutions. In particular, the underlying assumptions, and the main results, of this section are very close to those in \cite{lieber2018well}. In addition to studying the well-posedness of MKV-SDE \eqref{eq:mkv_degenerate_den}, we also employ known regularity bounds for the kinetic transition kernels to establish the intrinsic (thus in time and space) regularity of the solutions. 

We also point out that the framework considered here is not the most general one available in the current state of the art on density-dependent MKV SDEs. For instance, recent developments, see \cite{issoglio2023mckean} in the case of non-degenerate diffusion and \cite{issoglio2026degenerate} for the current kinetic framework, allow the drift coefficient $b_0$ to be distributional.

Hereafter, $T>0$ is a fixed time horizon. For $\gamma\geq 0$ and $p\in [1,\infty]$, we also  denote by $L^{\infty}_{T}\Cc^{\nu}_B(\Rd)$ and $L^{\infty}_{T}L^{p}(\Rd)$, respectively, the set of measurable functions $f:(0,T)\times \Rd \to \R$ such that
\begin{align}
\| f \|_{L^{\infty}_{T}\Cc^{\nu}_B(\Rd)} &: =  \sup_{t\in (0,T)} \| f(t, \cdot) \|_{B,\nu} < \infty, \\
\| f \|_{L^{\infty}_{T}L^{p}(\Rd)} &: =  \sup_{t\in (0,T)} \| f(t, \cdot) \|_{L^{p}(\Rd)} < \infty,
\end{align} 

Recall as well the weak H\"ormander condition \eqref{horcon}, which is in force throughout this section.

\begin{assumption}\label{assum:sigma_0}
The function $\sigma_0=\sigma_0(t,x,z)$ in \eqref{eq:b_0_and_sigma_0} is independent of $z$ and belongs to $L^{\infty}_{T}\Cc^{\alpha}_B(\Rd)$ for some $\alpha \in (0,1)$. Furthermore, there exists $\lambda>0$ such that the following coercivity condition holds:
\begin{equation}
\lambda^{-2} |\eta|^2 \leq \big\langle \sigma_0 \sigma_0^\top (t,x) \eta , \eta \big\rangle \leq \lambda^2 |\eta|^2, \qquad   \eta\in\R^{d_0},\ (t,x)\in(0,T)\times\Rd.
\end{equation}
\end{assumption}

\begin{assumption}\label{assum:b_0}
The function $b_0$ in \eqref{eq:b_0_and_sigma_0} is measurable and such that: 
\begin{itemize}
\item[(a)] The function $b_0(t, x, \cdot)$ is Lipschitz continuous on $\R$, uniformly w.r.t. $(t,x)\in(0,T)\times\Rd$.
\item[(b)] [(b-loc)] The function 
\begin{equation}\label{eq:con_b0_bou}
\R\ni y\mapsto \| b_0(\cdot, \cdot, y) \|_{L^\infty((0,t)\times \Rd)}
\end{equation}
is bounded [bounded on compacts]. 
\end{itemize}
\end{assumption}
\begin{example}
It is a simple exercise to show that the function $b_0$ defined by
\begin{equation}
b_0(t, x , y) :=\beta(t,x) F(y), \qquad (t,x,y)\in (0,T)\times\Rd\times\R, 
\end{equation}
for $\beta\in L^{\infty}_{T}L^{\infty}(\Rd)$ and $F$ Lipschitz continuous, 
satisfies Assumption \ref{assum:b_0}-(a),(b-loc). In particular, $F$ can be the identity function and thus the kinetic-Burgers-type equation
\begin{equation}\label{eq:mkv_degenerate_den_bur}
\begin{cases}
\dd X_t = \bigg[  B X_t + \Bigg( 
\begin{matrix}
\beta(t,X_t) v_t(X_t) \\
{\bf 0}_{(d - d_0) \times q}
\end{matrix}
\Bigg)  \bigg] \dd t   +\Bigg( 
\begin{matrix}
 \sigma_0 \big(t,X_t\big) \\
{\bf 0}_{(d - d_0) \times q}
\end{matrix}
\Bigg)  \,  \dd W_t       \\
[X_t] = v_t(x) \dd x
\end{cases},
\end{equation} 
is included in the analysis. If $F$ is also bounded, then $b_0$ satisfies Assumption \ref{assum:b_0}-(a),(b).
\end{example}

The following remark sheds light on the role of Assumption \ref{assum:b_0}-(b), or alternatively Assumption \ref{assum:b_0}(b-loc), within our analysis. It will be recalled several times in the proofs below.

\begin{remark}\label{rem:b0_comp_v}
Let $v:(0,T)\times \Rd \to \R$ be a measurable function. Under 
Assumption  \ref{assum:b_0}-(a), together with 
either Assumption \ref{assum:b_0}-(b), or Assumption \ref{assum:b_0}-(b-loc) with $v\in L^{\infty}_{T}L^{\infty}(\Rd)$, 
we have 
\begin{equation}\label{eq:b0_comp_v}
\|(t,x)\mapsto  b_0\big( t,x , v_t(x) \big) \|_{L^{\infty}((0,T)\times \Rd)} \leq \sup_{|y|\leq \| v \|_{L^{\infty}_{T}L^{\infty}(\Rd)}} \|(t,x)\mapsto  b_0\big( t,x ,y \big) \|_{L^{\infty}((0,T)\times \Rd)}   <\infty. 
\end{equation}
\end{remark}

Under the assumptions above, we can prove that \eqref{eq:mkv_degenerate_den} admits a unique (in law) weak solution and its marginal density is intrinsically H\"older continuous.

\begin{theorem}[Existence and uniqueness of weak solutions]\label{th_den_MKV_kin}
Let Assumptions \ref{assump:hor}, \ref{assum:sigma_0} and \ref{assum:b_0}-(a) be in force. Then:
\begin{itemize}
\item[(i)] Under Assumption \ref{assum:b_0}-(b),  Kin-D-MKV$(B,b_0,\sigma_0)$ is weakly well-posed (see Definition \ref{def:sol_kinetic_density_MKV}).
\item[(ii)] Under Assumption \ref{assum:b_0}-(b-loc), for any $\nu\in\mathcal{P}(\Rd)$ such that $\nu = v_0(x) dx$ with $v_0 \in L^{\infty}(\Rd)\cap L^{1}(\Rd)$,
there exists a unique (in law) weak solution to Kin-D-MKV$(b_0,\sigma_0; \nu)$ (see Definition \ref{def:sol_kinetic_density_MKV})  such that $v\in L^{\infty}_{T}L^{\infty}(\Rd)$, 
where $v$ denotes the flow of time-marginal densities of the solution.
\end{itemize}
\end{theorem}

\begin{theorem}[Regularity of the density]\label{th:reg} Let the assumptions of Theorem \ref{th_den_MKV_kin}-(ii) be in force, and assume, additionally,  the diffusion coefficient $\sigma_0\equiv\sigma>0$ be constant. 
If $v_0 \in C^{\gamma}_B$ with $\gamma\in (0,1)$, then $v \in \Cc_B^{\gamma}((0,T)\times \Rd)$.
\end{theorem}

\begin{remark}[!]
Note that the uniqueness result in 
Theorem \ref{th_den_MKV_kin}-(i) 
 is stronger than the one in 
Theorem \ref{th_den_MKV_kin}-(ii)
, as it is obtained under a stronger boundedness assumption on $b_0$. 
In particular: 
\begin{itemize}
\item under Assumption \ref{assum:b_0}-(b-loc), uniqueness is ensured only within the 
weak solutions whose flow of time marginals are in $L^{\infty}_{T}L^{\infty}(\Rd)$, while existence requires the initial distribution to have a bounded density;
\item under Assumption \ref{assum:b_0}-(b), uniqueness is ensured over all possible weak solutions, and existence is ensured for any initial distribution.
\end{itemize}
Note also that the regularity result in Theorem \ref{th:reg} holds under Assumption \ref{assum:b_0}-(b-loc), and thus under Assumption \ref{assum:b_0}-(b), provided that the initial distribution is H\"older continuous.
\end{remark}

%

The proof of Theorems \ref{th_den_MKV_kin} and \ref{th:reg} strongly relies on the following propositions, whose proof is postponed until Section \ref{sec:proof_prop_fp}.

\begin{proposition}[!]\label{prop:fp_kinetic_Density}
Let Assumptions \ref{assump:hor}, \ref{assum:sigma_0} and \ref{assum:b_0}-(a) be in force. Let also $\sigma$, $b$ be as in \eqref{eq:sig_b_deg_den} and $\nu \in \mathcal{P}(\Rd)$. Then
\begin{itemize}
\item[(i)] Under Assumption \ref{assum:b_0}-(b), 
there exists a (weak) solution $v$ to D-NL-FP$(b,\sigma; \nu )$ (see Definition \ref{def:sol_nl_fp_den}).
\item[(ii)] Under Assumption \ref{assum:b_0}-(b-loc) and assuming $\nu = v_0(x) dx$ with $v_0 \in L^{\infty}(\Rd)\cap L^{1}(\Rd)$, 
there exists a unique (weak) solution $v$ to D-NL-FP$(b,\sigma;\nu )$ (see Definition \ref{def:sol_nl_fp_den}) 
in $L^{\infty}_{T}L^{\infty}(\Rd)$. 
%
 \end{itemize}
 In both cases, $v$ is non-negative and belongs to $L^{\infty}_{T}L^{1}(\Rd)$.
\end{proposition}
\begin{proposition}[!]\label{prop:reg_sol_weak}
 Let the assumptions of Proposition \ref{prop:fp_kinetic_Density}-(ii) be in force, and assume, additionally,  the diffusion coefficient $\sigma_0\equiv\sigma>0$ be constant. 
If $v_0 \in C^{\gamma}_B$ with $\gamma\in (0,1)$, then $v \in \Cc_B^{\gamma}((0,T)\times \Rd)$.
\end{proposition}
%
%

The following lemma is crucial to prove Proposition \ref{prop:fp_kinetic_Density}. The part of the statement dealing with weak uniqueness for the SDE is also crucial in the proof of Theorem \ref{th_den_MKV_kin}.

\begin{lemma}[!]\label{lem:sde_kin_dens}
Let Assumptions \ref{assump:hor} and \ref{assum:sigma_0} be in force. 
Setting
\begin{equation}\label{eq:b_sig_lin}
\beta(t,x) : = B x +  \left( 
\begin{matrix}
\beta_0(t,x)  \\
{\bf 0}_{(d - d_0) \times 1}
\end{matrix}
\right)  , \quad \sigma(t,x) : = \left( 
\begin{matrix}
\sigma_0 (t,x)  \\
{\bf 0}_{(d - d_0) \times 1}
\end{matrix}
\right), \qquad (t,x)\in(0,T)\times\Rd,
\end{equation}
with $\beta_0\in L^{\infty}((0,T)\times \Rd)$, we have:
\begin{itemize}
\item[(i)] SDE$(\beta,\sigma)$ is weakly well-posed (see Definition \ref{def:well_posed_lin}). For any $(s,x)\in [0,T)$, the law of the (unique in law) solution to SDE$(\beta,\sigma;s,\delta_x)$, at time $t\in (s,T)$, has a density $p(s,x; t, \cdot)$.
\item[(ii)] 
There exists $C>0$, depending only on $\| \sigma_0 \|_{L^{\infty}_{T}\Cc^{\alpha}_B(\Rd)}$, $\| \beta_0 \|_{L^{\infty}((0,T)\times \Rd)}$, $\lambda$, $T$ and $B$, such that
\begin{equation}\label{eq:gaussian_1}
 |  \nabla^{m}_{x^0} p(s,x;t,y)  |  \leq C (t-s)^{-\frac{m}{2}} \, p_{\Kcc_{2\lambda}}(t-s, x  , y).
\end{equation}
for any $0\leq s < t < T$, any $x,y\in\Rd$ and $m=0,1,2$. Here, $p_{\Kcc,\cdot}$ is the Gaussian density in \eqref{eq:kol_dens}.
\item[(iii)] Furthermore, for any $s\in[0,T)$ and $\nu\in\mathcal{P}(\Rd)$, the flow of time-marginal densities of the solution to SDE$(\beta,\sigma;s,\nu)$ is the unique (weak) solution to FP$(\beta,\sigma;s, \nu)$ (see Definition \ref{def:sol_lin_fp}). In particular, for any $x\in\Rd$, $(p(s,x; t, \cdot))_{t\in(s,T)}$ is a density solution to FP$(\beta,\sigma;s, \delta_x)$.
\end{itemize}
\end{lemma}
\begin{proof}
Part (i) was proved in  \cite[Th. 1]{chaudru2022regularization}. Part (ii) was proved in \cite[Th. 1.8]{lucertini2022optimal}. Concerning Part (iii), Proposition \ref{th:lin_FP}-(i) yields that the flow of time-marginal densities of the solution to SDE$(\beta,\sigma;0,\nu)$ is a (weak) solution to FP$(\beta,\sigma; 0, \nu)$. To obtain uniqueness for the solutions to FP$(\beta,\sigma; \nu)$, one can use a standard argument based on forward-backward duality, owing to the Schauder estimates in \cite[Th. 2.7]{lucertini2023optimal} for the solution to the backward Kolmogorov Cauchy problem associated to SDE$(\beta,\sigma)$. We leave the details as an exercise for the reader.
\end{proof}
\begin{remark}
Note that, in Lemma \ref{lem:sde_kin_dens}, uniqueness for the solutions to FP$(\beta,\sigma; \nu)$ does not stem immediately from the weak uniqueness of the solutions to SDE$(\beta,\sigma;0,\nu)$, as solutions to FP$(\beta,\sigma; \nu)$ that are not non-negative could exist in principle.
\end{remark}

We are now in the position to prove Theorems \ref{th_den_MKV_kin} and \ref{th:reg}.

\begin{proof}[Proof of Theorem \ref{th_den_MKV_kin}]
Let $v
$ be the non-negative unique (weak) solution to D-NL-FP$(b,\sigma; \nu )$, with $\sigma$ and $b$ defined as in \eqref{eq:sig_b_deg_den}, in Proposition \ref{prop:fp_kinetic_Density}. First note that condition \eqref{eq:hyp_b_s_loc_bound_bis} in Lemma \ref{lem:fokker_superpo_den} holds true in light of the boundedness assumption on $\sigma_0$, together with \eqref{eq:b0_comp_v}. 

\emph{Existence:} It stems from Lemma \ref{lem:fokker_superpo_den}-(ii), once condition \eqref{eq:trevisan_cond_den} is checked. The latter holds true because 
$v\in L^{\infty}_{T}L^{1}(\Rd)$ and because of the boundedness assumption on $\sigma_0$, together with \eqref{eq:b0_comp_v}. 

\emph{Uniqueness:} Let 
$X, \tilde X$ be two solutions to Kin-D-MKV$(b_0,\sigma_0;\nu)$. 
By Lemma \ref{lem:fokker_superpo_den}-(i) and by the uniqueness result of Proposition \ref{prop:fp_kinetic_Density}, their flows of time-marginal densities coincide with $v$. Thus, by definition, both $X$ and $\tilde X$ are weak solutions to the frozen equation SDE$({\bf b}_{v},\sigma; 0, v_0(x) dx)$, with $\sigma$ as in \eqref{eq:sig_b_deg_den} and
\begin{equation}\label{eq:b_0_kin}
{\bf b}_{v}(t,x) = B x + \left( 
\begin{matrix}
b_0\big( t,x , v_t(x) \big) \\
{\bf 0}_{(d - d_0) \times 1}
\end{matrix}
\right), \qquad (t,x)\in (0,T)\times\Rd.
\end{equation}
However, by \eqref{eq:b0_comp_v}, 
Lemma \ref{lem:sde_kin_dens}-(i) yields uniqueness in law for the frozen SDE and thus $[X] = [\tilde X]$.
\end{proof}

\begin{proof}[Proof of Theorem \ref{th:reg}]
In the proof of Theorem \ref{th_den_MKV_kin}, we identified the flow of time-marginals of the solution to Kin-D-MKV$(b_0,\sigma_0; \nu)$ as the unique (weak) solution to D-NL-FP$(b,\sigma;\nu )$ (see Definition \ref{def:sol_nl_fp_den}) 
in $L^{\infty}_{T}L^{\infty}(\Rd)$. Therefore, the statement is a direct consequence of Proposition \ref{prop:reg_sol_weak}.
\end{proof}

\vspace{2pt}

A natural question is whether the assumptions of Theorem \ref{th_den_MKV_kin} are enough to prove an analogous result for the pathwise solutions to the kinetic-type MKV-SDE \eqref{eq:mkv_degenerate_den}. 
In the general kinetic setting, the answer seems to be no (see Remark \ref{rem:path_kin}). However, in the non-degenerate case, which is when $d= d_0$ and thus the MKV-SDE \eqref{eq:mkv_degenerate_den} reduces to 
\begin{equation}\label{eq:mkv_degenerate_non_den}
\begin{cases}
\dd X_t = \Big( b_0\big(t,X_t,v_t(X_t)\big) + B_0 X_t \Big) \dd t   +\sigma_0 \big(t,X_t,v_t(X_t)\big)  \,  \dd W_t     \\
[X_t] = v_t(x) \dd x
\end{cases},
\end{equation}
we have the result below. 
\begin{corollary}\label{cor:pathwise_den}
Let $d=d_0$ and let Assumptions \ref{assump:hor}, \ref{assum:sigma_0} and \ref{assum:b_0}-(a) be in force. Furthermore, assume 
$\sigma_0(t, \cdot)$ is Lipschitz-continuous on $\R^{d_0}$, uniformly with respect to $t\in (0,T)$. 
Then:
\begin{itemize}
\item[(i)] Under Assumption \ref{assum:b_0}-(b),  Kin-D-MKV$(B,b_0,\sigma_0)$ is pathwise well-posed (see Definition \ref{def:sol_kinetic_density_MKV}).
\item[(ii)] Under Assumption \ref{assum:b_0}-(b-loc), for any \textcolor{black}{$\xi\in m\Fc_0$} 
with a density $v_0 \in L^{\infty}(\Rd)\cap L^{1}(\Rd)$, and for any $(W, (\Fc_t)_t )$ $d_0$-dimensional Brownian motion, there exists a (pathwise) unique solution $X$ to D-MKV$(b_0,\sigma_0;\xi, W, (\Fc_t)_t)$ (see Definition \ref{def:weak_path_sol_den_MKV})  
whose flow of time-marginal densities belongs to $L^{\infty}_{T}L^{\infty}(\Rd)$.
\end{itemize}
\end{corollary}
\begin{proof}
Let 
$v$ be the flow of densities in Theorem \ref{th_den_MKV_kin}. 
 By the Lipschitz and non-degeneracy assumption on $\sigma_0$, together with \eqref{eq:b0_comp_v}, weak existence and pathwise uniqueness hold for the standard SDE$({\bf b}_{0,v},0,\sigma_0)$ (see \cite [Theorem 2]{veretennikov1984stochastic})
, with 
\begin{equation}
{\bf b}_{0,v}(t,x) = B_0 x + b_0\big( t,x , v_t(x) \big), \qquad (t,x)\in (0,T)\times\R^{d_0}.
\end{equation} 
The statement then stems from Theorem \ref{th_den_MKV_kin} together with Lemma \ref{lem:wat_yam_mkv}. 
%
\end{proof}

\begin{remark}\label{rem:path_kin}
In the kinetic case, the proof of Corollary \ref{cor:pathwise_den} breaks down because pathwise uniqueness of the kinetic-type frozen SDE$({\bf b}_{v},\sigma_0)$ typically requires stronger regularity assumptions on the drift coefficient (see \cite{de2017strong}). 
For instance, recalling notation \eqref{eq:notation_r}, ${\bf b}_{v}(t,x)$ has to be at least $2/3$-H\"older continuous in the space variables $x^{1}$. Though it is not known whether these regularity thresholds are sharp, as no counter example has been provided so far, they appear to be sharp relative to the Zvonkin-type (PDEs) techniques employed to prove pathwise uniqueness. Proving that ${\bf b}_{v}(t,x)$  satisfies these extra regularity conditions is not a trivial task; it requires not only stronger assumptions on $b_0$ but also to prove a higher regularity result for $v_t$.
\end{remark}

\subsection{Proof of Propositions \ref{prop:fp_kinetic_Density} and \ref{prop:reg_sol_weak}}\label{sec:proof_prop_fp}

Assumptions \ref{assump:hor}, \ref{assum:sigma_0} and \ref{assum:b_0}-(a) are in force throughout this section.
For the reader's convenience, we write explicitly the non-linear Fokker-Planck equation associated to the density-dependent kinetic-type MKV 
\eqref{eq:mkv_degenerate_den}, under the assumption that $\sigma_0(t,x,y) = \sigma_0(t,x)$, i.e.
\begin{equation}\label{eq:nl_fp_cp_den_kin}
\begin{cases}
 \partial_t v_t(y) =  - \text{div}_y \big( B y\, v_t(y)\big) +\frac{1}{2}\text{div}_{y^0}\Big(\text{\bf div}_{y^0} \big(   \sigma_0 \sigma_0^\top(t,y) \, v_t(y) \big)\Big) - \text{div}_{y^0} \Big(b_0\big(t,y, v_t(y)\big) \,  v_t(y) \Big)     , \quad t\in (0,T), \\
\lim_{t\to 0^+} v_t(y) dy = \nu \in \mathcal{P}(\Rd).
 \end{cases}
\end{equation}
Here, $\text{div}$ and $\text{\bf div}$ denote the standard and matrix divergence operators (see \eqref{eq:div_stand} and \eqref{eq:div_matrix}). In particular, the superscript in $y^0$ means that the operator acts only on the $y^0$ component of $y\in \Rd$, namely on the first $d_0$ space variables.

To study the Fokker-Planck equation above, we introduce the concept of \emph{pseudo-mild solutions}. The latter is a modification of the known concept of mild solutions, previously employed by several authors in similar settings, employed also in \cite{lieber2018well}. This turns out to be an amenable notion of solution within the current setting, which can be proved to be equivalent to the one of weak solution in Definition \ref{def:sol_nl_fp}.

\paragraph{Heuristic idea:} Re-write the Fokker Planck equation \eqref{eq:nl_fp_cp_den_kin} as 
\begin{equation}
 \Kcc_{\sigma_0,B}' v_t(x)=  - \text{div}_{x^0} \Big(b_0\big(t,x, v_t(x)\big) \,  v_t(x) \Big)     , \qquad t\in (0,T),
\end{equation}
where $\Kcc_{\sigma_0,B}'$ is the Fokker Planck operator associated to the diffusion 
\begin{equation}\label{eq:mkv_degenerate_den_mult}
\dd X_t = B X_t \,  \dd t   + \left( 
\begin{matrix}
\sigma_0 \big(t,X_t\big)  \\
{\bf 0}_{(d - d_0) \times 1}
\end{matrix}
\right) \,  \dd W_t   ,   
\end{equation}
namely
 \begin{equation}
 \Kcc_{\sigma_0,B}' =  \partial_t  + \text{div}_x \big( B x\, \cdot\big) -\frac{1}{2}\text{div}_{x^0}\Big(\text{\bf div}_{x^0} \big(   \sigma_0 \sigma_0^\top(t,x) \, \cdot \big)\Big).
\end{equation}
Denoting by $p_{B,\sigma_0}=p_{B,\sigma_0}(s,x; t, y)$ the transition density of  the diffusion \eqref{eq:mkv_degenerate_den_mult} (see Lemma \ref{lem:sde_kin_dens}-(i)), the forward semigroup of $\Kcc_{\sigma_0,B}'$ is defined, for any $\mu\in\mathcal{P}(\Rd)$, as 
\begin{equation}
\Pcc'^{B,\sigma_0}_{s,t} \mu : = \int_{\Rd} \mu(dx) p_{B,\sigma_0}(s,x; t, \cdot)  , \qquad 0\leq s<t< T .
\end{equation}
Assuming that the latter is regular enough, Duhamel's principle yields
\begin{equation}\label{eq:mild_sol}
v_t = \Pcc'^{B,\sigma_0}_{0,t} \nu + \int_0^t \Pcc'^{B,\sigma_0}_{s,t} \big[ \text{div}_{x^0} \big(b_0\big(s,x, v_s(x)\big)v_s(x)\big) dx \big] ds, \qquad t\in(0,T).
\end{equation}
The formula above is the standard concept of mild solution. However, notice that, 
under the current assumptions on $b_0$, the 
divergence term above is a distribution. 
Therefore, in order to perform the analysis, one would have to extend the forward semigroup $\Pcc'^{B,\sigma_0}_{s,t}$ on a negative Besov space and prove appropriate Schauder estimates for the latter. In the case of multiplicative noise that we are considering here, this would require an extension of the results in \cite{menozzi2026heat}, where the program above has been realized for the backward semigroup. In order to avoid negative Besov spaces, the idea is to use integration by-parts and charge the derivatives onto the semigroup. Indeed, assuming that the transition density $p_{B,\sigma_0}(s,x; t, y) $ enjoys sufficient regularity in the $x$ variable, one can write $v$ in \eqref{eq:mild_sol} as
\begin{equation}\label{eq:pseudo_mild}
v_t = \Pcc'^{B,\sigma_0}_{0,t} \nu + \int_0^t  \int_{\Rd}  v_s(x) \big\langle b_0\big(s,x, v_s(x)\big) , \nabla_{x^0} p_{B,\sigma_0}(s,x; t, \cdot) \big\rangle dx \, ds, \qquad t\in (0,T).
\end{equation}
The formula above will be our notion of \emph{pseudo-mild solution} to \eqref{eq:nl_fp_cp_den_kin} (see also \cite{lieber2018well}). 

\vspace{5pt}

To prove Proposition \ref{sec:proof_prop_fp}, we implement the following program:
\begin{itemize}
\item Prove existence and uniqueness of pseudo-mild solutions, 
and study their regularity.
\item Show that 
the notion of pseudo-mild solution is equivalent to the one of weak solution.
\end{itemize}

We start with the following
\begin{remark}
By Lemma \ref{lem:sde_kin_dens}-(ii) with $\beta_0\equiv 0$, there exists $C>0$, depending only on $\| \sigma_0 \|_{L^{\infty}_{T}\Cc^{\alpha}_B(\Rd)}$, $\lambda$, $T$ and $B$, such that
\begin{equation}\label{eq:gaussian_1}
 |  \nabla^{m}_{x^0} p_{B,\sigma_0}(s,x;t,y)  |  \leq C (t-s)^{-\frac{m}{2}} \, p_{\Kcc_{2\lambda}}(t-s, x  , y),
\end{equation}
for any $0\leq s < t < T$, any $x,y\in\Rd$ and $m=0,2$. Here, $p_{\Kcc,\cdot}$ is the Gaussian density in \eqref{eq:kol_dens}.
\end{remark}

We can now give the following
\begin{definition}\label{def:psudo_mild}
Let $\nu\in\mathcal{P}(\Rd)$. A function 
$v:(0,T)\to L^{1}(\Rd)$ 
is called a \emph{pseudo-mild solution} to \eqref{eq:nl_fp_cp_den_kin} if \eqref{eq:pseudo_mild} holds true.
\end{definition}

Proposition \ref{prop:fp_kinetic_Density} is a direct consequence of the following two propositions.

\begin{proposition}\label{prop:equiv_weak}
Let $v:(0,T)\to L^{1}(\Rd)$, and let also $\sigma$, $b$ be as in \eqref{eq:sig_b_deg_den} and $\nu \in \mathcal{P}(\Rd)$. If either
\begin{itemize}
\item[(i)] Assumption \ref{assum:b_0}-(b) is satisfied,
\item[(ii)] or, Assumption \ref{assum:b_0}-(b-loc) is satisfied and $v\in L^{\infty}_{T}L^{\infty}(\Rd)$,
 \end{itemize}
then  $v$ is a \emph{pseudo-mild solution} to \eqref{eq:nl_fp_cp_den_kin} if and only if it is a weak solution to D-NL-FP$(b,\sigma; \nu)$ (see Definition \ref{def:sol_nl_fp_den}). 
\end{proposition}

\begin{proposition}\label{prop:wellposed_mild} Let $\nu\in\mathcal{P}(\Rd)$. Then: 
\begin{itemize}
\item[(i)] Under Assumption \ref{assum:b_0}-(b), 
there exists a unique pseudo-mild solution $v$ to \eqref{eq:nl_fp_cp_den_kin}. 
\item[(ii)] Under Assumption \ref{assum:b_0}-(b-loc), and if $\nu = v_0(x) dx$ with $v_0 \in L^{\infty}(\Rd)\cap L^{1}(\Rd)$, 
there exists a unique pseudo-mild solution $v$ to \eqref{eq:nl_fp_cp_den_kin} in $L^{\infty}_{T}L^{\infty}(\Rd)$. 

Furthermore, under the additional assumption that the diffusion coefficient $\sigma_0\equiv\sigma>0$ is constant, if $v_0 \in C^{\gamma}_B$ with $\gamma\in (0,1)$, then $v \in \Cc_B^{\gamma}((0,T)\times \Rd)$.
 \end{itemize}
In both cases, $v$ is non-negative and belongs to $L^{\infty}_{T}L^{1}(\Rd)$.
\end{proposition}

We conclude this section with the proof of Propositions \ref{prop:wellposed_mild} and \ref{prop:equiv_weak}.

\begin{proof}[Proof of Proposition \ref{prop:equiv_weak}]
Set
\begin{equation}
u_t:= \Pcc'^{B,\sigma_0}_{0,t} \nu ,\qquad 
w_t:=  \int_0^t  \int_{\Rd}  v_s(x) \big\langle b_0\big(s,x, v_s(x)\big) , \nabla_{x^0} p_{B,\sigma_0}(s,x; t, \cdot) \big\rangle dx \, ds, 
 \qquad t\in(0,T),
\end{equation}

Fix now $\varphi\in C^{\infty}_0(\Rd)$. By Lemma \ref{lem:sde_kin_dens}-(iii), with $\beta_0\equiv 0$, $u$ satisfies
\begin{align}
\int_{\Rd} u_t(y) \varphi(y) dy &=  \int_{\Rd} \varphi(y) \nu (\dd y)\\ 
& + \int_0^{t} \int_{\Rd}  \Big[\big\langle B y , \nabla \varphi(y)\big\rangle + \frac{1}{2} \text{tr}\big( \sigma_0\sigma_0^\top (s , y ) \nabla^2_{y^0} \varphi(y) \big)   \Big]  u_s(y) \dd y \, \dd s,
\label{eq:def_sol_NL_FP_den_bis}
\end{align}
for any $t\in (0,T)$. 
By applying once more Lemma \ref{lem:sde_kin_dens}-(iii), with $\beta_0\equiv 0$ and $\nu = \delta_x$, we also obtain
\begin{align}
\int_{\Rd} p_{B,\sigma_0}(s,x; t, y) \varphi(y) dy &= \varphi(x) \\
&\quad +  \int_s^{t} \int_{\Rd}  \Big[\big\langle B y , \nabla \varphi(y)\big\rangle + \frac{1}{2} \text{tr}\big( \sigma_0\sigma_0^\top (r , y ) \nabla^2_{y^0} \varphi(y) \big)   \Big] p_{B,\sigma_0}(s,x; r, y)  \dd y \, \dd r,\\
& =: \varphi(x)  + \tilde w_{s,t}(x),
\end{align}
for any $0<s<t<T$. Therefore, employing Fubini Theorem and swapping integration and gradient signs, every time we need it, yields
\begin{align}
\int_{\Rd} w_t(y) \varphi(y) dy &=  \int_0^t  \int_{\Rd}  v_s(x) \Big\langle b_0\big(s,x, v_s(x)\big) , \nabla_{x^0} \int_{\Rd} p_{B,\sigma_0}(s,x; t, y) \varphi(y) dy \Big\rangle dx \, ds 
\intertext{(by \eqref{eq:def_sol_NL_FP_den_ter})}
&=  \int_0^t  \int_{\Rd}  v_s(x) \big\langle b_0\big(s,x, v_s(x)\big) , \nabla_{x^0} \big( \varphi(x) + \tilde w_{s,t}(x) \big) \big\rangle dx \, ds\\
&=  \int_0^t  \int_{\Rd}  v_s(x) \big\langle b_0\big(s,x, v_s(x)\big) , \nabla_{x^0}  \varphi(x)  \big\rangle dx \, ds \\
&\quad +\int_0^t  \int_{\Rd}   \Big[\big\langle B y , \nabla \varphi(y)\big\rangle + \frac{1}{2} \text{tr}\big( \sigma_0\sigma_0^\top (r , y ) \nabla^2_{y^0} \varphi(y) \big)   \Big] w_r(y) dy dr .  \label{eq:def_sol_NL_FP_den_ter}
\end{align}
Note that the application of Fubini Theorem, as well as swapping gradient $\nabla_{x^0}$ with integral signs, are fully justified by the Gaussian estimate \eqref{eq:gaussian_1} with $m=1$ and by the boundedness assumptions on $\sigma_0$ and $b_0$. See, in particular,  \eqref{eq:b0_comp_v} in Remark \ref{rem:b0_comp_v}.

Combining \eqref{eq:def_sol_NL_FP_den_bis} with \eqref{eq:def_sol_NL_FP_den_ter}, finally yields
\begin{align}
\int_{\Rd}\big( u_t(y) + w_t(y) \big) \varphi(y) dy &=  \int_{\Rd} \varphi(y) \nu (\dd y)\\ 
&\quad + \int_0^{t} \int_{\Rd}  \Big[\big\langle B y , \nabla \varphi(y)\big\rangle + \frac{1}{2} \text{tr}\big( \sigma_0\sigma_0^\top (s , y ) \nabla^2_{y^0} \varphi(y) \big)   \Big]  \big( u_s(y) + w_s(y) \big)  \dd y \, \dd s \\
& \quad + \int_0^t  \int_{\Rd}   \big\langle b_0\big(s,y, v_s(y)\big) , \nabla_{y^0}  \varphi(y)  \big\rangle v_s(y) dy \, ds.
 \label{eq:def_sol_NL_FP_den_quat}
\end{align}
Therefore, if $v$ is a \emph{pseudo-mild solution} to \eqref{eq:nl_fp_cp_den_kin}, which is $v = u + w$, we have
\begin{align}
\hspace{-10pt}\int_{\Rd} v_t(y)  \varphi(y) dy &=  \int_{\Rd} \varphi(y) \nu (\dd y)\\ 
& + \int_0^{t} \int_{\Rd}  \Big[\big\langle b_0\big(s,y, v_s(y)\big) , \nabla_{y^0}  \varphi(y)  \big\rangle + \big\langle B y , \nabla \varphi(y)\big\rangle + \frac{1}{2} \text{tr}\big( \sigma_0\sigma_0^\top (s , y ) \nabla^2_{y^0} \varphi(y) \big)   \Big]  v_s(y)   \dd y \, \dd s,\\
\label{eq:weak_kin}
\end{align}
for any $t\in(0,T)$, which means $v$ is a weak solution to D-NL-FP$(b,\sigma; \nu)$ with $\sigma$, $b$ as in \eqref{eq:sig_b_deg_den}, according to  Definition \ref{def:sol_nl_fp_den}. 

On the other hand, if \eqref{eq:weak_kin} is true, then \eqref{eq:def_sol_NL_FP_den_quat} gives
\begin{equation}\label{eq:def_sol_NL_FP_den_pent}
\int_{\Rd} z_t(y) \varphi(y) dy =  \int_0^{t} \int_{\Rd}  \Big[\big\langle B y , \nabla \varphi(y)\big\rangle + \frac{1}{2} \text{tr}\big( \sigma_0\sigma_0^\top (s , y ) \nabla^2_{y^0} \varphi(y) \big)   \Big]  z_t(y)  \dd y \, \dd s,
\end{equation}
for any $t\in(0,T)$, where we set $z : = u + w - v$. As $\varphi \in C_0^{\infty}(\Rd)$ is arbitrary, 
\eqref{eq:def_sol_NL_FP_den_pent} forces $z\equiv 0$ (by Lemma \ref{lem:sde_kin_dens}-(iii) with $\beta_0 \equiv 0$), and thus $v = u + w$.
\end{proof}

\begin{proof}[Proof of Proposition \ref{prop:wellposed_mild}] 
We first prove Part (ii). For simplicity, we assume $T<<1$ suitably small. The proof for a general $T$ can be obtained by introducing an exponentially-weighted norm, similarly to \cite[Proposition 3.5 and Lemma 3.6]{issoglio2023mckean}. For ease of reading, we organize the proof in several steps.

\emph{Step 1}:
For any $u \in L^{\infty}_{T}L^{\infty}(\Rd)$, set
\begin{equation}
J(u)_t =  \Pcc'^{B,\sigma_0}_{0,t} \nu + \int_0^t  \int_{\Rd}  u_s(x) \big\langle b_0\big(s,x, u_s(x)\big) , \nabla_{x^0} p_{B,\sigma_0}(s,x; t, \cdot) \big\rangle dx \, ds, \qquad t\in (0,T).
\end{equation}

 We first prove that $J$ is well defined as a map from $L^{\infty}_{T}L^{\infty}(\Rd)$ onto itself. As $v_0 \in L^{\infty}(\Rd)$ and employing \eqref{eq:gaussian_1} with $m=0$, we have
\begin{equation}
\big\| \Pcc'^{B,\sigma_0}_{0,t} \nu \big\|_{L^{\infty}(\Rd)} \lesssim \| v_0 \|_{L^{\infty}(\Rd)} \int_{\Rd} p_{\Kcc_{2\lambda}}(s, x  , \cdot) dx \lesssim \kappa  \| v_0 \|_{L^{\infty}(\Rd)}, \qquad t\in(0,T),
\end{equation}
with
\begin{equation}
\kappa_0(\lambda) = \kappa_0  : = \sup_{(s,y)\in(0,1)\times\Rd} \int_{\Rd} p_{\Kcc_{2\lambda}}(s, x  , y) dx <\infty .
\end{equation}
Furthermore, for $u\in L^{\infty}_T L^{\infty}(\Rd)$, by \eqref{eq:gaussian_1} with $m=1$ we have
\begin{equation}\label{eq:pseudo_mild_bound_bis}
\big |   u_s(x) \big\langle b_0\big(s,x, u_s(x)\big) , \nabla_{x^0} p_{B,\sigma_0}(s,x; t, \cdot) \big\rangle \big| 
 \lesssim 
(t-s)^{-\frac{1}{2}}\big|   p_{\Kcc_{2\lambda}}(t-s, x  , \cdot)\big|  
 \| u_s \|_{L^{\infty}(\Rd)} 
 \| b_0\big(s, \cdot, u_s(\cdot)\big) \|_{L^{\infty}(\Rd)} 
\end{equation}
for any $(s,x)\in (0,t)\times\Rd$, 
and thus, by integrating,
\begin{align}
&\Big\|  \int_0^t  \int_{\Rd}  u_s(x) \big\langle b_0\big(s,x, u_s(x)\big) , \nabla_{x^0} p_{B,\sigma_0}(s,x; t, \cdot) \big\rangle dx \, ds \Big\|_{L^{\infty}(\Rd)} \\
&\qquad \lesssim \sqrt{T}\,
  \| u \|_{L^{\infty}_T L^{\infty}(\Rd)} 
  \sup_{|y|\leq \| v \|_{L^{\infty}_T L^{\infty}(\Rd)}}
 \| b_0(\cdot, \cdot, y) \|_{L^{\infty}((0,T)\times\Rd)},
\end{align}
for any $t\in(0,T)$. Therefore, recalling Remark \ref{rem:b0_comp_v}, $J(u)\in L^{\infty}_T L^{\infty}(\Rd)$. 

In particular, taking $M\gtrsim  \kappa  \| v_0 \|_{L^{\infty}(\Rd)}$ and assuming $\sqrt{T}\lesssim \kappa^{-1}_1$, with
\begin{equation}
\kappa_1 := \sup_{|y|\leq M}
 \| b_0(\cdot, \cdot, y) \|_{L^{\infty}((0,T)\times\Rd)}<\infty,
\end{equation} 
one has $J(B_{M})\subset B_M$, 
where $B_M$ denotes the ball of radius $M$ in $L^{\infty}_T L^{\infty}(\Rd)$, centered at zero.
\vspace{2pt}

\emph{Step 2}:
We now prove that $J$ is a contraction on $B_{M}$. For any $u,\tilde u \in B_{M}$ we have
\begin{align}
\big|\big(J(u) - J(\tilde u)\big)_t\big| & \leq \int_0^t  \int_{\Rd}\big|  u_s(x)  b_0\big(s,x, u_s(x)\big) - \tilde u_s(x)  b_0\big(s,x, \tilde u_s(x)\big)\big| \, \big| \nabla_{x^0} p_{B,\sigma_0}(s,x; t, \cdot) \big| dx \, ds \\
&\leq  \int_0^t  \int_{\Rd}|  u_s(x)|\, \big|  b_0\big(s,x, u_s(x)\big) -  b_0\big(s,x, \tilde u_s(x)\big)\big| \, \big| \nabla_{x^0} p_{B,\sigma_0}(s,x; t, \cdot) \big| dx \, ds \\
&  \quad +  \int_0^t  \int_{\Rd} |  u_s(x)  - \tilde u_s(x)| \, \big|  b_0\big(s,x, \tilde u_s(x)\big)\big| \, \big| \nabla_{x^0} p_{B,\sigma_0}(s,x; t, \cdot) \big| dx \, ds 
\intertext{(as $u\in B_M$ and by employing Assumption \ref{assum:b_0}-(a))}
& \lesssim (M + \kappa_1) \int_0^t  \int_{\Rd} |  u_s(x)  - \tilde u_s(x)| \,  \big| \nabla_{x^0} p_{B,\sigma_0}(s,x; t, \cdot) \big| dx \, ds, \qquad t\in(0,T). \label{eq:est_contr}
\end{align}
Therefore, Gaussian estimate \eqref{eq:gaussian_1} with $m=1$ and integrating in both $dx$ and $ds$ yields
\begin{equation}
\| J(u) - J(\tilde u) \|_{L^{\infty}_T L^{\infty}(\Rd)} \lesssim  (M + \kappa_1) \kappa_0 \sqrt{T} \| u - \tilde u \|_{L^{\infty}_T L^{\infty}(\Rd)}.
\end{equation}
Thus, the contraction property is true if $\sqrt{T} \lesssim \big( (M + \kappa_1) \kappa_0 \big)^{-1}$. By the Banach fixed-point theorem, $J$ admits a unique fixed point $v$ on $B_ M$. Equivalently, there exists a unique $v\in B_ M$ that satisfies \eqref{eq:pseudo_mild}. 
\vspace{2pt}

\emph{Step 3}: We prove that $v$ in Step 2 is the unique pseudo-mild solution $v$ to \eqref{eq:nl_fp_cp_den_kin} in $L^{\infty}_{T}L^{\infty}(\Rd)$, and $v\in L^{\infty}_T L^{1}(\Rd)$. 
Assume that $\tilde v$ is another function in $L^{\infty}_T L^{\infty}(\Rd)$ (not in $B_M$) that satisfies \eqref{eq:pseudo_mild}. By \eqref{eq:est_contr}, and then by applying the Gaussian estimate \eqref{eq:gaussian_1} with $m=1$ and integrating in $dx$, we obtain
\begin{equation}
\| v - \tilde v \|_{L^{\infty}_t L^{\infty}(\Rd)} \lesssim  (M + \kappa_1) \kappa_0 \int_0^t  (t-s)^{-1/2} \| v - \tilde v \|_{L^{\infty}_s L^{\infty}(\Rd)} ds, \qquad t\in(0,T).
\end{equation}
Gr\"onwall's inequality thus yields $\| v - \tilde v \|_{L^{\infty}_T L^{\infty}(\Rd)} = 0$. Therefore, there exists a unique $v \in L^{\infty}_T L^{\infty}(\Rd)$ that satisfies \eqref{eq:pseudo_mild}.

We now prove that $v \in L^{\infty}_T L^{1}(\Rd)$. We do this by Picard iteration. Let us define
\begin{equation}
v^{(0)}\equiv v_0 , \qquad v^{(n)}: = J\big(v^{(n-1)}\big), \quad n\in N. 
\end{equation}
We have
\begin{align}
\| v^{(n)}_t \|_{L^1(\Rd)} &\leq \int_{\Rd} [\Pcc'^{B,\sigma_0}_{0,t} \nu ](y) dy + \int_{\Rd} \Big| \int_0^t  \int_{\Rd}  v^{(n-1)}_s(x) \big\langle b_0\big(s,x, v^{(n-1)}_s(x)\big) , \nabla_{x^0} p_{B,\sigma_0}(s,x; t, y) \big\rangle dx \, ds \Big| dy \\
& \leq 1 + \int_{\Rd}  \int_0^t  \int_{\Rd} \big|v^{(n-1)}_s(x)\big|\, \big|  b_0\big(s,x, u_s(x)\big| \, \big |\nabla_{x^0} p_{B,\sigma_0}(s,x; t, y) \big| dx \, ds\,  dy
\intertext{(by the Gaussian estimate \eqref{eq:gaussian_1} with $m=1$, and recalling that $v\in B_M$)}
& \leq 1 + C \kappa_1  \int_0^t   (t-s)^{-1/2} \int_{\Rd}   \int_{\Rd} \big|v^{(n-1)}_s(x)\big| \, p_{\Kcc_{2\lambda}}(t-s, x  , y) dy\, dx \, ds \\
& =  1 + C \kappa_1  \int_0^t   (t-s)^{-1/2} \| v^{(n-1)}_s \|_{L^1(\Rd)}  \, ds, \qquad t\in(0,T).
\end{align}
Iterating in $n$, this yields
\begin{equation}
\| v^{(n)} \|_{L^{\infty}_T L^1(\Rd)} \leq   \sum_{k =0}^{\infty}  \big( \sqrt{\pi} C \kappa_1 \big)^{k}
\frac{T^{\frac{k}{2}}}{\Gamma_{\text{Euler}}\left(\frac{k+2}{2}\right)} < \infty.
\end{equation}
Therefore, as $v^{(n)}$ converges to $v$ almost everywhere, Fatou's Lemma immediately gives
\begin{equation}
\|v\|_{L^{\infty}_T L^1(\mathbb{R}^d)}
\leq
\sup_{n\in\mathbb{N}} \|v^{(n)}\|_{L_T^\infty L^1(\mathbb{R}^d)}
< \infty.
\end{equation}

\vspace{2pt}

\emph{Step 4}: we set $\gamma\in (0,1)$ and show that $v \in \Cc_B^{\gamma}((0,T)\times \Rd)$ under the additional assumptions:
\begin{equation}\label{eq:assump_add_noise}
v_0 \in C^{\gamma}_B, \qquad \sigma_0 \equiv \sigma>0 \ \text{(additive noise)}.
\end{equation}
Under \eqref{eq:assump_add_noise}, the diffusion \eqref{eq:mkv_degenerate_den_mult} reduces to the linear one \eqref{eq:linear_SDE}, and thus the transition density 
$p_{B,\sigma_0}(s,x; t, y)$ and the forward semigroup $\Pcc'^{B,\sigma_0}_{s,t}$ reduce to those in  Section \ref{sec:kinet_semigroups}, precisely
\begin{equation}
p_{B,\sigma_0}(s,x; t, y) = p_{\Kcc_{\sigma}}(t-s,x, y) , \qquad \Pcc'^{B,\sigma_0}_{s,t} = \Pcc'^{\Kcc,\sigma}_{s,t},
\end{equation}
with $p_{\Kcc_{\sigma}}$, $\Pcc'^{\Kcc,\sigma}$ as in \eqref{eq:kol_dens}, \eqref{eq:def_semigroup_for}, respectively. Now, Theorem \ref{th:Schauder_est_intr} (kinetic Schauder estimates) readily yields
\begin{equation}\label{eq:bund_kin_se}
\|\Pcc'_{\cdot}\, v_0\|_{\Cc_B^{\gamma}((0,T)\times\Rd)}<\infty.
\end{equation}
We now study the regularity of the time-convolution in \eqref{eq:pseudo_mild}. 
Setting 
\begin{equation}
w_t : = \int_0^t  \int_{\Rd}  v_s(x) \big\langle b_0\big(s,x, v_s(x)\big) , \nabla_{x^0} p_{\Kcc_{\sigma}}(t-s,x, \cdot)  \big\rangle dx \, ds,  \qquad t\in(0,T),
\end{equation}
we have
\begin{align}
|w_t(y) - w_t(y')|&\leq \bigg|  \int_0^t  \int_{\Rd}  v_s(x) \big\langle b_0\big(s,x, v_s(x)\big) , \nabla_{x^0} p_{\Kcc_{\sigma}}(t-s,x, y) - \nabla_{x^0} p_{\Kcc_{\sigma}}(t-s,x, y')  \big\rangle dx \, ds \bigg| \\
& \leq \int_0^t  \int_{\Rd} \big| v_s(x) b_0\big(s,x, v_s(x)\big| \, \big| \nabla_{x^0} p_{\Kcc_{\sigma}}(t-s,x, y) - \nabla_{x^0} p_{\Kcc_{\sigma}}(t-s,x, y')  \big| dx \, ds
\intertext{(by the Gaussian estimate \eqref{eq:bound_gauss_for_grad} with $m=1,n=0$, and recalling that $v\in B_M$)}
& \leq M \kappa_1 | y - y' |^{\gamma}_B  \int_0^t \frac{1}{(t-s)^{\frac{1+\gamma}{2}}}  \int_{\Rd} \big( p_{\Kcc_{2\sigma}}(t-s,x, y) + p_{\Kcc_{2\sigma}}(t-s,x, y')  \big) dx \, ds \\
& \lesssim M \kappa_1 \kappa_0(2 \sigma)  | y - y' |^{\gamma}_B, 
\end{align}
for any $(t,y,y')\in(0,T)\times \Rd \times \Rd$, which gives 
\begin{equation}
\sup_{t\in(0,T)} \|w_t\|_{B,\gamma} < \infty.
\end{equation}
Fix now $0<t<t'<T$. We have 
\begin{equation}\label{eq:incr_Y_w}
\big| w_{t'}\big( e^{(t'-t)B} \cdot \big) - w_t  \big|  \leq   \int_{t}^{t'}  J_{1}(t',s)  \, ds  + \int_0^t J_{2}(t,t',s)  \, ds  \\
\end{equation}
with 
\begin{align}
J_{1}(t',s)& = \bigg|  \int_{\Rd}  v_s(x) \big\langle b_0\big(s,x, v_s(x)\big) , \nabla_{x^0} p_{\Kcc_{\sigma}}(t'-s,x, \cdot)  \big\rangle dx \bigg| ,\\
J_{2}(t,t',s)& = \bigg|  \int_{\Rd}  v_s(x) \big\langle b_0\big(s,x, v_s(x)\big) , \nabla_{x^0} p_{\Kcc_{\sigma}}(t'-s,x,  e^{(t'-t)B} \cdot) - \nabla_{x^0} p_{\Kcc_{\sigma}}(t-s,x,   \cdot)  \big\rangle dx \bigg|. 
\end{align}
By the Gaussian estimate \eqref{eq:bound_Grad_kol} with $n=0,m=1$, and recalling that $v\in B_M$, we obtain 
\begin{equation}
J_{1}(t',s) \leq \frac{M \kappa_1 }{(t'-s)^{\frac{1}{2}}}   \int_{\Rd} p_{\Kcc_{2\sigma}}(t'-s,x, y)   dx  \leq \frac{ M \kappa_1 \kappa(0)(2 \sigma) }{(t'-s)^{\frac{1}{2}}}  , 
\end{equation}
and thus 
\begin{equation}\label{eq:est_J1_int}
 \int_{t}^{t'}  J_{1}(t',s)  \, ds \lesssim M \kappa_1 \kappa(0)(2 \sigma) \sqrt{t'-t}.
\end{equation}
To bound $ J_{2}$, we consider two different regimes, namely:
\\
\underline{Case $t' - t \leq t -s$:} applying Gaussian estimate \eqref{eq:bound_gauss_for_Y} with $m=1,n=0$ and $r=t-s$, $r' = t' -s$, and noticing that $(t'-s)/(t-s) \leq 2$, yields
\begin{equation}
 J_{2}(t,t',s) \lesssim  M \kappa_1 \frac{ (t' -t)^{\frac{\gamma}{2}} }{(t-s)^{\frac{1+\gamma}{2}}}  \int_{\Rd} p_{\Kcc_{2\sigma}}(t'-s,x, \cdot)   dx \leq M \kappa_1 \kappa_0(2\sigma)  \frac{ (t' -t)^{\frac{\gamma}{2}} }{(t-s)^{\frac{1+\gamma}{2}}} .
\end{equation}
\underline{Case $t' - t > t -s$:} Gaussian estimate \eqref{eq:bound_Grad_kol} with $n=0,m=1$ together with the triangle inequality yields
\begin{equation}
 J_{2}(t,t',s) \lesssim   \frac{M \kappa_1}{(t-s)^{\frac{1}{2}}}   \int_{\Rd}  \big(  p_{\Kcc_{2\sigma}}(t'-s,x,  e^{(t'-t)B} \cdot) + p_{\Kcc_{2\sigma}}(t-s,x,   \cdot)  \big)   dx \leq  \frac{M \kappa_1  \kappa_0(2\sigma)  }{(t-s)^{\frac{1}{2}}}.
\end{equation}
Combining the two regimes, we obtain
\begin{align}
 \int_{0}^{t}  J_{2}(t,t',s)  \, ds &\lesssim M \kappa_1 \kappa(0)(2 \sigma)  \bigg( \int_0^{0\vee(t -(t'-t) )}  \frac{ (t' -t)^{\frac{\gamma}{2}} }{(t-s)^{\frac{1+\gamma}{2}}}  ds + \int_{0\vee(t -(t'-t))}^t  \frac{1  }{(t-s)^{\frac{1}{2}}}ds \bigg)\\
& \lesssim  M \kappa_1 \kappa(0)(2 \sigma)  \big( (t' -t)^{\frac{\gamma}{2}} + (t' -t)^{\frac{1}{2}} \big). \label{eq:est_J2_int}
\end{align}

Plugging \eqref{eq:est_J1_int}-\eqref{eq:est_J2_int} thus proved into \eqref{eq:incr_Y_w} yields
\begin{equation}
\sup_{y\in\Rd} \big\|w_{\cdot}(e^{\cdot B} y)\big\|_{(0,T),\gamma/2} < \infty.
\end{equation}
Recalling the definition of intrinsic norm in \eqref{eq:intr_holder_gen}, the latter together with \eqref{eq:bund_kin_se} prove
\begin{equation}
\|v\|_{\Cc_B^{\gamma}((0,T)\times\Rd)}<\infty,
\end{equation}
which completes the proof of Part (i).
\vspace{2pt} 

To prove Part (ii), owing to Assumption \ref{assum:b_0}-(b), it is possible to run the contraction argument directly on $L_T^{\infty} L^1(\Rd)$. We omit the details because the computations are nearly identical to those in Part (i), Steps 1 and 2. 

\vspace{2pt}

To conclude, we need to show that, in both cases, $v$ is non-negative. By Proposition \ref{prop:equiv_weak}, $v$ is a weak solution to D-NL-FP$(b,\sigma; \nu)$ where $\sigma$, $b$ are as in \eqref{eq:sig_b_deg_den}. In particular, it is a weak density solution to the linear FP$({\bf b}_{v},\sigma;0, \nu)$, with ${\bf b}_{v}$ as in \eqref{eq:b_0_kin}. However, by Remark \ref{rem:b0_comp_v}, we have ${\bf b}_{v}\in L^{\infty}((0,T)\times \Rd)$ and thus, by Lemma \ref{lem:sde_kin_dens}-(iii) with $\beta = {\bf b}_{v}$, $v$ is the flow of time-marginal densities of the solution to SDE$({\bf b}_{v},\sigma;0, \nu)$. In particular, $v$ is non-negative.

\end{proof}






\bibliographystyle{siam}
\bibliography{merged_bib}

\appendix

\chapter{Elements of $d$-dimensional diffusion theory}  

We recall some basic elements of the theory of diffusion processes on $\Rd$. In particular, we recall the basic notions that compose the three main building blocks that are utilized in the body of the notes: It\^o stochastic differential equations (SDEs), martingale problems (MPs) and Fokker-Planck equations (FPs).

Hereafter, we fix $T\in [0,\infty]$ and the functions 
\begin{equation}
b: [0,T)\times \Rd \to \R^d, \qquad \sigma: [0,T)\times \Rd \to \mathcal{M}^{d\times q},
\end{equation}
and assume that they are Borel measurable. 

We also employ the following 
\begin{notation}\label{not:measure_dt_dx}
Let $s\geq 0$. For a given measurable flow ${\boldsymbol\mu}:[s,T) \to \mathcal{M}_{+}(\Rd)$, we denote by $\overline{\boldsymbol\mu}=\overline{\boldsymbol\mu} (\dd t , \dd x) $ the measure on $[s,T)\times \Rd$ defined as 
\begin{equation}
\overline{\boldsymbol\mu} ([s,t]\times H) : = \int_{s}^t \int_H {\boldsymbol\mu}_r(dx) dr, \qquad t\in[s,T), \ H\in \Bc(\Rd).
\end{equation}
\end{notation}

\section{Ito SDEs}

Let us consider the following It\^o stochastic differential equation (SDE), formally written as:

\begin{equation}
\hspace{-90pt}{\blue\text{[SDE$({\bf b},{\boldsymbol\sigma})$]}}\qquad\qquad d X_t = b(t,X_t) dt + \sigma(t,X_t) dW_t ,
\end{equation}
It\^o SDEs can be formulated in the weak sense, where the ($q$-dimensional) Brownian motion is part of the solution and an initial distribution is assigned, or in the pathwise sense, where the Brownian motion is fixed and an initial random variable is assigned.

\begin{definition}[general solution to an SDE starting at time $s$]\label{def:sol_gen_SDE}
Let $s\in [0,T)$. A (general) solution to SDE$(b,\sigma;s)$ is a 
triple $( (\Fc_t)_{t},W, X)$, where: 
\begin{itemize}
\item[(i)] $(\Fc_t)_{t\in [s,T)}$ is a filtration on a probability space $(\Omega , \Fc, \Pb)$, satisfying the usual assumptions of completeness and continuity from the right;
\item[(ii)] $W=(W_t)_{t\in [s,T)}$ is a $q$-dimensional (uncorrelated) Brownian motion (starting at $s$) with respect to $(\Fc_t)_{t\in [s,T)}$;
\item[(iii)] $X=(X_t)_{t\in [s,T)}$ is an $\R^d$-valued process, adapted to  $(\Fc_t)_{t\in [s,T)}$, such that: 
\begin{itemize}
\item the processes $(b(t,X_t))_{t\in [s,T)}$ and  $(\sigma(t,X_t))_{t\in [s,T)}$ are progressively measurable and such that 
\begin{equation}
\int_s^t |b(r,X_r)| dr + \int_s^t \big|\sigma\sigma^\top(r,X_r)\big| dr  <\infty, \qquad \Pb\text{-almost surely.}
\end{equation}
\item 
the following representation holds $\Pb$-almost surely:
\begin{equation}\label{eq:gen_sol}
X_t = X_s + \int_s^t b(r,X_r) dr + \int_s^t \sigma(r,X_r) dW_r ,\qquad t\in [s,T).
\end{equation}
\end{itemize} 
\end{itemize}
\end{definition}

For ease of reading, we will sometimes refer to a solution simply by the $X$ compnent of the triple. However, the reader should bear in mind that the filtration, the Brownian motion, and hence the underlying probability space, are part of the solution.

\begin{definition}[Weak solution to an SDE]\label{def:weak_sol_SDE}
Let $s\in [0,T)$ and \textcolor{black}{$\nu\in \mathcal{P}(\R^d)$}. A (weak) solution to SDE$(b,\sigma;s,\textcolor{black}{\nu})$ 
is a 
(general) solution to SDE$(b,\sigma;s)$ such that
\begin{equation}
[X_s] =  \nu.
\end{equation}
\end{definition}

\begin{definition}[Pathwise solution to an SDE]\label{def:path_sol_SDE}
Let $(W,(\Fc_t)_t)$ be a Brownian motion and \textcolor{black}{$\xi\in m\Fc_s$}. A (pathwise) solution to SDE$(b,\sigma;\textcolor{black}{s, \xi, W, (\Fc_t)_t})$ 
is a 
process $X$ such that the triple $(\textcolor{black}{W}, (\Fc_t)_t, X)$ 
is a (general) solution to SDE$(b,\sigma;s)$ with 
\begin{equation}
\textcolor{black}{X_s = \xi}. 
\end{equation}
\end{definition}

Consequently, we have a weak and a pathwise notion of well-posedness (existence and uniqueness) for SDEs.

\begin{definition}
\label{def:well_posed_lin}
Let $s\in [0,T)$. We say that SDE$(b,\sigma;s)$ enjoys
\begin{itemize}
\item[-] \emph{\textcolor{black}{weak} existence}, if there exists a solution to SDE$(b,\sigma;s,\textcolor{black}{\nu})$ for any $\nu\in \Pc(\R^d)$;
\item[-] \emph{\textcolor{black}{weak} uniqueness}, if, for any $\nu\in \Pc(\R^d)$, two solutions to SDE$(b,\sigma;s,\textcolor{black}{\nu})$ have the same law;
\item[-] \emph{\textcolor{black}{pathwise} existence}, if there exists a solution to SDE$(b,\sigma;s,\textcolor{black}{\xi, W, (\Fc_t)_t})$ for any $(W, (\Fc_t)_t )$ Brownian motion  and any \textcolor{black}{$\xi\in m\Fc_s$};
\item[-] \emph{\textcolor{black}{pathwise} uniqueness}, if, for any pair of general solutions $(( \Fc_t)_{t},W, X)$, $( (\Fc_t)_{t},W, X')$ to SDE$(b,\sigma;s)$ with $X_s = X'_s$ almost surely, $X = X'$ almost surely.
\end{itemize} 
%
We say that SDE$(b,\sigma;s)$ is weakly (or pathwise) well-posed if weak (or pathwise) existence and uniqueness hold.

Finally, we say that SDE$(b,\sigma)$ is weakly (or pathwise) well-posed if SDE$(b,\sigma;s)$ is weakly (or pathwise) well-posed for any $s\in [0,T)$.
\end{definition}

\begin{remark}
Note that pathwise well-posedness implies weak well-posedness. However, this is not a trivial claim. In order to show that almost sure uniqueness of pathwise solutions implies uniqueness in law of weak solutions, one needs to compare two solutions $X,X'$ with respect to two different Brownian motions. To do this, one has to construct two new solutions, $\tilde X$ and $\tilde X'$, with respect to a common Brownian motion and having the same laws as $X$ and $X'$, respectively.
\end{remark}

The following important theorem is due to Watanabe-Yamada (see \cite[Th. 1.1, Ch. 4]{ikeda2014stochastic})
\begin{theorem}[Watanabe-Yamada]\label{th:watanabe_yamada}
Let $s\in [0,T)$ and assume that weak existence and pathwise uniqueness hold for SDE$(b,\sigma;s,\textcolor{black}{\nu})$.
Then SDE$(b,\sigma;s)$ is pathwise well-posed.
\end{theorem}

\begin{remark}[!]
The Watanabe-Yamada Theorem is actually stronger than the version stated in Theorem \ref{th:watanabe_yamada}. Under the same assumptions, it states the existence of \emph{strong solutions}, namely solutions adapted to the standard Brownian filtration. For simplicity, in these lecture notes we prefer not to discuss the notion of strong solutions. 
\end{remark}

\section{Martingale problem}

We start with the classical definition by Strook-Varadhan (\cite[Chapter 6]{stroock_multidimensional_1979}). Hereafter, we will always denote by $(x_t)_{t\in [s,T)}$ the canonical process on $(C_{[s,T)},\Bc)$, for a given $s\in [0,T)$.
\begin{definition}\label{def:MP_lin}
Let $s\in [0,T)$ and $\nu\in\Pc(\Rd)$. A probability measure $\Pb=\Pb_{s,\nu}$, on $(C_{[s,T)},\Bc)$, is a solution to MP$(b,\sigma;s,\nu)$ if: 
\begin{itemize}
\item[(i)] $x_{s} \sim  \nu$ under $\Pb$;
\item[(ii)] for any $\varphi\in C_0^{\infty}(\Rd)$, the process defined by 
\begin{equation}
M_t : =\varphi(x_{t}) - \varphi(x_{s}) - \int_{s}^t \Big(   \langle  b(r, x_r) , \nabla \varphi(x_{r})  \rangle + \frac{1}{2}\text{tr}\big(\sigma \sigma^\top (r , x_r)  \nabla^2 \varphi(x_r)  \big) \Big) dr , \qquad t\in [s, T),
\end{equation}
is a martingale w.r.t. the natural filtration of $(x_t)_{t\in [s,T)}$.
\end{itemize} 
\end{definition}

The following result can be found in \cite[Th. 8.1.1 and Cor. 8.1.2]{stroock_multidimensional_1979} (see also \cite[Th. 18.1.3]{pascucci2024probability}).

\begin{theorem}\label{th:equiv_MP_weak}
Let $s\in [0,T)$, $\nu\in\Pc(\Rd)$, and let $\Pb$ be a probability measure on $(C_{[s,T)}, \Bc )$, such that
\begin{equation}\label{eq:cond_coeff_MP}
b, \sigma \sigma^\top \in L^{1}_{\text{loc}}([s,T) \times \Rd , \overline{\boldsymbol\mu}), 
\end{equation}
where $\overline{\boldsymbol\mu}$ is as in Notation \ref{not:measure_dt_dx}, with ${\boldsymbol\mu}_t:=\Pb(x_t \in dx)$, $t\in [s,T)$.
Then $\Pb$ is a solution to MP$({\bf b},{\boldsymbol\sigma};s,\nu)$ if and only if there exists a (weak) solution $X$ to SDE$({\bf b},{\boldsymbol\sigma};s,\nu)$ such that $[X] =\Pb$.
\end{theorem}

%
%
%

\section{Linear Fokker-Planck equations}

Consider the Fokker-Planck Cauchy problem
\begin{equation}\label{eq:lin_fp_cp}
\hspace{-30pt}{\blue\text{[FP$(b,\sigma;s,\nu)$]}}\qquad\qquad
\begin{cases}
 \partial_t {\boldsymbol\mu}_t = \text{div}_y \Big(-b(t,y) \,  {\boldsymbol\mu}_t  +\frac{1}{2}\text{\bf div}_y \big(   \sigma\sigma^\top(t,y) \, {\boldsymbol\mu}_t \big)   \Big), \quad t\in [s,T), \\
 {\boldsymbol\mu}_s = \nu \in \Pc(\Rd),
 \end{cases}
\end{equation}
where $\text{div}_y $ denotes the usual divergence operator acting on a function $f:\Rd \to \Rd$, i.e.
\begin{equation}\label{eq:div_stand}
\text{div}_y f = \partial_{y_1} f_1 + \cdots + \partial_{y_d} f_d ,
\end{equation}
and $\text{\bf div}_y$ denotes the divergence operator acting on a matrix-valued function $f:\Rd \to \mathcal{M}^{d\times d}$ as
\begin{equation}\label{eq:div_matrix}
\text{\bf div}_y f =\left( \begin{matrix}
 \partial_{y_1} f_{1,1} +  \cdots + \partial_{y_d} f_{1,d} \\
\vdots \\
 \partial_{y_1} f_{d,1} + \cdots + \partial_{y_d} f_{d,d}
\end{matrix}\right).
\end{equation}

\begin{definition}\label{def:sol_lin_fp}
Let $s\in [0,T)$ and $\nu\in \mathcal{M}(\Rd)$. A measurable flow ${\boldsymbol\mu}:[s,T)\to \mathcal{M}(\Rd)$ is a (weak) solution to FP$(b,\sigma;s,\nu)$ if
\begin{equation}
\int_{\Rd} \varphi(y) {\boldsymbol\mu}_t (\dd y)    = \int_{\Rd} \varphi(y) \nu (\dd y) + \int_{s}^{t} \int_{\Rd}  \Big[\big\langle b(r , y) , \nabla \varphi(y)\big\rangle + \frac{1}{2} \text{tr}\big( \sigma\sigma^\top(r , y ) \nabla^2 \varphi(y) \big)   \Big]  {\boldsymbol\mu}_r (\dd y) \dd r, \qquad t\in [s,T),
\label{eq:def_sol_Lin_FP}
\end{equation}
for any $\varphi\in C^{\infty}_0(\Rd)$. 

Furthermore, if ${\boldsymbol\mu}_t$ is absolutely continuous with respect to the Lebesgue measure, namely
\begin{equation}
{\boldsymbol\mu}_t (\dd x) = v_t(x) \dd x, \qquad t\in [s,T), 
\end{equation}
then the density flow $v:[s,T)\to L^{1}(\Rd)$ will be called a density solution to FP$(b,\sigma;\nu)$.
\end{definition}

\begin{remark}\label{rem:sol_weak_fp}
In Definition \ref{def:sol_lin_fp}, it is implicit that all the integrals in \eqref{eq:def_sol_Lin_FP} must be well defined and finite. Also, \eqref{eq:def_sol_Lin_FP} implies that ${\boldsymbol\mu}$ is continuous. 
\end{remark}

\begin{proposition}\label{th:lin_FP}
Let $s\in [0,T)$, $\nu\in\Pc(\Rd)$ and let ${\boldsymbol\mu}:[s,T)\to \mathcal{M}_{+}(\Rd)$ (continuous, with ${\boldsymbol\mu}_{s} = \nu$) be such that
\begin{equation}\label{eq:cond_L1}
b,\sigma \in L^{1}_{\text{loc}}([s,T) \times \Rd , \overline{\boldsymbol\mu}),
\end{equation}
with $\overline{\boldsymbol\mu}$ as in Notation \ref{not:measure_dt_dx}.
Then:
\begin{itemize}
\item[(i)] If $X$ is a solution to SDE$(b,\sigma;s,\nu)$, with $[X_t]={\boldsymbol\mu}_t $ for $t\in [s,T)$, then ${\boldsymbol\mu}$ is a weak solution to FP$(b,\sigma;s,\nu)$.
\item[(ii)] [Superposition] Under the additional assumption that 
\begin{equation}\label{eq:cond_agg_super}
\Big[ t \mapsto \int_{\Rd} \Big( |b(t,x)| + \| \sigma\sigma^\top(t,x) \| \Big) {\boldsymbol\mu}_t (\dd x)\Big], \  \big[ t \mapsto {\boldsymbol\mu}_{t} (\Rd) \big]  \in L^{\infty}_{\text{loc}}([s,T) ,
\end{equation}
if ${\boldsymbol\mu}$ is a (weak) solution to FP$(b,\sigma;s,\nu)$, 
 then there exists a solution $X$ to SDE$(b,\sigma;s,\nu)$ such that $ [X_t] = {\boldsymbol\mu}_t $.
\end{itemize}
\end{proposition}
\begin{remark} Note that the first condition in \eqref{eq:cond_agg_super} is stronger than \eqref{eq:cond_L1}. In particular, assumption $b,\sigma \in L^{\infty}_{\text{loc}}([s,T)\times\Rd)$ is enough for condition \eqref{eq:cond_L1} to be satisfied but does not ensure \eqref{eq:cond_agg_super}.
\end{remark}
The proof of Proposition \ref{th:lin_FP}-(ii) is based on the following result, proved in \cite[Th. 2.5]{trevisan2016well} (see also  \cite{figalli2008existence}).
\begin{lemma}[Superposition]\label{lem:super_lin}Let $s\in [0,T)$, $\nu\in\Pc(\Rd)$ and ${\boldsymbol\mu}$ be a (weak) solution to FP$(b,\sigma;s,\nu)$ satisfying \eqref{eq:cond_agg_super}. Then, there exists a solution $\Pb$ to MP$(b,\sigma;s,\nu)$ such that $ \Pb(x_t \in x) = {\boldsymbol\mu}_t$, $t\in [s,T)$.
\end{lemma}
\begin{proof}[Proof of Proposition \ref{th:lin_FP}]
\emph{Part (i)}. By definition, $X$ is an It\^o process of the form \eqref{eq:gen_sol}. Thus, \eqref{eq:def_sol_Lin_FP} stems from It\^o formula together with \eqref{eq:cond_L1}.

\emph{Part (ii)}. It is a direct consequence of Lemma \ref{lem:super_lin} together with Theorem \ref{th:equiv_MP_weak}.
\end{proof}

%


\end{document}